\documentclass[11pt]{amsart}
\usepackage{amsmath, amssymb}
\usepackage{amsfonts}
\usepackage{mathrsfs}
\usepackage[arrow,matrix,curve,cmtip,ps]{xy}

\usepackage{amsthm}
\usepackage{mathtools} %math arrows
\usepackage{array}
\usepackage{mathdots}

\usepackage[top= 2.25cm, bottom=2cm, left = 2cm, right= 2cm]{geometry}
\usepackage{hyperref}

\allowdisplaybreaks

\newtheorem{hypothesis}{Hypothesis}

\numberwithin{equation}{section}

\newtheorem{theorem}{Theorem}[section]
\newtheorem{lemma}[theorem]{Lemma}
\newtheorem{proposition}[theorem]{Proposition}
\newtheorem{corollary}[theorem]{Corollary}

\newtheorem{conjecture}[theorem]{Conjecture}

\newtheorem*{theorem*}{Theorem}
\theoremstyle{remark}
\newtheorem{remark}[theorem]{Remark}

\numberwithin{equation}{section}

\newcommand{\lif}[1]{\widetilde{#1}}  

\newcommand{\com}[2][\theta]{#2^{#1}}

\newcommand{\End}[2]{\mathcal{E}_{#1}(#2)}

\newcommand{\tEnd}[3][\omega]{\mathcal{E}_{#2}(#3,#1)}

\renewcommand{\P}[1]{\Phi(#1)}

\newcommand{\Pbd}[1]{\Phi_{bdd}(#1)}

\newcommand{\cP}[1]{\bar{\Phi}(#1)}

\newcommand{\cPbd}[1]{\bar{\Phi}_{bdd}(#1)}

\newcommand{\p}{\phi}

\newcommand{\q}{\psi}

\newcommand{\lp}{\tilde{\phi}}

\newcommand{\Pkt}[1]{\Pi_{#1}}

\newcommand{\lPkt}[1]{\tilde{\Pi}_{#1}}

\newcommand{\cPkt}[1]{\bar{\Pi}_{#1}}

\newcommand{\clPkt}[1]{\tilde{\bar{\Pi}}_{#1}}

\renewcommand{\r}{\pi}

\newcommand{\lr}{\tilde{\pi}}

\newcommand{\sH}{\bar{\mathcal{H}}}

\newcommand{\cS}[1]{\bar{S}_{#1}}

\renewcommand{\S}[1]{\mathcal{S}_{#1}}

\renewcommand{\c}{\lambda}

\newcommand{\lG}{\widetilde{G}}

\newcommand{\x}{\omega}

\renewcommand{\L}[1]{{}^L#1}

\newcommand{\D}[1]{\widehat{#1}}

\newcommand{\Irr}[1]{\text{Irr}(#1)}

\newcommand{\Gal}[1]{\Gamma_{#1}}

\renewcommand{\a}{\alpha}

\renewcommand{\Im}{\text{Im}\,}

\newcommand{\lZ}{Z_{\widetilde{G}}}

\newcommand{\Z}{Z_{G}}

\newcommand{\iG}[1]{{G^{1}_{#1}}}

\newcommand{\ir}{\pi^{1}}

\newcommand{\Res}{\text{Res}}

\newcommand{\Ind}{\text{Ind}}

\newcommand{\Cent}{\text{Cent}}

\newcommand{\Two}{\mathbb{Z}_{2}}

\newcommand{\C}{\mathbb{C}}

\newcommand{\lf}{\tilde{f}}

\newcommand{\Int}{\text{Int}}

\newcommand{\Hom}{\text{Hom}}

\newcommand{\Ker}{\text{Ker}}

\newcommand{\e}{\varepsilon}

\begin{document}
\title{On a lifting problem of L-packets}

\author{Bin Xu}

\address{Department of Mathematics and Statistics \\  University of Calgary, Canada }
\email{bin.xu2@ucalgary.ca}

% 'MSC classification, keywords and grant acknowledgements'
\subjclass[2010]{22E50 (primary); 11F70 (secondary)}
\keywords{quasisplit reductive groups, endoscopic transfer, L-packets}

%\date{\today}

\maketitle

\begin{abstract}

Let $G \subseteq \lG$ be two quasisplit connected reductive groups over a local field of characteristic zero and they have the same derived group. Although the existence of L-packets is still conjectural in general, it is believed that the L-packets of $G$ should be the restriction of that of $\lG$. Motivated by this, we hope to construct the L-packets of $\lG$ from that of $G$. The primary example in our mind is when $G = Sp(2n)$, whose L-packets have been determined by Arthur (2013), and $\lG = GSp(2n)$. As a first step, we need to consider some well-known conjectural properties of L-packets. In this paper, we show how they can be deduced from the conjectural endoscopy theory. As an application, we obtain some structural information about L-packets of $\lG$ from that of $G$.  

\end{abstract}

\maketitle

%-------------------------------------------------------------------------------------------

\section{Some standard notations}% notation SECTION
\label{sec: notation}

Suppose $F$ is a field, we denote its algebraic closure by $\bar{F}$. Let $G$ be a reductive algebraic group over $F$ and $\theta$ be an $F$-automorphism of $G$. We denote the identity component of $G$ by $G^{0}$. If $G$ is connected, we denote the derived group of $G$ by $G_{der}$, and the adjoint group of $G$ by $G_{ad}$. Let $G_{sc}$ be the simply connected cover of $G_{der}$. If $\D{G}$ is the complex reductive group dual to $G$, we write $\D{G}_{der}$, $\D{G}_{ad}$ for the derived group and adjoint group of $\D{G}$ respectively, and $\D{G}_{sc}$ is the simply connected cover of $\D{G}_{der}$. We denote the centre of $G$ by $Z_{G}$ or $Z(G)$. If $G$ is abelian, let $G^{\theta}$ be the $\theta$-invariant subgroup of $G$, and $G_{\theta}$ be the $\theta$-coinvariant group of $G$, i.e., $G_{\theta} = G / (\theta - 1)G$. For a finite group $S$, we denote its set of linear characters by $S^{*}$.

%-------------------------------------------------------------------------------------------

\section{Introduction}% introduction SECTION
\label{sec: introduction}

Let $F$ be a local field of characteristic zero and $G$ be a quasisplit connected reductive group over $F$. The local Langlands conjecture asserts the set $\Pkt{}(G(F))$ of isomorphism classes of irreducible smooth representations of $G(F)$ can be parametrized by the set $\P{G}$ of local Langlands parameters. This parametrization is usually not a bijection. In fact it is conjectured that each parameter $\p \in \P{G}$ is associated with a finite set $\Pkt{\p}$ of isomorphism classes of irreducible smooth representations of $G(F)$, and they give a disjoint decomposition of
\begin{align}% disjoint decomposition Eq
\label{eq: disjoint decomposition}
\Pkt{}(G(F)) = \bigsqcup_{\p \in \P{G}} \Pkt{\p}.
\end{align}
Such finite sets are called L-packets. This parametrization is based on the belief that there should be certain arithmetic invariants (e.g., L-factor) defined on both the representation side and the parameter side so that one could match them. From this point of view, one can think the L-packet $\Pkt{\p}$ attached to some $\p \in \P{G}$ consists of all irreducible smooth representations of $G(F)$ whose arithmetic invariants match that of $\p$. However, it can be very difficult to define these arithmetic invariants on the representation side in general. On the other hand, there are some elementary properties that one would require this parametrization always satisfy. These properties are usually given under the name ``Desiderata" (see \cite{Borel:1979}, \cite{GGP:2012}). In this paper, we will mainly concern the following three desiderata.

\begin{itemize}

\item {\bf Desideratum 1: Central character}

The first desideratum asserts all irreducible smooth representations in $\Pkt{\p}$ have the same central character and it can be constructed from $\p$. To see this construction, we need to give the definition of local Langlands parameters. Let $\Gal{} = \text{Gal}(\bar{F}/F)$ be the absolute Galois group, $W_{F}$ be the Weil group and $\D{G}$ be the complex reductive group dual to $G$. The Langlands dual group is $\L{G} = \D{G} \rtimes W_{F}$, where the action of $W_{F}$ factors through $\Gal{}$. We define the local Langlands group to be 
\[
L_{F} := \begin{cases}
                      W_{F},   & F \text{ is archimedean},  \\
                      W_{F} \times SL(2, \C), & F \text{ is nonarchimedean}.
                      \end{cases}                   
\]
Then a Langlands parameter $\p$ is a $\D{G}$-conjugacy class of admissible homomorphisms from $L_{F}$ to $\L{G}$ (see \cite{Borel:1979}). In particular, it respects the projections on $W_{F}$ from both $L_{F}$ and $\L{G}$. We take a torus $Z$ defined over $F$, containing the centre $Z_{G}$ of $G$. For example, $Z$ can be a maximal torus of $G$. Let $\lG = (G \times Z) / Z_{G}$, where $Z_{G}$ is included diagonally, and let $D = Z / Z_{G}$. Then we have an exact sequence 
\begin{align}% central character Eq
\label{eq: central character}
\xymatrix{1 \ar[r] & G \ar[r] & \lG \ar[r]  & D \ar[r] & 1.}
\end{align}
On the dual side, we have 
\[
\xymatrix{1 \ar[r] & \D{D} \ar[r] & \D{\lG} \ar[r]  & \D{G} \ar[r] & 1.}
\]
This induces a map from $\P{\lG}$ to $\P{G}$. It follows from a result of Labesse (\cite{Labesse:1985}, Theorem 8.1) that this map is surjective. Therefore we can lift any $\p \in \P{G}$ to some $\lp \in \P{\lG}$. Note $Z_{\lG} = Z$ is a torus, so dual to 
\[
\xymatrix{1 \ar[r] & Z_{\lG} \ar[r] & \lG \ar[r]  & G_{ad} \ar[r] & 1}
\]
we have
\[
\xymatrix{1 \ar[r] &  \D{G}_{sc} \ar[r] & \D{\lG} \ar[r]  & \D{Z}_{\lG} \ar[r] & 1.}
\]
So by composing with $\D{\lG} \rightarrow \D{Z}_{\lG}$, $\lp$ gives rise to an element ${\bold a}_{\lp} \in H^{1}(W_{F}, \D{Z}_{\lG})$. Then by the local Langlands correspondence for tori, ${\bold a}_{\lp}$ corresponds to a quasicharacter $\chi_{\lp}$ of $Z_{\lG}(F)$. After we take restriction to $Z_{G}(F)$, we get a quasicharacter $\chi_{\p}$ of $Z_{G}(F)$. To see this construction is independent of the torus $Z$, we need to know two things. First, if there is another torus $Z_{1} \supseteq Z$, let $\lG_{1} = (G \times Z_{1}) / Z_{G}$ and $\lp_{1} \in \P{\lG_{1}}$ be a lift of $\lp$, then $\chi_{\lp_{1}}|_{Z_{\lG}} = \chi_{\lp}$. Secondly, if there are two torus $Z_{1}$ and $Z_{2}$ both containing $Z_{G}$, then there exists a third torus $Z_{3}$ containing both $Z_{1}$ and $Z_{2}$. The first thing follows easily from some commutative diagrams. For the second one, we can simply take $Z_{3} = (Z_{1} \times Z_{2}) / Z_{G}$.

\item {\bf Desideratum 2: $G_{ad}(F)$-conjugate action}

The second desideratum is more involved, and in particular it requires a different point of view towards L-packet. Roughly speaking, there are two steps in constructing the L-packets. First one constructs the L-packets for the set  $\Pkt{temp}(G(F))$ of isomorphism classes of irreducible tempered representations, and then the other L-packets (nontempered) can be constructed from the tempered ones by using the theory of Langlands' quotient. Therefore it suffices to know the tempered L-packets. The same is also true for the Langlands parametrization \eqref{eq: disjoint decomposition}. That is to say it is enough to know the parametrization of the tempered L-packets, which should correspond to the ``bounded parameters", namely the images of the Weil group part have compact closure. 
%The tempered representations play an important role in harmonic analysis for they form the support of Plancherel measure for $G(F)$. 
From the point of view of harmonic analysis, irreducible smooth representations are characterized by their ``characters", which are $G(F)$-conjugate invariant locally integrable functions over $G(F)$ and smooth over the open dense subset of strongly regular semisimple elements $G_{reg}(F)$. A virtual character $\Theta$ (i.e., finite linear combination of characters) is called {\bf stable} if it is $G(\bar{F})$-conjugate invariant over $G_{reg}(F)$, namely $\Theta(\gamma) = \Theta(\gamma')$ for any $\gamma, \gamma' \in G_{reg}(F)$ such that $\gamma = g^{-1} \gamma' g$ for some $g \in G(\bar{F})$. It is conjectured that the tempered L-packets are the minimal subsets of irreducible tempered representations, within which some linear combination of the characters is stable (cf. Conjecture 9.2, \cite{Shahidi:1990}). Therefore the conjugate action by $G_{ad}(F)$ on $\Pkt{}(G(F))$ permutes the elements in each tempered L-packet. Moreover, there is an explicit conjectural formula for describing this action, which will be the second desideratum. To state the formula, we need to introduce a parametrization for elements inside tempered L-packets, which will be called endoscopic parametrization.

Let us denote the set of bounded Langlands parameters by $\Pbd{G}$. For $\p \in \Pbd{G}$, we choose a representative $\underline{\p}: L_{F} \rightarrow \L{G}$ and define 
\[
S_{\underline{\p}} = \Cent(\Im \underline{\p}, \D{G}),
\] 
i.e., the centralizer of the image of $\underline{\p}$ in $\D{G}$. Let $S^{0}_{\underline{\p}}$ be the identity component of $S_{\underline{\p}}$ and $Z(\D{G})^{\Gamma}$ be the $\Gal{}$-invariant elements in the centre $Z(\D{G})$ of $\D{G}$. Then we also define $A_{\underline{\p}} = S_{\underline{\p}} / S^{0}_{\underline{\p}}$ and $\S{\underline{\p}} = S_{\underline{\p}}/S^{0}_{\underline{\p}}Z(\D{G})^{\Gamma}$. There is an exact sequence
\[
\xymatrix{1 \ar[r] & Z_{\underline{\p}} \ar[r] & A_{\underline{\p}} \ar[r] & \S{\underline{\p}} \ar[r] & 1,}
\]
where $Z_{\underline{\p}} = Z(\D{G})^{\Gamma} /  Z(\D{G})^{\Gamma} \cap S^{0}_{\underline{\p}}$. If $\underline{\p}^{g} = \Int (g) \circ \underline{\p} $ for $g \in \D{G}$, there is an isomorphism $S_{\underline{\p}} \rightarrow S_{\underline{\p}^{g}}$ unique up to $S_{\underline{\p}}$-conjugation. This means one can not define a group ``$S_{\p}$" independent of the choice of representatives $\underline{\p}$, but rather one can define the conjugacy classes in ``$S_{\p}$". 

We define a Whittaker datum to be a pair $(B, \Lambda)$, where $B$ is a Borel subgroup of $G$ and $\Lambda$ is a nondegenerate character on the unipotent radical $N(F)$ of $B(F)$. All Whittaker data can be constructed as follows. We fix an $F$-splitting $(B, T, \{X_{\alpha}\})$ of $G$ and a nontrivial additive character $\q_{F}: F \rightarrow \C^{\times}$, then we define 
\[
\Lambda(exp(\sum_{\alpha} n_{\alpha} X_{\alpha})) = \q_{F}(\sum_{\alpha} n_{\alpha}),
\]
which extends uniquely to a character of $N(F)$.

\begin{conjecture}% endoscopic parametrization CONJECTURE
\label{conj: endoscopic parametrization}
We fix a Whittaker datum $(B, \Lambda)$ for $G$, and suppose $\p \in \Pbd{G}$.
\begin{enumerate}
\item
There is a unique $(B, \Lambda)$-generic representation in $\Pkt{\p}$.
\item
There is a canonical pairing between $\Pkt{\p}$ and $\S{\underline{\p}}$, which induces an inclusion from $\Pkt{\p}$ to the set $\D{\S{\underline{\p}}}$ of characters  of irreducible representations of $\S{\underline{\p}}$
\[
\xymatrix{\Pkt{\p} \ar[r] & \D{\S{\underline{\p}}} \\
                 \r \ar@{{|}->}[r]  & <\cdot, \r>_{\underline{\p}},}
\]
such that it sends the $(B, \Lambda)$-generic representation to the trivial character. This becomes a bijection when $F$ is nonarchimedean. Moreover, if $\underline{\p}^{g} = \Int (g )\circ \underline{\p}$ for $g \in \D{G}$, then
\[
< g x g^{-1}, \r>_{\underline{\p}^{g}} = <x, \r>_{\underline{\p}},
\]
for $\r \in \Pkt{\p}$ and $x \in \S{\underline{\p}}$.
\end{enumerate}
\end{conjecture}

Since $\D{\S{\underline{\p}}}$ are functions on conjugacy classes of $\S{\underline{\p}}$, the parametrization of elements inside $\Pkt{\p}$ can be actually stated independent of the choice of representative $\underline{\p}$ in the conjecture. Nevertheless, we would like to work with the group $\S{\underline{\p}}$ rather than its conjugacy classes, so throughout this paper we will always fix a representative $\underline{\p}$. Let $\Irr{\S{\underline{\p}}}$ be the set of isomorphism classes of irreducible representations of $\S{\underline{\p}}$. If $\rho \in \Irr{\S{\underline{\p}}}$, we will denote the corresponding representation in $\Pkt{\p}$ by $\r(\rho)$. We call a parameter $\p \in \Pbd{G}$ {\bf simple} if $\S{\p} = 1$. For simple parameters, it follows from this conjecture that their corresponding packets are singletons. Finally, we want to point out part (i) of the conjecture is often referred to as the {\bf generic packet conjecture}, and the pairing in part (ii) comes from the conjectural endoscopic character identity (see Conjecture~\ref{conj: twisted character identity}), while its ``canonicity" depends on the choice of Whittaker datum.

Let $\S{\underline{\p}}^{*}$ be the group of linear characters of $\S{\underline{\p}}$. Then the explicit formula for describing the action of $G_{ad}(F)$ on $\Pkt{\p}$ can be stated in the following conjecture.

\begin{conjecture}% twisting character 1 CONJECTURE
\label{conj: twisting character 1}
There exists a homomorphism
\[
\xymatrix{G_{ad}(F) \ar[r] & \S{\underline{\p}}^{*} \\
                 g \ar[r]  & \eta_{g}}
\]  
such that
\[
<\cdot, \r^{g}>_{\underline{\p}} = \eta_{g}<\cdot, \r>_{\underline{\p}}.
\]
\end{conjecture}
The statement of this conjecture is first given in (\cite{GGP:2012}, Section 9.(3)), where they construct the homomorphism $G_{ad}(F) \rightarrow \S{\underline{\p}}^{*}$. There are three ingredients in that construction.
\begin{itemize}
 
\item (Tate local duality): There exists a perfect pairing 
\[
H^{1}(F, Z_{G}/Z^{0}_{G}) \times H^{1}(F, \pi_{1}(\D{G}_{der})) \rightarrow \C^{\times}.
\]

\item There is a coboundary map $A_{\underline{\p}} \rightarrow H^{1}(F, \pi_{1}(\D{G}_{der}))$. 

\item There is a homomorphism $G_{ad}(F) \rightarrow H^{1}(F, Z_{G}/Z^{0}_{G})$.

\end{itemize}
Clearly this gives a homomorphism $G_{ad}(F) \rightarrow A_{\underline{\p}}^{*}$, and in fact one will see the image is in $\S{\underline{\p}}^{*}$ (see Section~\ref{subsec: proof of conjecture}).

\item {\bf Desideratum 3: Twist by automorphism and quasicharacter}

Let $\theta$ be an $F$-automorphism of $G$ preserving an $F$-splitting of $G$, then $\theta$ acts on $\Pkt{}(G(F))$ by acting on $G(F)$. Let $\D{\theta}$ be the dual automorphism of $\theta$ on $\D{G}$, and it gives a semidirect product $\D{G} \rtimes <\D{\theta}>$. Then $\theta$ also acts on $\P{G}$ through the action of $\D{\theta}$ on $\D{G}$. Let ${\bold a}$ be an element in $H^{1}(W_{F}, Z(\D{G}))$, and ${\bold a}$ act on $\P{G}$ by twisting on $Z(\D{G})$. One can associate a quasicharacter $\x$ of $G(F)$ with ${\bold a}$ (see \eqref{eq: local character}). 
This desideratum asserts: for $\p \in \Pbd{G}$,
\[
\Pkt{\p^{\theta}} = \Pkt{\p}^{\theta} \text{ and } \Pkt{\p \otimes {\bold a}} = \Pkt{\p} \otimes \x.
\] 
In fact, one can refine this desideratum by making more precise the action of $\theta$ and $\x$ on the elements in $\Pkt{\p}$. Namely, if $\r \in \Pkt{\p}$, then 
\begin{align}% theta equivariant Eq
\label{eq: theta equivariant}
<x, \r^{\theta}>_{\underline{\p}^{\theta}} = <\D{\theta}^{-1}x\D{\theta}, \r>_{\underline{\p}}
\end{align} 
for $x \in \S{\underline{\p}^{\theta}}$, and 
\begin{align}
\label{eq: twist equivariant}
<x, \r \otimes \x> _{\underline{\p} \otimes \underline{{\bold a}}}= <x, \r>_{\underline{\p}}
\end{align} 
for $x \in \S{\underline{\p}} = \S{\underline{\p} \otimes \underline{{\bold a}}}$, where $\underline{{\bold a}}$ is a $1$-cocyle of $W_{F}$ in $Z(\D{G})$ representing ${\bold a}$. 

The refined desideratum has the following consequence. For $\p \in \Pbd{G}$, suppose $\p^{\theta} = \p \otimes {\bold a}$, i.e., there exists $g \in \D{G}$ such that $(\underline{\p}^{\theta})^g = \underline{\p} \otimes \underline{{\bold a}}$, then by \eqref{eq: theta equivariant} and \eqref{eq: twist equivariant}, we have for $x \in \S{\underline{\p}^{\theta}}$
\begin{align*}
<\D{\theta}^{-1}x\D{\theta}, \r>_{\underline{\p}} = <x, \r^{\theta}>_{\underline{\p}^{\theta}} = <g x g^{-1}, \r^{\theta}>_{(\underline{\p}^{\theta})^{g}} \\
= <g x g^{-1}, \r^{\theta}>_{\underline{\p} \otimes \underline{{\bold a}}} = <g x g^{-1}, \r^{\theta} \otimes \x^{-1}>_{\underline{\p}}.
\end{align*} 
By setting $s = g \rtimes \D{\theta}$, we have shown the following statement.  

\begin{conjecture}% twisting equivariant CONJECTURE
\label{conj: twisting equivariant}
Suppose $\p \in \Pbd{G}$ and $\p^{\theta} = \p \otimes {\bold a}$. Let $s \in \D{G} \rtimes \D{\theta}$ satisfying $\underline{\p}^{s} = \underline{\p} \otimes \underline{{\bold a}}$, then
\[
<s x s^{-1}, \r^{\theta} \otimes \x^{-1}>_{\underline{\p}} = <x , \r>_{\underline{\p}}
\]
for any $\r \in \Pkt{\p}$ and $x \in \S{\underline{\p}}$. In other words,
\[
\r(\rho^{s})^{\theta} \cong \r(\rho) \otimes \x,
\]
for any $\rho \in \Irr{\S{\underline{\p}}}$. 
%Furthermore, this statement can be made independent of choice of $\underline{\p}$.
\end{conjecture}

\end{itemize}

The first goal of this paper is to suggest a strategy towards proving the above three desiderata about L-packets. To do so, we need to assume \eqref{eq: disjoint decomposition}, Conjecture~\ref{conj: endoscopic parametrization} (together with its generalized form: Conjecture~\ref{conj: twisted endoscopic parametrization}), and also the (twisted) endoscopic character identities (see Conjecture~\ref{conj: twisted character identity}), which will be described in Section~\ref{sec: endoscopy}. Since these conjectures can be viewed as part of the conjectural endoscopy theory, we would like to call the collection of these assumptions {\bf Endoscopic Hypothesis}. For the first desideratum, we will prove the following result under this hypothesis.

\begin{proposition}% central character desideratum PROPOSITION
\label{prop: central character desideratum}
The desideratum about central characters of L-packets holds as long as it holds for simple parameters.
\end{proposition}

For the second desideratum, i.e., Conjecture~\ref{conj: twisting character 1}, we will prove a stronger result.  The setup that we are going to work on is as follows. Let $G \subseteq \lG$ be two quasisplit connected reductive groups over $F$ such that $G_{der} = \lG_{der}$. Then $\lG/G$ is a torus, and we denote it by $D$. There is an exact sequence
\begin{align}% extension Eq
\label{eq: extension}
\xymatrix{1 \ar[r] & G \ar[r] & \lG \ar[r]^{\lambda}  & D \ar[r] & 1.}
\end{align}
Let $\Sigma$ be a finite abelian group of $F$-automorphisms of $\lG$ preserving a fixed $F$-splitting of $\lG$, and we assume $\c$ is {\bf $\Sigma$-invariant}. This implies $\Sigma$ also acts on $G$. Let $\lG^{\Sigma} = \lG \rtimes \Sigma$ and $G^{\Sigma} = G \rtimes \Sigma$. Since $\Sigma$ induces dual automorphisms on $\D{\lG}$ and $\D{G}$, we denote them by $\D{\Sigma}$ and define $\D{\lG}^{\Sigma} = \D{\lG} \rtimes \D{\Sigma}$ and $\D{G}^{\Sigma} = \D{G} \rtimes \D{\Sigma}$. 

Before we can state our result, we need to extend Conjecture~\ref{conj: endoscopic parametrization} to the nonconnected group $G^{\Sigma}$. 
%In order to do so we would like to make two assumptions. First we need to assume if $\r^{\Sigma}$ is an irreducible admissible representation of $G^{\Sigma}(F) = G(F) \rtimes \Sigma$, then its restriction to $G(F)$ is multiplicity free. 
Suppose $\p \in \Pbd{G}$ and we define $S^{\Sigma}_{\underline{\p}}$, $A^{\Sigma}_{\underline{\p}}$ and $\S{\underline{\p}}^{\Sigma}$ as before simply by taking $\D{G}^{\Sigma}$ in place of $\D{G}$, and they are all equipped with a natural map to $\D{\Sigma}$. Let $S^{\theta}_{\underline{\p}}$, $A^{\theta}_{\underline{\p}}$ and $\S{\underline{\p}}^{\theta}$ be the preimage of $\D{\theta} \in \D{\Sigma}$ in $S^{\Sigma}_{\underline{\p}}$, $A^{\Sigma}_{\underline{\p}}$ and $\S{\underline{\p}}^{\Sigma}$ respectively. Note these are not $\D{\theta}$-invariant elements in $S_{\underline{\p}}$, $A_{\underline{\p}}$ and $\S{\underline{\p}}$. Since the image in $\D{\Sigma}$ is the same for $S^{\Sigma}_{\underline{\p}}$, $A^{\Sigma}_{\underline{\p}}$ and $\S{\underline{\p}}^{\Sigma}$, we denote it by $\D{\Sigma}_{\underline{\p}}$. 
%The second assumption is the restriction of any irreducible representation of $\S{\underline{\p}}^{\Sigma}$ to $\S{\underline{\p}}$ is multiplicity free. 
Let $\Pkt{\p}^{\Sigma}$ be the set of all irreducible smooth representations of $G^{\Sigma}(F)$, whose restriction to $G(F)$ have intersections with $\Pkt{\p}$.

A Whittaker datum $(B, \Lambda)$ is called $\Sigma$-stable if $\Sigma$ preserves $B$ and $\Lambda$ is $\Sigma$-invariant. In particular, if we fix a $\Sigma$-stable $F$-splitting of $G$ (i.e., $\Sigma$ preserves $B$ and $\{X_{\alpha}\}$) and a nontrivial additive character $\q_{F}$ of $F$, then the associated Whittaker datum is $\Sigma$-stable. We call a representation $\r^{\Sigma} \in \Pkt{\p}^{\Sigma}$ $(B, \Lambda)$-generic if $\r^{\Sigma}|_{G}$ is $(B, \Lambda)$-generic and the corresponding Whittaker functional is invariant under $\r^{\Sigma}(\theta)$ for all $\theta \in \Sigma$.

\begin{conjecture}% twisted endoscopic parametrization CONJECTURE
\label{conj: twisted endoscopic parametrization}
We fix a $\Sigma$-stable Whittaker datum $(B, \Lambda)$ for $G$, and suppose $\p \in \Pbd{G}$.
\begin{enumerate}
\item
There is a unique $(B, \Lambda)$-generic representation in $\Pkt{\p}^{\Sigma}$.
\item
There is a canonical pairing between $\Pkt{\p}^{\Sigma}$ and $\S{\underline{\p}}^{\Sigma}$, which induces an inclusion from $\Pkt{\p}^{\Sigma}$ to the characters $\D{\S{\underline{\p}}}^{\Sigma}$ of irreducible representations of $\S{\underline{\p}}^{\Sigma}$
\[
\xymatrix{\Pkt{\p}^{\Sigma} \ar[r] & \D{\S{\underline{\p}}}^{\Sigma} \\
                 \r^{\Sigma} \ar@{{|}->}[r]  & <\cdot, \r^{\Sigma}>_{\underline{\p}}.}
\]
such that it sends the $(B, \Lambda)$-generic representation to the trivial character. This becomes a bijection when $F$ is nonarchimedean. Moreover, if $\Sigma'$ is a subgroup of $\Sigma$, then we have the following relation:

%\begin{itemize}

%\item If $\theta \in \Sigma / \Sigma_{\underline{\p}}$, it induces a canonical isomorphism $\S{\underline{\p}}^{+} \cong \S{\underline{\p}^{\theta}}^{+}$ up to $\S{\underline{\p}}-$conjugation. Let 
%\[
%\r^{+} |_{G \rtimes \Sigma_{\underline{\p}}}= \sum_{\theta \in \Sigma / \Sigma_{\underline{\p}}} (\r^{\Sigma_{\underline{\p}}})^{\theta},
%\]
%where $\r^{\Sigma_{\underline{\p}}}|_{G} \subseteq \Pkt{\p}$. Then under those isomorphisms we have
%\[
%<\cdot, \r^{+}> = <\cdot, \r^{\Sigma_{\p}}> = <\cdot, (\r^{\Sigma_{\p}})^{\theta}>.
%\]

%\item If $\theta \in \Sigma_{\p}$ and $\r \in \Pkt{\p}$, then $\r^{\theta} \in \Pkt{\p}$ satisfying
%\[
%<\cdot, \r^{\theta}> = <s_{\theta}^{-1}(\cdot)s_{\theta}, \r>,
%\]
%where $s_{\theta}$ is the preimage of $\theta$ in $\S{\p}^{+}$.

%\item 
\begin{align}% twisted endoscopic parametrization Eq
\label{eq: twisted endoscopic parametrization}
<\cdot, \r^{\Sigma}>_{\underline{\p}} |_{\S{\underline{\p}}^{\Sigma'}} = \sum_{\r^{\Sigma'} \in \Pkt{\p}^{\Sigma'}} m(\r^{\Sigma}, \r^{\Sigma'}) <\cdot, \r^{\Sigma'}>_{\underline{\p}},
\end{align}
where $m(\r^{\Sigma}, \r^{\Sigma'})$ is the multiplicity of $\r^{\Sigma'}$ in $\r^{\Sigma}|_{G^{\Sigma'}}$.
%\end{itemize}
\end{enumerate}

\end{conjecture}

%To see why our assumptions are needed, one notes the following special case of Conjecture~\ref{conj: twisting equivariant}: 
%\[
%\text{ if $\p^{\theta} = \p$, then 
%\(
%\r(\rho^{x})^{\theta} \cong \r(\rho)
%\)
%for any $x \in \S{\underline{\p}}^{\theta}$ and $\rho \in \Irr{\S{\underline{\p}}}$. }
%\]
%Therefore if $\Sigma(\r(\rho))$ is the subgroup of $\Sigma$ stabilizing $\r(\rho)$, then the stabilizer $\S{\underline{\p}}^{\Sigma}(\rho)$ of $\rho$ in $\S{\underline{\p}}^{\Sigma}$ has image $\D{\Sigma}(\r(\rho)) \cong \Sigma(\r(\rho))$ in $\D{\Sigma}$. The first multiplicity one assumption guarantees $\r$ can be extended to $G(F) \rtimes \Sigma(\r(\rho))$ and the second multiplicity one assumption guarantees $\rho$ can be extended to $\S{\underline{\p}}^{\Sigma}(\rho)$. So in this way one could at least get the same cardinality for both $\Pkt{\p}^{\Sigma}$ and $\D{\S{\underline{\p}}}^{\Sigma}$ when $F$ is nonarchimedean.

Under the endoscopic hypothesis, we are able to prove the following result.

\begin{theorem}% twisting character THEOREM
\label{thm: twisting character}
There exists a homomorphism
\[
\xymatrix{\lG(F) \ar[r] & (\S{\underline{\p}}^{\Sigma})^{*} \\
                 g \ar[r]  & \e_{g}}
\]  
such that
\[
<\cdot, (\r^{\Sigma})^{g}>_{\underline{\p}} = \e_{g}<\cdot, \r^{\Sigma}>_{\underline{\p}}.
\]
\end{theorem}
%In Section~\ref{subsec: proof of conjecture}, we will show this theorem implies Conjecture~\ref{conj: twisting character 1}.

For the third desideratum, we will prove the following result under the endoscopic hypothesis.

\begin{proposition}% twisting equivariant PROPSOTION
\label{prop: twisting equivariant}
The refined desideratum about L-packets under twist by automorphism and quasicharacter holds if it holds for simple parameters. 
\end{proposition}

%The proofs for Proposition~\ref{prop: central character}, Theorem~\ref{thm: twisting character} and Propositon~\ref{prop: twisting equivariant} all rely on understanding the properties of the geometric transfer map (see Section~\ref{subsec: endoscopic transfer}). Especially, we want to point out in the proof of Theorem~\ref{thm: twisting character}, a key observation is the character $\x_{x}(g) := \e_{g}(x)$ of $\lG(F)$ for any fixed $x \in \S{\underline{\p}}^{\Sigma}$ is associated with some twisted endoscopic datum that $\lp$ (lift of $\p$) factors through (see Lemma~\ref{lemma: twisted character}).

Back to the setup in \eqref{eq: extension}, there is a conjectural relation between the L-packets for $G$ and $\lG$. That is to say if $\lp \in \P{\lG}$ maps to $\p \in \P{G}$, then the L-packet $\Pkt{\p}$ should be the restriction of $\Pkt{\lp}$. The restriction multi-map $\Pkt{}(\lG(F)) \rightarrow \Pkt{}(G(F))$ is surjective, in the sense that for any $\r \in \Pkt{}(G(F))$, there exists $\lr \in \Pkt{}(\lG(F))$, whose restriction to $G(F)$ contains $\r$ (see Corollary~\ref{cor: existence}). Therefore it is easy to construct the L-packets of $G$ from that of $\lG$. The other direction is more subtle, because for any $\r \in \Pkt{}(G(F))$, the preimage $\lr \in \Pkt{}(\lG(F))$ is usually not unique and they differ from each other by a twist of quasicharacters of $\lG(F)$. So our second goal in this paper is to make an attempt to address this problem in most generality. To be more precise, we want to establish the endoscopic hypothesis (i.e., \eqref{eq: disjoint decomposition}, Conjecture~\ref{conj: twisted endoscopic parametrization} and Conjecture~\ref{conj: twisted character identity}) for $\lG$ by assuming it for $G$ and the twisted endoscopic groups of $G$. When $G$ is a quasisplit symplectic group or special even orthogonal group, and $\lG$ is the corresponding similitude group, this has been essentially achieved in \cite{Xu:preprint2}.

%Since this paper is intended for setting up a general frame work, the main results of this paper are all conditional on various working assumptions. Here we want to explain how this paper is organized and clarify the working assumptions in different parts of the paper. 

Throughout this paper except for Section~\ref{sec: lifting L-packet}, we will take the endoscopic hypothesis as our working assumption. In Section~\ref{sec: endoscopy}, we will describe the conjectural endoscopy theory. In particular, we will introduce Conjecture~\ref{conj: twisted character identity}, which is part of the endoscopic hypothesis. We will prove Theorem~\ref{thm: twisting character}, and deduce Conjecture~\ref{conj: twisting character 1} as a special case. In Section~\ref{sec: central character}, we will prove Proposition~\ref{prop: central character desideratum}. In Section~\ref{sec: twist}, we will prove Proposition~\ref{prop: twisting equivariant}, and this implies Conjecture~\ref{conj: twisting equivariant} for non-simple parameters.

In Section~\ref{sec: lifting L-packet}, we consider the problem of lifting L-packets from $G$ to $\lG$, where $G$ and $\lG$ are in the setup of \eqref{eq: extension}. So we will only assume the endoscopic hypothesis for $G$ and its twisted endoscopic groups, in particular, we can not assume Conjecture~\ref{conj: twisting equivariant} for $\lG$. In Section~\ref{subsec: representation}, we will study the restriction multi-map $\Pkt{}(\lG(F)) \rightarrow \Pkt{}(G(F))$. In Sections~\ref{subsec: coarse L-packet}, \ref{subsec: compatibility with twist}, we will discuss some special cases of Conjecture~\ref{conj: twisting equivariant} for $\lG$, and from there we will obtain some structural information about the L-packets of $\lG$. In Section~\ref{subsec: refinement}, we will formulate a conjecture about the L-packets of $\lG$ (see Conjecture~\ref{conj: refined L-packet}). Finally, in Section~\ref{subsec: classical group}, we will take $G$ to be a symplectic group or special even orthogonal group, and we will review various results of Arthur in \cite{Arthur:2013}, which essentially prove the endoscopic hypothesis for $G$. We will also take $\lG$ to be the corresponding similitude group, and apply the previous discussion in Section~\ref{sec: lifting L-packet} to this case. So the results we obtain in Sections~\ref{subsec: coarse L-packet}, \ref{subsec: compatibility with twist} will become unconditional in this case. Moreover, we will restate Conjecture~\ref{conj: refined L-packet} as a theorem in this case, and the proof of this theorem is included in \cite{Xu:preprint2}.

%This paper is organized as follows. In Section~\ref{sec: endoscopy}, we will describe the conjectural endoscopy theory. We will prove Theorem~\ref{thm: twisting character} and deduce Conjecture~\ref{conj: twisting character 1} from there as a special case. In Section~\ref{sec: central character}, we will prove Proposition~\ref{prop: central character}. In Section~\ref{sec: twist}, we will prove Proposition~\ref{prop: twisting equivariant}. In Section~\ref{sec: lifting L-packet}, we will study the restriction map $\Pkt{}(\lG(F)) \rightarrow \Pkt{}(G(F))$. We will also prove a special case of Conjecture~\ref{conj: twisting equivariant} for $\lG$ only by using the endoscopy theory for $G$, and from there we will be able to obtain some structural information about the L-packets of $\lG$. In the end, we will formulate a conjecture about the L-packets of $\lG$, and we will also state our results when $\lG$ is the similitude group of a symplectic group or a special even orthogonal group. 

{\bf Acknowledgements}: The author wants to thank the referee for many suggestions on improving the readability of the manuscript. This paper is based upon work supported by the National Science Foundation number agreement No. DMS-1128155 and DMS-1252158. %Any opinions, findings and conclusions or recommendations expressed in this paper are those of the author and do not necessarily reflect the views of the National Science Foundation.

%-------------------------------------------------------------------------------------------
\section{Endoscopy theory}% endoscopy SECTION
\label{sec: endoscopy}

\subsection{Twisted endoscopic datum}% endoscopic group SUBSECTION
\label{subsec: endoscopic group}

Let $F$ be a local field of characteristic zero and $G$ be a quasisplit reductive group over $F$. We have an isomorphism 
\begin{align}% local character Eq
\label{eq: local character}
H^{1}(W_{F}, Z(\D{G})) \longrightarrow \Hom(G(F), \C^{\times})
\end{align}
defined by Langlands (see Appendix~\ref{sec: character}). Let $\theta$ be an automorphism of $G$, $\x$ be a quasicharacter of $G(F)$. A twisted endoscopic datum for $(G, \theta, \x)$ is a quadruple $(H, \mathcal{H}, s, \xi)$, where $H$ is a quasisplit reductive group over $F$, $\mathcal{H}$ is a split extension of $W_{F}$ by $\D{H}$
\[
\xymatrix{1 \ar[r] & \D{H} \ar[r] & \mathcal{H} \ar[r]  & W_{F} \ar[r] & 1,}
\]
such that the conjugate action of $W_{F}$ on $\D{H}$ falls into the same outer classes of automorphisms as for $\L{H}$. Note $\mathcal{H}$ may not be isomorphic to $\L{H}$. Inside the quadruple, $s$ is a semisimple element in $\D{G} \rtimes \D{\theta}$, $\xi$ is an L-embedding of $\mathcal{H}$ to $\L{G}$ (i.e., it respects the projections on $W_{F}$ from both $\mathcal{H}$ and $\L{G}$), and they satisfy the following conditions:
\begin{itemize}
\item 
\(
\Int(s) \circ \xi = \underline{{\bold a}} \cdot \xi, \text{ for a $1$-cocycle $\underline{{\bold a}}$ of $W_{F}$ in $Z(\D{G})$ mapped to $\x$ by \eqref{eq: local character}}; 
\)
\item $\D{H} \cong \Cent(s, \D{G})^{0}$ through $\xi$.
\end{itemize}
We call $H$ a twisted endoscopic group of $G$. Two twisted endoscopic data $(H, \mathcal{H}, s, \xi)$ and $(H', \mathcal{H}', s', \xi')$ are called isomorphic if there exists an element $g \in \D{G}$ such that $g\xi(\mathcal{H})g^{-1} = \xi'(\mathcal{H'})$ and $gsg^{-1} \in s'Z(\D{G})$. We denote by $\tEnd{}{\com{G}}$ the set of isomorphism classes of twisted endoscopic data for $(G, \theta, \x)$. For abbreviation, we will use the twisted endoscopic group to denote the twisted endoscopic datum if there is no confusion.

Let $G \subseteq \lG$ be two quasisplit connected reductive groups over $F$ such that $G_{der} = \lG_{der}$ and we denote $\lG/G$ by $D$.
\begin{align*}
\xymatrix{1 \ar[r] & G \ar[r] & \lG \ar[r]^{\lambda}  & D \ar[r] & 1}
\end{align*}
We assume $\theta$ is an automorphism of $\lG$, and $\c$ is $\theta$-invariant. Then we have the following proposition relating the twisted endoscopic data between $G$ and $\lG$.

\begin{proposition}%lifting endoscopic group lemma PROPOSITION
\label{prop: lifting endoscopic group}
There is a one to one correspondence between $\End{}{\com{G}, \x_{G}}$ and 
\[
\bigsqcup_{\x_{\lG}|_{G} = \x_{G}} \End{}{\com{\lG}, \x_{\lG}}.
\] 
\end{proposition}

\begin{proof}
Suppose $[(H, \mathcal{H}, s, \xi)] \in \End{}{\com{G}, \x_{G}}$, then $\xi(\D{H})  = \com[0]{\Cent(s, \D{G} )}$. Under the projection $ \L{\lG} \longrightarrow \L{G}$, the preimage of $\Cent(s, \D{G})$ is $\{ g \in \D{\lG} : \lif{s} g \lif{s}^{-1} g^{-1} \in \D{D} \}$, where $\lif{s}$ is a preimage of $s$ in $\D{\lG} \rtimes \D{\theta}$. We claim
\[
\com[0]{  \{ g \in \D{\lG} : \lif{s} g \lif{s}^{-1} g^{-1} \in \D{D} \}  } = \com[0]{  \{ g \in \D{\lG} : \lif{s} g \lif{s}^{-1} g^{-1} =1 \}  }
\]
To see this we can consider the homomorphism defined by 
\begin{align}% lifting endoscopic group Eq
\label{eq: lifting endoscopic group}
\xymatrix{ \{ g \in \D{\lG} : \lif{s} g \lif{s}^{-1} g^{-1} \in \D{D} \} \ar[rr] && \D{D} \subseteq \L{\lG} \\
g \ar@{|->}[rr]   &&  \lif{s} g \lif{s}^{-1} g^{-1}. }
\end{align}
Its composition with $A: \L{\lG} \longrightarrow \L((Z_{\lG}^{\theta})^{0})$ is trivial. Note $A$ induces an isogeny 
\[
(Z(\D{\lG})^{\D{\theta}})^{0} \rightarrow \D{(Z^{\theta}_{\lG})^{0}}.
\] 
Since $\c$ is $\theta$-invariant, $\D{D}$ included as a subgroup of $\D{\lG}$ is fixed by $\D{\theta}$. Therefore $\D{D} \subseteq (Z(\D{\lG})^{\D{\theta}})^{0}$, and we get $\Ker A|_{\D{D}}$ is finite. This means the homomorphism \eqref{eq: lifting endoscopic group} must have finite image, so our claim becomes obvious. Since $\D{D} \subseteq \com[0]{\Cent(\lif{s}, \D{\lG} ) }$, we can now conclude $\com[0]{\Cent(\lif{s}, \D{\lG} ) }$ is the preimage of $\xi(\D{H})$. Let us denote $\com[0]{\Cent(\lif{s}, \D{\lG} ) }$ by $\D{\lif{H}}$.

When $\theta = id$ and $s \in Z(\D{G})$, we have $\D{G} = \Cent(s, \D{G})$ and hence $\D{\lG} = \{ g \in \D{\lG} : \lif{s} g \lif{s}^{-1} g^{-1} \in \D{D} \}$. Since $\D{\lG}$ is connected, it follows from the above argument that $\D{\lG} = \com[0]{  \{ g \in \D{\lG} : \lif{s} g \lif{s}^{-1} g^{-1} =1 \}}$. Hence $\D{\lG} = \{ g \in \D{\lG} : \lif{s} g \lif{s}^{-1} g^{-1} =1 \}$, which means $\lif{s} \in Z(\D{\lG})$. This shows the preimage of $Z(\D{G})$ is $Z(\D{\lG})$, i.e., there is an exact sequence
\[
\xymatrix{1 \ar[r] & \D{D} \ar[r] & Z(\D{\lG}) \ar[r]  & Z(\D{G}) \ar[r] & 1.}
\]

Back to the general situation, we can choose a splitting $c: W_{F} \rightarrow \mathcal{H}$, so that the composition 
\(
\underline{\p}: \xymatrix{W_{F} \ar[r]^{c} & \mathcal{H} \ar[r]^{\xi} & \L{G}}
\)
is admissible. Then we can lift $\underline{\p}$ to $\underline{\lp}: W_{F} \rightarrow \L{\lG}$, which induces a Galois action on $\D{\lif{H}}$ and hence determines a quasisplit reductive group $\lif{H}$. We define $\mathcal{\lif{H}}$ to be the product $\D{\lif{H}} \cdot \Im \underline{\lp}$. Note $\underline{\lp}$ gives a splitting of 
\[
\xymatrix{1 \ar[r] & \D{\lif{H}} \ar[r] & \mathcal{\lif{H}} \ar[r]  & W_{F} \ar[r] & 1,}
\]
and we have a natural embedding $\lif{\xi} : \mathcal{\lif{H}} \rightarrow \L{\lG}$, which is identity on $\D{\lif{H}}$. The map
\[
w \mapsto \lif{s} \lif{\xi} (\underline{\lp}(w)) \lif{s}^{-1}  \lif{\xi} (\underline{\lp}(w))^{-1}, \,\,\,\,\,\, w \in W_{F}
\]
defines an element ${\bold a} \in H^{1}( W_{F}, Z(\D{\lG}))$. If ${\bold a}$ is associated with a quasicharacter $\x_{\lG}$ of $\lG(F)$, then $[(\lif{H}, \mathcal{\lif{H}}, \lif{s}, \lif{\xi})] \in \End{}{\com{\lG}, \x_{\lG}}$. It is not hard to show if we change either $(H, \mathcal{H}, s, \xi)$ within its isomorphism class or the splitting $c$ or the lifting $\underline{\lp}$, this lifted endoscopic datum $(\lif{H}, \mathcal{\lif{H}}, \lif{s}, \lif{\xi})$ is uniquely determined up to isomorphism. Here we need to use the fact that the preimage of $Z(\D{G})$ is $Z(\D{\lG})$. 

Finally, we have a commutative diagram 
\[
\xymatrix{H^{1}( W_{F}, Z(\D{\lG})) \ar[d] \ar[r] & \Hom(\lG(F), \C^{\times})  \ar[d] \\
                H^{1}( W_{F}, Z(\D{G})) \ar[r] & \Hom(G(F), \C^{\times}), }
\]
which shows $\x_{\lG}|_{G} = \x_{G}$. So we get a well defined map from $\End{}{\com{G}, \x_{G}}$ to 
\[
\bigsqcup_{\x_{\lG}|_{G} = \x_{G}} \End{}{\com{\lG}, \x_{\lG}}.
\] 
The other direction is more straightforward, namely one can simply take the quotient by $\D{D}$ on the dual side.

\end{proof}

\begin{remark}
Following the proof, there is an exact sequence
\[
\xymatrix{1 \ar[r] & \D{D} \ar[r]  & \D{\lif{H}} \ar[r] & \D{H} \ar[r] & 1 },
\]
whose dual is
\[
\xymatrix{1 \ar[r] & H \ar[r]  & \lif{H} \ar[r]^{\c_{H}} & D \ar[r] & 1 }.
\]
This suggests the twisted endoscopic groups $\lif{H}$ and $H$ also have the same derived group. 
\end{remark}

\subsection{Relation with Langlands parameter}% Langlands parameter SUBSECTION
\label{subsec: Langlands parameter}

We follow the setup in the introduction. Suppose $\p \in \P{G}$ and $\lp \in \P{\lG}$ is a lift of $\p$. Let $L_{F}$ act on $\D{D}$, $\D{\lG}^{\Sigma}$ and $\D{G}^{\Sigma}$ by conjugation through $\underline{\p}$. We denote the corresponding group cohomology by $H^{*}_{\underline{\p}}(L_{F}, \cdot)$. Note $H^{0}_{\underline{\p}}(L_{F}, \D{D}) = \D{D}^{\Gal{}}, H^{0}_{\underline{\p}}(L_{F}, \D{G}^{\Sigma}) = S^{\Sigma}_{\underline{\p}}$, $H^{0}_{\underline{\p}}(L_{F}, \D{\lG}^{\Sigma}) = S_{\underline{\lp}}^{\Sigma}$ and $H^{1}_{\underline{\p}}(L_{F}, \D{D}) = H^{1}(W_{F}, \D{D})$. The short exact sequence 
\[
\xymatrix{1 \ar[r] & \D{D} \ar[r] & \D{\lG}^{\Sigma} \ar[r]  & \D{G}^{\Sigma} \ar[r] & 1}
\]
induces a long exact sequence
\begin{align*}
\xymatrix{1 \ar[r] &  \D{D}^{\Gamma} \ar[r] & S_{\underline{\lp}}^{\Sigma} \ar[r] & S_{\underline{\p}}^{\Sigma} \ar[r]^{\delta \quad \quad} & H^{1}(W_{F}, \D{D}),}
\end{align*}
and hence
\begin{align}% old twisted endoscopic sequence Eq
\label{eq: old twisted endoscopic sequence}
\xymatrix{1 \ar[r] &  S_{\underline{\lp}}^{\Sigma}/\D{D}^{\Gal{}} \ar[r]^{\quad \iota} & S_{\underline{\p}}^{\Sigma} \ar[r]^{\delta \quad \quad} & H^{1}(W_{F}, \D{D}).}
\end{align}

To describe $\delta$, we can write 
\[
S_{\underline{\p}}^{\Sigma} = \{ \lif{s} \in \D{\lG}^{\Sigma}: \lif{s} \underline{\lp}(u) \lif{s}^{-1}\underline{\lp}(u)^{-1} \in \D{D}, \text{ for all } u \in L_{F}\} /  \D{D}. 
\]
Then $\delta(s) : u \longmapsto \lif{s} \underline{\lp}(u) \lif{s}^{-1}\underline{\lp}(u)^{-1}$, where $\lif{s}$ is a preimage of $s$ in $\D{\lG}^{\Sigma}$, and $\delta(s)$ factors through $W_{F}$. We have the following fact about $\delta$.

\begin{lemma}%centralizer LEMMA
\label{lemma: centralizer}
%\begin{enumerate}

%\item 
The image of $\delta$ consists of ${\bold a} \in H^{1}(W_{F}, \D{D})$ such that 
\[
\lp^{\theta} = \lp \otimes {\bold a}
\]
for some $\theta \in \Sigma$, and in particular it is finite.

%\item $\com[0]{(S^{\Sigma}_{\underline{\lp}}/ \D{D}^{\Gal{}})} = (S^{\Sigma}_{\underline{\p}})^{0}$ 

%and $\com[0]{\cS{\underline{\lp}}} = \com[0]{\cS{\underline{\p}}}$. 

%\end{enumerate}
\end{lemma}

\begin{proof}
By the definition of $\delta$, we have $\lif{s} \underline{\lp}(u) \lif{s}^{-1} = \delta(s)(u) \cdot \underline{\lp}(u)$, where $\lif{s} \in \D{\lG}^{\Sigma}$ is a preimage of $s$. Denote by $\D{\theta}_{s}$ the image of $s$ in $\D{\Sigma}$. Then this means $\lp^{\theta_{s}} = \lp \otimes \delta(s)$. Conversely, if $\lp^{\theta} = \lp \otimes {\bold a}$ for some ${\bold a} \in H^{1}(W_{F}, \D{D})$ and $\theta \in \Sigma$, then there exists $g \in \D{\lG}$ such that 
\[
(g \rtimes \D{\theta}) \underline{\lp}(u) (g \rtimes \D{\theta})^{-1} = \underline{{\bold a}}(u) \underline{\lp}(u).
\] 
for a $1$-cocycle $\underline{\bold a}$ representing ${\bold a}$. Then it is clear that $\lif{s} := g \rtimes \D{\theta} \in \D{\lG}^{\Sigma}$ maps to an element $s \in S_{\underline{\p}}^{\Sigma}$ and ${\bold a} = \delta(s)$. 

To see the image of $\delta$ is finite, we consider $\delta(s)$ and let $\theta = \theta_{s}$. The restriction of 

\begin{align*}
A: \L{\lG} \longrightarrow \L((Z_{\lG}^{\theta})^{0})
\end{align*}
to $\D{D}$ induces a homomorphism

\begin{align*}
B: H^{1}(W_{F}, \D{D}) \rightarrow H^{1}(W_{F}, \D{(Z_{\lG}^{\theta})^{0}}).
\end{align*}
We claim $\delta(s)$ lies in the kernel of $B$. To show the claim, recall 
\[
\delta(s) : u \longmapsto \lif{s} \underline{\lp}(u) \lif{s}^{-1}\underline{\lp}(u)^{-1}.
\] 
We can write $\lif{s} = g \rtimes \D{\theta}$ and $\underline{\lp}(u) = h \rtimes w_{u}$, where $g, h \in \D{\lG}$ and $w_{u} \in W_{F}$. Then
\[
\lif{s} \underline{\lp}(u) \lif{s}^{-1}\underline{\lp}(u)^{-1} := g \D{\theta} (h)  \cdot w_{u} (g^{-1})  h^{-1}.  
\]
Since 
\[
A(\D{\theta} (h)) = A(h),
\]
we have
\[
A(\lif{s} \underline{\lp}(u) \lif{s}^{-1}\underline{\lp}(u)^{-1}) = A(g) A(h)  \cdot w_{u} (A(g^{-1}))  A(h^{-1}) =  A(g)\cdot w_{u} (A(g)^{-1}).
\]
This proves the claim. Now the exact sequence 
\[
\xymatrix{1 \ar[r] & \Ker A|_{\D{D}} \ar[r] & \D{D} \ar[r]^{A|_{\D{D}} \quad } & \D{(Z_{\lG}^{\theta})^{0}}}
\]
induces the following exact sequence
\[
\xymatrix{1 \ar[r] & H^{1}(W_{F}, \Ker A|_{\D{D}}) \ar[r] & H^{1}(W_{F}, \D{D}) \ar[r]^{B \quad } & H^{1}(W_{F}, \D{(Z_{\lG}^{\theta})^{0}})}.
\]
Since $F$ is a local field and $\Ker A|_{\D{D}}$ is finite, it is not hard to see $H^{1}(W_{F}, \Ker A|_{\D{D}})$ is finite. Then it follows from the exact sequence that the kernel of $B$ is also finite, and hence $\Im \delta$ is finite.

%Its restriction to $\D{D}$ has finite kernel. Since $\delta(s)$ composed with this map is trivial, the image of $\delta(s)$ lies in that finite kernel. So $\Im \delta$ has to be finite. 
%Finally, one notes this implies $\com[0]{(S^{\Sigma}_{\underline{\lp}} / \D{D}^{\Gal{}})} = (S_{\underline{\p}}^{\Sigma})^{0}$.
\end{proof}

We would like to modify \eqref{eq: old twisted endoscopic sequence} to have $\S{\underline{\p}}^{\Sigma}$ and $\S{\underline{\lp}}^{\Sigma}$ in the sequence. To do so we need to know the kernel and image of $\delta$ restricted on $Z(\D{G})^{\Gal{}}$. Therefore we take
\[
\xymatrix{1 \ar[r] & \D{D} \ar[r] & Z(\D{\lG}) \ar[r]  & Z(\D{G}) \ar[r] & 1}
\]
which induces an exact sequence
\[
\xymatrix{1 \ar[r] &  \D{D}^{\Gamma} \ar[r] & Z(\D{\lG})^{\Gal{}} \ar[r] & Z(\D{G})^{\Gal{}} \ar[r]^{\delta \quad} & H^{1}(W_{F}, \D{D}) \ar[r] & H^{1}(W_{F}, Z(\D{\lG})). }
\]
So $\Ker \delta|_{Z(\D{G})^{\Gal{}}} = Z(\D{\lG})^{\Gal{}} / \D{D}^{\Gamma}$. Let 
\(
\bar{H}^{1}(W_{F}, \D{D}) : = H^{1}(W_{F}, \D{D}) / \delta(Z(\D{G})^{\Gal{}}),
\)
and we define $\cS{\underline{\p}}^{\Sigma} = S_{\underline{\p}}^{\Sigma} / Z(\D{G})^{\Gal{}}$ and $\cS{\underline{\lp}}^{\Sigma} = S_{\underline{\lp}}^{\Sigma} / Z(\D{\lG})^{\Gal{}}$. By taking the quotient of \eqref{eq: old twisted endoscopic sequence} by $Z(\D{G})^{\Gal{}}$, we get
\begin{align}% twisted endoscopic sequence mod center Eq
\label{eq: twisted endoscopic sequence mod center}
\xymatrix{1 \ar[r] &  \cS{\underline{\lp}}^{\Sigma} \ar[r]^{\iota} & \cS{\underline{\p}}^{\Sigma} \ar[r]^{\bar{\delta} \quad \quad}  & \bar{H}^{1}(W_{F}, \D{D})}.                 
\end{align}
Since $\Im \delta$ is finite, we have $(\cS{\underline{\lp}}^{\Sigma})^{0} = (\cS{\underline{\p}}^{\Sigma})^{0}$. After taking the quotient of \eqref{eq: twisted endoscopic sequence mod center} by the identity component, we get
\begin{align}% twisted endoscopic sequence Eq
\label{eq: twisted endoscopic sequence}
\xymatrix{1 \ar[r] &  \S{\underline{\lp}}^{\Sigma} \ar[r]^{\iota} & \S{\underline{\p}}^{\Sigma} \ar[r]^{\bar{\delta} \quad \quad}  & \bar{H}^{1}(W_{F}, \D{D})}.                
\end{align}

The local Langlands correspondence for tori gives us an isomorphism 
\[
H^{1}(W_{F}, \D{D}) \cong \Hom(D(F), \C^{\times}).
\]
By pulling back quasicharacters of $D(F)$ to $\lG(F)$, we get a homomorphism 
\[
H^{1}(W_{F}, \D{D}) \rightarrow \Hom(\lG(F)/G(F), \C^{\times}),
\]
which is surjective. Note $\delta(Z(\D{G})^{\Gal{}})$ is trivial in $H^{1}(W_{F}, Z(\D{\lG}))$, so it induces the trivial character on $\lG(F)$. Since \eqref{eq: local character} is an isomorphism, we then have an isomorphism
\[
r: \bar{H}^{1}(W_{F}, \D{D}) \rightarrow \Hom(\lG(F)/G(F), \C^{\times}).
\]
We denote the composition $r \circ \bar{\delta}$ by $\a$. Therefore we have the following exact sequence
\begin{align}% twisted endoscopic sequence with dual Eq
\label{eq: twisted endoscopic sequence with dual}
\xymatrix{1 \ar[r] &  \S{\underline{\lp}}^{\Sigma} \ar[r]^{\iota} & \S{\underline{\p}}^{\Sigma} \ar[r]^{\a \quad \quad \quad \quad}  & \Hom(\lG(F)/G(F), \C^{\times}).}               
\end{align}

\begin{lemma}% trivial central character LEMMA
\label{lemma: trivial central character}
The image $\a(\S{\underline{\p}})$ is contained in $\Hom(\lG(F)/\lZ(F)G(F), \C^{\times})$.
\end{lemma}

\begin{proof}
It follows from the proof of Lemma~\ref{lemma: centralizer} that the image $\delta(S_{\underline{\p}})$ is in the kernel of 
\(
H^{1}(W_{F}, \D{D}) \rightarrow H^{1}(W_{F}, \D{Z_{\lG}^{0}}).
\)
So $\a(\S{\underline{\p}})$ is contained in $\Hom(\lG(F)/\lZ^{0}(F)G(F), \C^{\times})$. When $\lZ = \lZ^{0}$, this is what we want.

Suppose $\lZ$ is not connected, we can take an $F$-torus $Z$ containing $\lZ$, and let $\lG' = (\lG \times Z)/\lZ$. Then $Z_{\lG'} = Z$ is connected. 
%Let $D' = \lG'/G$, then we have the following diagram
%\[
%\xymatrix{1 \ar[r] & G \ar@{=}[d] \ar[r] & \lG \ar[r] \ar[d]  & D \ar[r] \ar[d]& 1 \\ 
%1 \ar[r] & G \ar[r] & \lG' \ar[r]  & D' \ar[r] & 1.}
%\]
Let $\lp' \in \P{\lG'}$ be a lift of $\lp$, then we have the following commutative diagram:
\[
\xymatrix{1 \ar[r] \ar[d] &  \S{\underline{\lp'}} \ar[d] \ar[r]^{\iota'} & \S{\underline{\p}} \ar@{=}[d] \ar[r]^{\a' \quad \quad \quad \quad}  & \Hom(\lG'(F)/G(F), \C^{\times}) \ar[d] \\                
1 \ar[r] &  \S{\underline{\lp}} \ar[r]^{\iota} & \S{\underline{\p}} \ar[r]^{\a \quad \quad \quad \quad}  & \Hom(\lG(F)/G(F), \C^{\times}).}              
\]
Since the image of $\a'$ is trivial on $Z_{\lG'}(F)$, then the image of $\a$ is trivial on $\lZ(F)$. This finishes the proof.

\end{proof}

Suppose $\theta \in \Sigma$, for any semisimple element $s \in \com{\cS{\underline{\p}}}$, let $\D{H}  = \Cent(s, \D{G})^{0}$, and $\mathcal{H} = \D{H} \cdot \Im \underline{\p}$, then $\mathcal{H}$ is embedded identically in $\L{G}$. The conjugate action of $L_{F}$ on $\D{H}$ through $\underline{\p}$ determines a Galois action on $\D{H}$, and hence determines a quasisplit reductive group $H$. Therefore $\underline{\p}$ factors through $\mathcal{H}$ for $[(H, \mathcal{H}, s, \xi)] \in \End{}{\com{G}}$, where $\xi$ is the identity embedding. For any lift $\underline{\lp}$ of $\underline{\p}$, the restriction $\underline{\lp}|_{W_{F}}$ lifts $(H, \mathcal{H}, s, \xi)$ to a twisted endoscopic datum $[(\lif{H}, \mathcal{\lif{H}}, \lif{s}, \lif{\xi})] \in \End{}{\com{\lG}, \x}$ for some quasicharacter $\x$ of $\lG(F)/G(F)$ (cf. the proof of Proposition~\ref{prop: lifting endoscopic group}). By construction we know $\underline{\lp}$ factors through $\mathcal{\lif{H}}$. If we take a different lift $\underline{\lp}'$ of $\underline{\p}$, it is easy to see $\underline{\lp}'$ also factors through $\mathcal{\lif{H}}$. All of these can be summarized in the diagram below
\begin{displaymath}
\xymatrix{L_{F}    \ar[r]    \ar[dr]  & \mathcal{\lif{H}} \ar[r]^{\lif{\xi}}     \ar[d]   & \L{\lG}   \ar[d] \\
&    \mathcal{H}    \ar[r]^{\xi}     &    \L{G}.}
\end{displaymath}
Then we have the following simple fact.

\begin{lemma}%twisted character LEMMA
\label{lemma: twisted character}
$\a(s) = \x$.
\end{lemma}

\begin{proof}
By definition $\delta(s)(w) = \lif{s} \underline{\lp}(w) \lif{s} ^{-1}\underline{\lp}(w)^{-1}$ for any $w \in W_{F}$, and $\lif{s}$ is a preimage of $s$ in $\D{\lG}^{\Sigma}$. Since $\underline{\lp}$ factors through $\mathcal{\lif{H}}$ and $\D{\lif{H}}$ commutes with $\lif{s}$, we have $\lif{s} \underline{\lp}(w) \lif{s} ^{-1}\underline{\lp}(w)^{-1} = \lif{s} \lif{\xi}(w) \lif{s}^{-1} \lif{\xi}(w)^{-1}$, and this means $\a(s) = \x$.
\end{proof}

\subsection{Endoscopic transfer}% endoscopic transfer SUBSECTION
\label{subsec: endoscopic transfer}

Back to the setup in Section~\ref{subsec: endoscopic group}, the reason that endoscopic data are so important is because there is a transfer map from $C^{\infty}_{c}(G(F))$ to $C^{\infty}_{c}(H(F))$ if $\mathcal{H} = \L{H}$. If $\mathcal{H} \neq \L{H}$, we have to take an extension $H_{1}$ of $H$ by an induced torus $Z_{1}$ (called $z$-extension)
\[
\xymatrix{1 \ar[r] & Z_{1} \ar[r] & H_{1} \ar[r]  & H \ar[r] & 1}
\]
so that the dual homomorphism $\D{H} \rightarrow \D{H}_{1}$ can be extended to an L-embedding $\xi_{H_{1}}: \mathcal{H} \rightarrow \L{H_{1}}$. We call $(H_{1}, \xi_{H_{1}})$ a $z$-pair for $\mathcal{H}$. By choosing a section $c: W_{F} \rightarrow \mathcal{H}$, we get a quasicharacter $\chi_{1}$ on $Z_{1}$ from
\[
\xymatrix{ W_{F} \ar[r]^{c} & \mathcal{H} \ar[r]^{\xi_{H_{1}}} & \L{H_{1}} \ar[r] & \L{Z_{1}}.}
\]
It is easy to see that $\chi_{1}$ is independent of the choice of section $c$. So the transfer map will be from $C^{\infty}_{c}(G(F))$ to $C^{\infty}_{c}(H_{1}(F), \chi_{1})$, which is the space of $\chi_{1}^{-1}$-equivariant smooth functions on $H_{1}(F)$ with compact support modulo $Z_{1}$.

To define this transfer map, we need to introduce the space of twisted (stable) orbital integrals. Let $\x_{G}$ be a quasicharacter of $G(F)$ and $\delta$ be a strongly $\theta$-regular $\theta$-semisimple element of $G(F)$, namely $\Int(\delta) \circ \theta$ is semisimple and the $\theta$-twisted centralizer $\com{G}_{\delta}(F)$ (i.e., $\Int(\delta) \circ \theta$-invariant elements in $G(F)$) of $\delta$ is abelian. We assume $\x_{G}$ is trivial on $\com{G}_{\delta}(F)$. We fix Haar measures on $G(F)$ and $\com{G}_{\delta}(F)$, and they induce a $G(F)$-invariant measure on $\com{G}_{\delta}(F) \backslash G(F)$. Then we can form the $(\theta, \x_{G})$-twisted orbital integral of $f \in C^{\infty}_{c}(G(F))$ over $\delta$ as
\[
O^{\theta, \x_{G}}_{G}(f, \delta) := \int_{\com{G}_{\delta}(F) \backslash G(F)} \x_{G}(g)f(g^{-1}\delta \theta(g)) dg.
\]
We also form the $(\theta, \x_{G})$-twisted stable orbital integral over $\delta$ as
\[
SO^{\theta, \x_{G}}_{G}(f, \delta) := \sum_{\{\delta'\}_{G(F)}^{\theta} \thicksim_{st} \{\delta\}_{G(F)}^{\theta}} O_{G}^{\theta, \x_{G}}(f, \delta'), 
\]
where the sum is over $\theta$-twisted conjugacy classes $\{\delta'\}^{\theta}_{G(F)}$ in the $\theta$-twisted stable conjugacy class of $\delta$ (i.e., $\delta' = g^{-1} \delta \theta(g)$ for some  $g \in G(\bar{F})$), and the Haar measure on $\com{G}_{\delta'}(F)$ is translated from that on $\com{G}_{\delta}(F)$ by conjugation. Let $\mathcal{I}(G^{\theta, \x_{G}})$ ($\mathcal{SI}(G^{\theta, \x_{G}})$) be the space of $(\theta, \x_{G})$-twisted (stable) orbital integrals of $C^{\infty}_{c}(G(F))$ over the set $G^{\theta}_{reg}(F)$ of strongly $\theta$-regular $\theta$-semisimple elements of $G(F)$, then by definition we have projections 
\[
\xymatrix{C^{\infty}_{c}(G) \ar@{->>}[r] & \mathcal{I}(G^{\theta, \x_{G}}) \ar@{->>}[r] & \mathcal{SI}(G^{\theta, \x_{G}}).}
\]

Suppose $[(H, \mathcal{H}, s, \xi)] \in \End{}{\com{G}, \x_{G}}$, we fix a $z$-pair $(H_{1}, \xi_{H_{1}})$ for $\mathcal{H}$. We assume $\theta$ preserves an $F$-splitting of $G$, then there is a map from the semisimple $H_{1}(\bar{F})$-conjugacy classes of $H_{1}(\bar{F})$ to the $\theta$-twisted semisimple $G(\bar{F})$-conjugacy classes of $G(\bar{F})$. By our assumption on $\theta$, this map is defined over $F$. We call a strongly regular semisimple element $\gamma_{1} \in H_{1}(\bar{F})$ is strongly $G$-regular if its associated $H_{1}(\bar{F})$-conjugacy class maps to a $\theta$-twisted strongly regular semisimple $G(\bar{F})$-conjugacy class of $G(\bar{F})$. We denote the set of strongly $G$-regular semisimple elements of $H_{1}(F)$ by $H_{1, G-reg}(F)$. The transfer factor defined in \cite{KottwitzShelstad:1999} is a function
\[
\Delta_{G, H_{1}}(\cdot, \cdot): H_{1, G-reg}(F) \times G^{\theta}_{reg}(F) \rightarrow \C,
\]
which is nonzero only when $\gamma_{1} \in H_{1, G-reg}(F)$ is a norm of $\delta \in G^{\theta}_{reg}(F)$, i.e., the $H_{1}(\bar{F})$-conjugacy class of $\gamma_{1}$ maps to the $\theta$-twisted $G(\bar{F})$-conjugacy class of $\delta$. Note if $\delta \in G^{\theta}_{reg}(F)$ has a norm $\gamma_{1} \in H_{1, G-reg}(F)$, then $\x_{G}$ is trivial on $\com{G}_{\delta}(F)$ (see Lemma 4.4.C, \cite{KottwitzShelstad:1999}). In this paper we always normalize the transfer factor with respect to a fixed $\theta$-stable Whittaker datum $(B, \Lambda)$ for $G$.
The transfer factor has the following basic properties (see \cite{KottwitzShelstad:1999}):

\begin{itemize}

\item $\Delta_{G, H_{1}}(\cdot, \cdot)$ is invariant over stable conjugacy class in the first variable.

\item There is a canonical inclusion $(Z_{G})_{\theta} \hookrightarrow Z_{H}$, so that we get a homomorphism
\begin{align}% central inclusion Eq
\label{eq: central inclusion}
Z_{G} \rightarrow (Z_{G})_{\theta} \hookrightarrow Z_{H}.
\end{align}
Let $C$ be the fiber product of $\Z$ and $Z_{H_{1}}$ over $Z_{H}$. Then there exists a quasicharacter $\chi_{C}$ of $C(F)$ such that
\[
\Delta_{G, H_{1}}(z_{1}\gamma_{1}, z\delta) = \chi_{C}^{-1}(z_{1}, z) \Delta_{G, H_{1}}(\gamma_{1}, \delta), \,\,\,\,\, z_{1} \in Z_{H_{1}}(F), z \in Z_{G}(F)
\]
where $z_{1}$ and $z$ have the same image on $Z_{H}(F)$, and the restriction of $\chi_{C}$ to $Z_{1}(F)$ is $\chi_{1}$.

\item For $g \in G(F)$, $\Delta_{G, H_{1}}(\gamma_{1}, g^{-1} \delta \theta (g)) = \x_{G}(g)\Delta_{G, H_{1}}(\gamma_{1}, \delta)$.

\end{itemize}

The transfer map is a correspondence from $f \in C^{\infty}_{c}(G(F))$ to $f^{H_{1}} \in C^{\infty}_{c}(H_{1}(F), \chi_{1})$ such that
\begin{align}% geometric transfer Eq
\label{eq: geometric transfer}
SO_{H_{1}}(f^{H_{1}}, \gamma_{1}) = \sum_{\{\delta'\}_{G(F)}^{\theta} \thicksim_{st} \{\delta\}_{G(F)}^{\theta}} \Delta_{G, H_{1}}(\gamma_{1}, \delta') O_{G}^{\theta, \x_{G}}(f, \delta')
\end{align}
where the sum is over $\theta$-twisted conjugacy classes $\{\delta'\}_{G(F)}^{\theta}$ in the $\theta$-twisted stable conjugacy class of $\delta$. The existence of such a correspondence has been conjectured by Langlands, Shelstad and Kottwitz. In the real case, this is now a theorem of Shelstad \cite{Shelstad:2012}. In the $p$-adic case, Waldspurger \cite{Waldspurger:1995} \cite{Waldspurger:1997} \cite{Waldspurger:2006} \cite{Waldspurger:2008} reduced it to the fundamental lemma for Lie algebras over the function field, and Ngo \cite{Ngo:2010} proved the fundamental lemma in this form. 

%a longstanding problem, which had been brought to an end by Ngo's proof of the celebrated ``Fundamental Lemma".

This transfer map can also be defined for equivariant functions. To be more precise, let $Z_{F}$ be a closed subgroup of $Z_{G}(F)$ such that $Z_{F} \rightarrow Z_{H}(F)$ through \eqref{eq: central inclusion} is injective. In particular, the preimage of $Z_{F}$ in $Z_{H_{1}}(F)$ forms a closed subgroup of $C(F)$, we denote it by $C_{1, F}$. We fix a quasicharacter $\chi$ on $Z_{F}$, it pulls back to a quasicharacter on $C_{1, F}$. Denote the restriction of $\chi_{C}$ on $C_{1, F}$ by $\chi_{C_{1}}$, then we claim there is a correspondence from the space $C^{\infty}_{c}(G(F), \chi)$ of $\chi^{-1}$-equivariant functions to the space $C^{\infty}_{c}(H_{1}(F), \chi \chi_{C_{1}})$ of $( \chi \chi_{C_{1}})^{-1}$-equivariant functions characterized by \eqref{eq: geometric transfer}. This correspondence can be constructed as follows. There is a surjection from $C^{\infty}_{c}(G(F))$ to $C^{\infty}_{c}(G(F), \chi)$ defined by 
\[
f \mapsto \bar{f} = \int_{Z_{F}} f(zg)\chi(z) dz.
\]
Similarly, we have a surjection from $C^{\infty}_{c}(H_{1}(F), \chi_{1})$ to $C^{\infty}_{c}(H_{1}(F), \chi\chi_{C_{1}})$ defined by
\[
f \mapsto \bar{f} = \int_{Z_{1}(F) \backslash C_{1, F}} f(zg)\chi\chi_{C_{1}}(z) dz.
\]
Then it suffices to check the commutativity of the following diagram
\[
\xymatrix{C^{\infty}_{c}(G(F)) \ar[r] \ar[d] & C^{\infty}_{c}(H_{1}(F), \chi_{1}) \ar[d] \\
C^{\infty}_{c}(G(F), \chi) \ar[r] & C^{\infty}_{c}(H_{1}(F), \chi\chi_{C_{1}}).}
\]
Suppose $f \in C^{\infty}_{c}(G(F))$, $\gamma_{1}$ is a norm of $\delta$
\[
O_{G}^{\theta, \x_{G}}(\bar{f}, \delta) = \int_{Z_{F}} O_{G}^{\theta, \x_{G}}(f, z \delta) \chi(z) dz.
\]
So
\begin{align*}
& \sum_{\{\delta'\}_{G(F)}^{\theta} \thicksim_{st} \{\delta\}_{G(F)}^{\theta}} \Delta_{G, H_{1}}(\gamma_{1}, \delta') O_{G}^{\theta, \x_{G}}(\bar{f}, \delta') \\
= & \sum_{\{\delta'\}_{G(F)}^{\theta} \thicksim_{st} \{\delta\}_{G(F)}^{\theta}} \Delta_{G, H_{1}}(\gamma_{1}, \delta') \int_{Z_{F}} O_{G}^{\theta, \x_{G}}(f, z \delta') \chi(z) dz \\
= & \int_{Z_{F}} \chi(z) \sum_{\{\delta'\}_{G(F)}^{\theta} \thicksim_{st} \{\delta\}_{G(F)}^{\theta}} \Delta_{G, H_{1}}(\gamma_{1}, \delta') O_{G}^{\theta, \x_{G}}(f, z \delta') dz \\
= & \int_{Z_{F}} \chi(z) \sum_{\{\delta'\}_{G(F)}^{\theta} \thicksim_{st} \{\delta\}_{G(F)}^{\theta}} \chi_{C}(z_{1}, z) \Delta_{G, H_{1}}(z_{1}\gamma_{1}, z\delta') O_{G}^{\theta, \x_{G}}(f, z \delta')  dz \\
= & \int_{Z_{F}} \chi(z)\chi_{C}(z_{1}, z) SO_{H_{1}}(f^{H_{1}}, z_{1}\gamma_{1})dz \\
= & \int_{Z_{1}(F) \backslash C_{1, F}} \chi\chi_{C_{1}}(z_{1}) SO_{H_{1}}(f^{H_{1}}, z_{1}\gamma_{1})dz \\
\end{align*}
Hence $\bar{f}$ corresponds to the image of $f^{H_{1}}$ in $C^{\infty}_{c}(H_{1}(F), \chi\chi_{C_{1}})$.

Let $G \subseteq \lG$ be two quasisplit connected reductive groups over $F$ such that $G_{der} = \lG_{der}$ and $\lG/G = D$. Suppose $[(H, \mathcal{H}, s, \xi)] \in \End{}{\com{G}, \x_{G}}$, let $[(\lif{H}, \mathcal{\lif{H}}, \lif{s}, \lif{\xi})] \in \End{}{\com{\lG}, \x_{\lG}}$ be the corresponding lift. We also fix a $z$-pair $(\lif{H}_{1}, \lif{\xi}_{\lif{H}_{1}})$ for $\mathcal{\lif{H}}$ with a $z$-extension
\[
\xymatrix{1 \ar[r] & Z_{1} \ar[r] & \lif{H}_{1} \ar[r]  & \lif{H} \ar[r] & 1.}
\]
Let $H_{1}$ be the preimage of $H$ in $\lif{H}_{1}$, then we get a $z$-extension for $H$
\[
\xymatrix{1 \ar[r] & Z_{1} \ar[r] & H_{1} \ar[r]  & H \ar[r] & 1.}
\]
Note $\lif{H}_{1} / H_{1} = D$, so on the dual side $\lif{\xi}_{\lif{H}_{1}}$ gives rise to an L-embedding $\xi_{H_{1}}: \mathcal{H} \rightarrow \L{H_{1}}$ by taking quotient of $\D{D}$. We fix a $\theta$-stable Whittaker datum for $\lG$, which determines that for $G$. Then the relation of the transfer factors $\Delta_{\lG, \lif{H}_{1}}$ and $\Delta_{G, H_{1}}$ can be stated in the following lemma.

\begin{lemma}%twisted transfer factor
\label{lemma: twisted transfer factor}
Suppose $\delta$ is a strongly $\theta$-regular $\theta$-semisimple element in $G(F) \subseteq \lG(F)$, and $\gamma_{1}$ is a strongly $G$-regular semisimple element in $H_{1}(F) \subseteq \lif{H}_{1}(F)$. Then one has 
\[
\Delta_{\lG, \lif{H}_{1}}(\gamma_{1}, \delta) = \Delta_{G, H_{1}}(\gamma_{1}, \delta).
\]
\end{lemma} 

\begin{proof}
Suppose the $\theta$-stable Whittaker datum for $\lG$ is constructed with respect to a $\theta$-stable $F$-splitting $(\mathbb{\lif{B}}, \mathbb{\lif{T}}, \{X_{\alpha}\})$ of $\lG$, and a nontrivial additive character $\q_{F}$ of $F$. This also determines a $\theta$-stable $F$-splitting $(\mathbb{B}, \mathbb{T}, \{X_{\alpha}\})$ of $G$. Then the unnormalized transfer factor can be defined as a product
\[
\Delta_{0}(\gamma_{1}, \delta) = \Delta_{I}(\gamma_{1}, \delta) \Delta_{II}(\gamma_{1}, \delta)\Delta_{III}(\gamma_{1}, \delta)\Delta_{IV}(\gamma_{1}, \delta).
\]
It depends on the $\theta$-stable $F$-splitting that we have fixed. First, we would like to compare the unnormalized transfer factors for $(\lG, \lif{H}_{1})$ and $(G, H_{1})$ term by term. To set things up, let $\lif{T}_{H_{1}}$ be the centralizer of $\gamma_{1}$ in $\lif{H}_{1}$ and let  $T_{H_{1}} = \lif{T}_{H_{1}} \cap H_{1}$. Let $\lif{T}_{H}$ (resp. $T_{H}$) be the projection of $\lif{T}_{H_{1}}$ (resp. $T_{H_{1}}$) on $\lif{H}$ (resp. $H$). We fix an admissible embedding $\lif{T}_{H} \longrightarrow \lif{T}_{\theta}$ and this gives an admissible embedding $T_{H} \longrightarrow T_{\theta}$ by restriction. Since the root system $R(\lG, \lif{T})$ is isomorphic to $R(G, T)$ and the isomorphism is equivariant under the Galois action, one can assign the same $a$-data and $\chi$-data \cite{LanglandsShelstad:1987} to them. They induce $a$-data and $\chi$-data for $R_{res}(\lG, \lif{T})$ (resp. $R_{res}(G, T)$) which are roots in $R(\lG, \lif{T})$ (resp. $R(G, T)$) restricted to $(\lif{T}^{\theta})^{0}$ (resp. $(T^{\theta})^{0}$). 

Let $<\cdot, \cdot>$ denote the Tate-Nakayama pairing between $H^{1}(F, T_{sc}^{\theta})$ and $\pi_{0}((\D{T_{sc}^{\theta}})^{\Gal{}})$, then the first term in the unnormalized transfer factor is defined by
\[
\Delta_{I, (G, H_{1})}(\gamma_{1}, \delta) = <\lambda_{a_{\a}}(T_{sc}^{\theta}), s_{T, \theta}>
\]
where $\lambda_{a_{\a}}(T_{sc}^{\theta})$ is defined by using $a$-data and the $\theta$-stable $F$-splitting, and $s_{T, \theta}$ is the projection of the semisimple element $s \in \D{G} \rtimes \D{\theta}$ in the endoscopic datum $(H, \mathcal{H}, s, \xi)$ onto $(\D{T}_{ad})_{\D{\theta}} = \D{T_{sc}^{\theta}}$.  Because $\lG_{sc} = G_{sc}$ and we choose $a$-data and the $\theta$-stable $F$-splitting for $\lG$ and $G$ in a consistent way,  $\lambda_{a_{\a}}(T_{sc}^{\theta}) = \lambda_{a_{\a}}(\lif{T}_{sc}^{\theta})$. Moreover $\lif{s}$ and $s$ have the same image in $\D{T_{sc}^{\theta}}$, hence 
\[
\Delta_{I, (G, H_{1})}(\gamma_{1}, \delta) = \Delta_{I, (\lG, \lif{H}_{1})}(\gamma_{1}, \delta).
\]

For the second term we adopt Waldspurger's modification here (see \cite{KottwitzShelstad:2012}). It is defined by the $a$-data and $\chi$-data, and again because we choose them for $\lG$ and $G$ in a consistent way, the second term will be the same for $(\lG, \lif{H}_{1})$ and $(G, H_{1})$. Before discussing the third term, let us consider the fourth term first. The fourth term is defined by 
\[
\Delta_{IV, (G, H_{1})} (\gamma_{1}, \delta)= \frac{D_{\com{G}} (\delta) }{D_{H_{1}}(\gamma_{1}) }
\]
where 
\begin{align*}
D_{\com{G}}(\delta) & = |det (Ad(\delta) \circ \theta - 1)_{Lie(G)/ Lie(\Cent((G_{\delta}^{\theta})^{0}, G)}|^{1/2}_{F}  \\
D_{H_{1}} (\gamma_{1}) & = |det(Ad(\gamma_{1}) - 1)_{Lie(H_{1})/ Lie(T_{H_{1}})}|^{1/2}_{F}
\end{align*}
And it is easy to see that $D_{\com{\lG}}(\delta) = D_{\com{G}}(\delta)$ and $D_{\lif{H}_{1}}(\gamma_{1}) = D_{H_{1}}(\gamma_{1})$, therefore the fourth term is also the same for $(\lG, \lif{H}_{1})$ and $(G, H_{1})$.

We are now left with the third term $\Delta_{III, (G,H_{1})}$, and it is given by a pairing of  hypercohomology groups $H^{1}(F, \xymatrix{T_{sc} \ar[r]^{1-\theta_{1}} & T_{1}})$ and $H^{1}(W_{F}, \xymatrix{\D{T}_{1} \ar[r]^{1-\D{\theta}_{1}} & \D{T}_{ad} })$, where $T_{1}$ is the fiber product of $T$ and $T_{H_{1}}$ over $T_{\theta} \cong T_{H}$, and $\theta_{1}$ is a lift of $\theta$ on $T_{1}$ which fixes $Z_{1} \subseteq T_{1}$. Similarly we can define $\lif{T}_{1}$, and the inclusion $T_{1} \rightarrow \lif{T}_{1}$ induces maps on hypercohomology groups

\begin{align*}
\varphi  &: H^{1}(F, \xymatrix{T_{sc} \ar[r]^{1-\theta_{1}} & T_{1}})       \longrightarrow     H^{1}(F, \xymatrix{\lif{T}_{sc} \ar[r]^{1-\theta_{1}} & \lif{T}_{1}}) \\
\varphi^{*}  &:   H^{1}(W_{F}, \xymatrix{\D{\lif{T}}_{1} \ar[r]^{1-\D{\theta}_{1}} & \D{\lif{T}}_{ad} }) \longrightarrow H^{1}(W_{F}, \xymatrix{\D{T}_{1} \ar[r]^{1-\D{\theta}_{1}} & \D{T}_{ad} }). 
\end{align*}
It is an easy exercise to check that they are adjoint to each other with respect to the Tate-Nakayama pairing on hypercohomology groups, i.e.
\[
<\varphi(\textbf{V}), \textbf{A}> = <\textbf{V}, \varphi^{*}(\textbf{A})> 
\]
where $\textbf{V} \in H^{1}(F, \xymatrix{T_{sc} \ar[r]^{1-\theta_{1}} & T_{1}})$ and $\textbf{A} \in H^{1}(W_{F}, \xymatrix{\D{\lif{T}}_{1} \ar[r]^{1-\D{\theta}_{1}} & \D{\lif{T}}_{ad} })$. It follows from the definition in \cite{KottwitzShelstad:1999} that there exist $\textbf{V}_{0} \in H^{1}(F, \xymatrix{T_{sc} \ar[r]^{1-\theta_{1}} & T_{1}})$ and $\textbf{A}_{0} \in H^{1}(W_{F}, \xymatrix{\D{\lif{T}}_{1} \ar[r]^{1-\D{\theta}_{1}} & \D{\lif{T}}_{ad} })$ such that 
\begin{align*}
\Delta_{III, (\lG, \lif{H}_{1})}(\gamma_{1}, \delta) &= <\varphi(\textbf{V}_{0}), \textbf{A}_{0}>, \\
\Delta_{III, (G,H_{1})}(\gamma_{1}, \delta) &= <\textbf{V}_{0}, \varphi^{*}(\textbf{A}_{0})>.
\end{align*}
Hence, $\Delta_{III, (\lG, \lif{H}_{1})}(\gamma_{1}, \delta) = \Delta_{III, (G,H_{1})}(\gamma_{1}, \delta)$. 

Up to now, we have shown the equality for the unnormalized transfer factors. To define the normalizing factor, we need to choose an $F$-splitting $(\mathbb{\lif{B}}_{H}, \mathbb{\lif{T}}_{H}, \{X_{\alpha_{H}}\})$ of $\lif{H}$ and it determines an $F$-splitting $(\mathbb{B}_{H}, \mathbb{T}_{H}, \{X_{\alpha_{H}}\})$ of $H$. Let $V_{\lG}$ be the representation of $\Gal{}$ on $X^{*}(\mathbb{\lif{T}})^{\theta} \otimes \C$ and $V_{\lif{H}}$ be representation of $\Gal{}$ on $X^{*}(\mathbb{\lif{T}}_{H}) \otimes \C$. Let $\lif{V} = V_{\lG} - V_{\lif{H}}$, then the normalizing factor for $(\lG, \lif{H}_{1})$ is given by the local $\epsilon$-factor 
\[
\epsilon_{L}(\lif{V}, \q_{F}) = \epsilon_{L}(V_{\lG}, \q_{F}) / \epsilon_{L}(V_{\lif{H}}, \q_{F})
\] 
(see \cite{Tate:1979}, 3.6). Similarly, we can define $V_{G}, V_{H}$ and $V = V_{G} - V_{H}$. Then it is enough to show $\epsilon_{L}(\lif{V}, \q_{F}) = \epsilon_{L}(V, \q_{F})$. Note
\[
\epsilon_{L}(\lif{V}, \q_{F})/\epsilon(V, \q_{F}) = \epsilon_{L}(V_{\lG}, \q_{F}) / \epsilon_{L}(V_{G}, \q_{F}) \cdot \epsilon_{L}(V_{H}, \q_{F}) / \epsilon_{L}(V_{\lif{H}}, \q_{F}).
\]
By the following exact sequences
\[
\xymatrix{1 \ar[r] & \mathbb{T} \ar[r] & \mathbb{\lif{T}} \ar[r]  & D \ar[r] & 1}
\]
\[
\xymatrix{1 \ar[r] & \mathbb{T}_{H} \ar[r] & \mathbb{\lif{T}}_{H} \ar[r]  & D \ar[r] & 1}
\]
we have 
\[
\xymatrix{1 \ar[r] & V_{D} \ar[r] & V_{\lG} \ar[r]  & V_{G} \ar[r] & 1}
\]
\[
\xymatrix{1 \ar[r] & V_{D} \ar[r] & V_{\lif{H}} \ar[r]  & V_{H} \ar[r] & 1}
\]
where $V_{D} = X^{*}(D) \otimes \C$ is $\theta$-invariant. Therefore,
\[
\epsilon_{L}(V_{\lG}, \q_{F}) / \epsilon_{L}(V_{G}, \q_{F}) = \epsilon_{L}(V_{\lif{H}}, \q_{F}) / \epsilon_{L}(V_{H}, \q_{F}) = \epsilon_{L}(V_{D}, \q_{F}),
\]
This finishes the proof.

\end{proof}

Following the notations in this lemma, note $G(F)$ is $\theta$-twisted conjugate invariant under $\lG(F)$, so we have the following corollary.

\begin{corollary}%
Suppose $\delta$ is a strongly $\theta$-regular $\theta$-semisimple element in $G(F)$, and $\gamma_{1}$ is a strongly $G$-regular semisimple element in $H_{1}(F)$. Then one has 
\[
\Delta_{G, H_{1}}(\gamma_{1}, g^{-1} \delta \theta(g)) = \x_{\lG}(g) \Delta_{G, H_{1}}(\gamma_{1}, \delta).
\]
\end{corollary}

\begin{proof}
From the previous lemma we know 
\[
\Delta_{G, H_{1}}(\gamma_{1}, \delta) = \Delta_{\lG, \lif{H}_{1}}(\gamma_{1}, \delta) \text { and } \Delta_{\lG, \lif{H}_{1}}(\gamma_{1}, g^{-1} \delta \theta(g)) = \Delta_{G, H_{1}}(\gamma_{1}, g^{-1} \delta \theta(g)).
\]
It follows from the property of transfer factor that 
\[
\Delta_{\lG, \lif{H}_{1}}(\gamma_{1}, g^{-1} \delta \theta(g)) = \x_{\lG}(g) \Delta_{\lG, \lif{H}_{1}}(\gamma_{1}, \delta).
\]
Then the corollary is clear.

\end{proof}

\begin{remark}
An equivalent way of stating this corollary is as follows. Let $f \in C^{\infty}_{c}(G(F) \rtimes \theta)$, we can view it as in $C^{\infty}_{c}(G(F))$ by sending $g$ to $g \rtimes \theta$, and define its transfer as before. For $g \in \lG(F), h \in G(F) \rtimes \theta$ we denote $f^{g} (h) = f(ghg^{-1})$. Then this corollary says
\begin{align}% conjugate function Eq
\label{eq: conjugate function}
(f^{g})^{H_{1}} = \x_{\lG}(g)f^{H_{1}}.
\end{align}
\end{remark}

Let $\lif{Z}_{F}$ be a closed subgroup of $Z_{\lG}(F)$ such that $\lif{Z}_{F} \rightarrow (Z_{\lG})_{\theta}(F)$ is injective and $D(F) / \c(\lif{Z}_{F})$ is finite (this is possible because we assume $\c$ is $\theta$-invariant). Let $Z_{F} = \lif{Z}_{F} \cap G(F)$. We choose Haar measures on $\lif{Z}_{F}$ and $Z_{F}$ such that the measure on $Z_{F} \backslash G(F)$ is the restriction of that on $\lif{Z}_{F} \backslash \lG(F)$. We fix a quasicharacter $\lif{\chi}$ of $\lif{Z}_{F}$ and denote its restriction to $Z_{F}$ by $\chi$. Note $\lif{Z}_{F}G(F) \backslash \lG(F)$ is finite, so we get an inclusion map
\begin{align}% inclusion of Hecke algebra MAP
\label{map: inclusion of Hecke algebra}
\xymatrix{C^{\infty}_{c}(G(F), \chi) \ar@{^{(}->}[r]   &  C^{\infty}_{c}(\lG(F), \lif{\chi})  \\
f \ar@{|->}[r]   & \lif{f}, }
\end{align}
where $\lif{f}$ is the extension of $f$ by zero outside  $\lif{Z}_{F}G(F)$. We can identify $C^{\infty}_{c}(G(F), \chi)$ with its image in $C^{\infty}_{c}(\lG(F), \lif{\chi})$. Because $\lif{Z}_{F}G(F)$ is $\theta$-twisted conjugate invariant under $\lG(F)$, the map \eqref{map: inclusion of Hecke algebra} induces a map from $\mathcal{I}(G^{\theta, \x_{G}}, \chi)$ to $\mathcal{I}(\lG^{\theta, \x_{\lG}}, \lif{\chi})$ where $\x_{\lG}|_{G} = \x_{G}$. Moreover $\lif{Z}_{F}G(\bar{F})$ is  $\theta$-twisted conjugate invariant under $\lG(\bar{F})$, so it also induces a map from $\mathcal{SI}(G^{\theta, \x_{G}}, \chi)$ to $\mathcal{SI}(\lG^{\theta, \x_{\lG}}, \lif{\chi})$. 

Let $\x_{\lG}$ be a quasicharacter of $\lG(F)$ and $\x_{G} = \x_{\lG} | _{G}$. Let $\delta$ be a strongly $\theta$-regular $\theta$-semisimple element of $G(F) \subseteq \lG(F)$ such that $\x_{\lG}$ is trivial on the $\theta$-twisted centralizer $\com{\lG}_{\delta}(F)$ of $\delta$. We choose Haar measures on $\com{\lG}_{\delta}(F)$ and $\com{G}_{\delta}(F)$ such that the measure on $\com{G}_{\delta}(F) \backslash G(F)$ is the restriction of that on $\com{\lG}_{\delta}(F) \backslash \lG(F)$. 
Then
\[
O_{\lG}^{\theta, \x_{\lG}}(\lf, \delta) = \sum_{\{\delta'\}_{G(F)}^{\theta} \thicksim_{\lG(F)} \{\delta\}_{G(F)}^{\theta}} O_{G}^{\theta, \x_{G}}(f, \delta')\x_{\lG}(g)
\]
where the sum is over $\theta$-twisted $G(F)$-conjugacy classes $\{\delta'\}_{G(F)}^{\theta}$ in the $\theta$-twisted $\lG(F)$-conjugacy classes $\{\delta\}_{\lG(F)}^{\theta}$ with $\delta' = g^{-1} \delta g$ for $g \in \lG(F)$, and the Haar measure on $G^{\theta}_{\delta'}(F)$ is translated from that on $G^{\theta}_{\delta}$ by conjugation.

Suppose $[(\lif{H}, \mathcal{\lif{H}}, \lif{s}, \lif{\xi})] \in \End{}{\com{\lG}, \x_{\lG}}$ and $[(H, \mathcal{H}, s, \xi)] \in \End{}{\com{G}, \x_{G}}$ correspond to each other according to Proposition~\ref{prop: lifting endoscopic group}. We also fix a $z$-pair $(\lif{H}_{1}, \lif{\xi}_{\lif{H}_{1}})$ for $\mathcal{\lif{H}}$ which induces a $z$-pair $(H_{1}, \xi_{H_{1}})$ for $\mathcal{H}$. Let $\lif{C}_{1, F}$ be the preimage of $\lif{Z}_{F}$ in $Z_{\lif{H}_{1}}(F)$ and $C_{1, F}$ be the preimage of $Z_{F}$ in $Z_{H_{1}}(F)$. Then 
\[
C_{1, F} \hookrightarrow \lif{C}_{1, F} \xrightarrow{\c_{1}} D(F)
\]
%\[
%\xymatrix{1 \ar[r] & C_{1, F} \ar[r] & \lif{C}_{1, F} \ar[r]  & D(F) \ar[r] & 1.}
%\]
with $\c_{1}(\lif{C}_{1, F}) = \c(\lif{Z}_{F})$. It is easy to check that the restriction of $\chi_{\lif{C}}$ to $C(F)$ is $\chi_{C}$. Note $\lif{\chi}$ and $\chi$ pull back to quasicharacters of $\lif{C}_{1, F}$ and $C_{1, F}$ respectively. So let $\lif{\chi}' = \lif{\chi}\chi_{\lif{C}_{1}}$ and $\chi' = \chi\chi_{C_{1}}$, then we have an inclusion map analogous to \eqref{map: inclusion of Hecke algebra}
\begin{align*}
\xymatrix{C^{\infty}_{c}(H_{1}(F), \chi') \ar@{^{(}->}[r]   &  C^{\infty}_{c}(\lif{H}_{1}(F), \lif{\chi}')  \\
f \ar@{|->}[r]   & \lif{f} },  
\end{align*}
The next lemma shows these inclusion maps are compatible with twisted endoscopic transfers.

\begin{lemma}%twisted endoscopic transfer LEMMA
\label{lemma: twisted endoscopic transfer}
Suppose $f \in C^{\infty}_{c}(G(F), \chi)$, then the $(\theta, \x_{\lG})$-twisted endoscopic transfer of the extension $\lf$ of $f$ is equal to the extension of $(\theta, \x_{G})$-twisted endoscopic transfer $f^{H_{1}}$ of $f$ as elements in $\mathcal{SI}({\lif{H}_{1}}, \lif{\chi}')$, i.e.
\begin{align*}
\lf^{\lif{H}_{1}} = \lif{(f^{H_{1}})} 
\end{align*}
\end{lemma}

\begin{proof}
Let us assume $\delta$ is a strongly $\theta$-regular $\theta$-semisimple element in $G(F) \subseteq \lG(F)$ and $\gamma_{1}$ is a strongly $G$-regular semisimple element in $H_{1}(F) \subseteq \lif{H}_{1}(F)$, and $\gamma_{1}$ is a norm of $\delta$. By the definition of twisted endoscopic transfer
\[
SO_{\lif{H}_{1}}(\lf^{\lif{H}_{1}}, \gamma_{1}) = \sum_{\{\delta'\}_{\lG(F)}^{\theta} \thicksim_{st} \{\delta\}_{\lG(F)}^{\theta}} \Delta_{\lG, \lif{H}_{1}}(\gamma_{1}, \delta') O_{\lG}^{\theta, \x_{\lG}}(\lf, \delta')
\]
where the sum is over $\theta$-twisted $\lG(F)$-conjugacy classes $\{\delta'\}_{\lG(F)}^{\theta}$ in the $\theta$-twisted stable conjugacy class of $\delta$. Meanwhile, 
\[
O_{\lG}^{\theta, \x_{\lG}}(\lf, \delta') = \sum_{\{\delta''\}_{G(F)}^{\theta} \thicksim_{\lG(F)} \{\delta'\}_{G(F)}^{\theta}} O_{G}^{\theta, \x_{G}}(f, \delta'')\x_{\lG}(g)
\]
where the sum is over $\theta$-twisted $G(F)$-conjugacy classes $\{\delta''\}_{G(F)}^{\theta}$ in the $\theta$-twisted $\lG(F)$-conjugacy class $\{\delta'\}_{\lG(F)}^{\theta}$, and $\delta'' = g^{-1} \delta' \theta(g)$ for $g \in \lG(F)$. By the property of twisted transfer factor, one has 
\[
\Delta_{\lG, \lif{H}_{1}}(\gamma_{1}, g^{-1} \delta' \theta(g)) = \Delta_{\lG, \lif{H}_{1}}(\gamma_{1}, \delta') \x_{\lG}(g). 
\]
Therefore
\begin{align*}
SO_{\lif{H}_{1}}(\lf^{\lif{H}_{1}}, \gamma_{1}) & = \sum_{\{\delta'\}_{\lG(F)}^{\theta} \thicksim_{st} \{\delta\}_{\lG(F)}^{\theta} } \Delta_{\lG, \lif{H}_{1}}(\gamma_{1}, \delta') (\sum_{\{\delta''\}_{G(F)}^{\theta} \thicksim_{\lG(F)} \{\delta'\}_{G(F)}^{\theta}} O_{G}^{\theta, \x_{G}}(f, \delta'')\x_{\lG}(g) \,\,\,\,\,\,\, ) \\
& = \sum_{\{\delta''\}_{G(F)}^{\theta} \thicksim_{st} \{\delta\}_{G(F)}^{\theta}}  \Delta_{\lG, \lif{H}_{1}}(\gamma_{1}, \delta'') O_{G}^{\theta, \x_{G}}(f, \delta'').
\end{align*}
On the other hand 
\[
SO_{\lif{H}_{1}}(\lif{f^{H_{1}}}, \gamma_{1}) =  SO_{H_{1}}(f^{H_{1}}, \gamma_{1}) = \sum_{\{\delta''\}_{G(F)}^{\theta} \thicksim_{st} \{\delta\}_{G(F)}^{\theta}}  \Delta_{G, H_{1}}(\gamma_{1}, \delta'') O_{G}^{\theta, \x_{G}}(f, \delta'').
\]
It follows from Lemma~\ref{lemma: twisted transfer factor} that
\[
\Delta_{\lG, \lif{H}_{1}}(\gamma_{1}, \delta'') = \Delta_{G, H_{1}}(\gamma_{1}, \delta''),
\]
where $\delta''$ is in the $\theta$-twisted stable $G(F)$-conjugacy class of $\delta$. So we have shown
\begin{align}% twisted endoscopic transfer Eq
\label{eq: twisted endoscopic transfer}
SO_{\lif{H}_{1}}(\lf^{\lif{H}_{1}}, \gamma_{1}) = SO_{\lif{H}_{1}}(\lif{f^{H_{1}}}, \gamma_{1})
\end{align}
for $\gamma_{1} \in H_{1}(F)$ being a norm. 

If $\gamma_{1}$ is not a norm, it follows from the property of transfer factor that both sides of \eqref{eq: twisted endoscopic transfer} are zero. By equivariance property, we can extend \eqref{eq: twisted endoscopic transfer} to $\lif{C}_{1, F}H_{1}(F)$. It is also easy to see that $SO_{\lif{H}_{1}}(\lif{f^{H_{1}}}, \gamma_{1}) \neq 0$ only when $\gamma_{1} \in \lif{C}_{1, F}H_{1}(F)$. Finally, one can show $SO_{\lif{H}_{1}}(\lf^{\lif{H}_{1}}, \gamma_{1}) \neq 0$ only when $\gamma_{1} \in \lif{C}_{1, F}H_{1}(F)$ by using the condition on the support of $\lf$. This finishes the proof.

%Let us take any $\gamma_{1} \in \lif{H}_{1, \lG-reg}(F)$ and assume it is a norm of $\delta \in \lG^{\theta}_{reg}(F)$. Let $\lif{T}_{H_{1}}$ be the centralizer of $\gamma_{1}$ in $\lif{H}_{1}$ and let  $T_{H_{1}} = \lif{T}_{H_{1}} \cap H_{1}$. Let $\lif{T}_{H}$ (resp. $T_{H}$) be the projection of $\lif{T}_{H_{1}}$ (resp. $T_{H_{1}}$) on $\lif{H}$ (resp. $H$). We fix an admissible embedding $\lif{T}_{H} \longrightarrow \lif{T}_{\theta}$, where $\lif{T}$ is a $\theta$-stable maximal torus of $\lG$. This induces an admissible embedding $T_{H} \rightarrow T_{\theta}$ by restriction, where $T = \lif{T} \cap G$. By definition (see \cite{KottwitzShelstad:1999}, Section 3), there exists $g \in \lG(\bar{F})$ such that $\delta' = g^{-1}\delta \theta(g) \in \lif{T}(\bar{F})$ and $\gamma_{1}$ maps to the image of $\delta'$ in $\lif{T}_{\theta}$ through the composition $\lif{T}_{H_{1}} \rightarrow \lif{T}_{H}$ and $\lif{T}_{H} \longrightarrow \lif{T}_{\theta}$. Suppose $SO_{\lif{H}_{1}}(\lf^{\lif{H}_{1}}, \gamma_{1}) \neq 0$, then necessarily $\delta \in \lif{Z}_{F}G(F)$ by condition on the support of $\lf$. Let us write $\delta = z \delta_{0}$ for $z \in \lif{Z}_{F}, \delta_{0} \in G(F)$, and $\delta' = z g^{-1}\delta_{0} \theta(g) = z \delta'_{0}$. We can choose any element $z_{1}$ in the preimage of $z$ in $\lif{C}_{1, F}$, then $z_{1}^{-1}\gamma_{1}$ maps to the image of $\delta'_{0}$ in $T_{\theta}$. Therefore, $z_{1}^{-1}\gamma_{1} \in T_{H_{1}}$.

\end{proof}

\subsection{Character identity}% character identity SUBSECTION
\label{subsec: character identity}

Let $\r$ be an irreducible smooth representation of $G(F)$ and $\chi_{\r}$ be the central character of $\r$. Let $C_{F}$ be a closed subgroup of $Z_{G}(F)$, and $\zeta = \chi_{\r}|_{Z_{F}}$. Suppose $\r^{\theta} \cong \r \otimes \x_{G}$, let $A_{\r}(\theta, \x_{G})$ be an intertwining operator between $\r \otimes \x_{G}$ and $\r^{\theta}$ (this is uniquely determined up to a scalar), we then define the $(\theta, \x_{G})$-twisted character of $\r$ to be the distribution 
\begin{align}% twisted character Eq
\label{eq: twisted character}
f_{\com{G}}(\r, \x_{G}) :=  trace \int_{C_{F} \backslash G(F)} f(g)\r(g) dg \circ A_{\r}(\theta, \x_{G}),
\end{align}
for $f \in C^{\infty}_{c}(G(F), \zeta)$. In particular, we can define the distribution for $f \in C^{\infty}_{c}(G(F))$ by taking $C_{F}$ to be trivial. By results of Harish-Chandra \cite{H-C:1963} \cite{H-C:1999} in the non-twisted case, Bouaziz \cite{Bouaziz:1987} and Lemaire \cite{Lemaire:2013} in the twisted case, there exists a locally integrable function $\Theta^{\com{G}, \x_{G}}_{\r}$ on $G(F)$ such that for $x \in G^{\theta}_{reg}(F), g \in G(F)$
\[
\Theta^{\com{G}, \x_{G}}_{\r} (g^{-1} x \theta(g)) = \x_{G}(g) \Theta^{\com{G}, \x_{G}}_{\r}(x),
\]
and
\[
f_{\com{G}}(\r, \x_{G}) = \int_{C_{F} \backslash G(F)} f(g)\Theta^{\com{G}, \x_{G}}_{\r}(g) dg.
\]
By the twisted Weyl integration formula (cf. \cite{Lemaire:2013}, 7.3 and \cite{Mezo:2013}, 5.4.1), one can show this character defines a linear functional on $\mathcal{I}(G^{\theta, \x_{G}}, \chi)$. A linear functional on $\mathcal{I}(G^{\theta, \x_{G}}, \chi)$ is called {\bf stable} if it factors through $\mathcal{SI}(G^{\theta, \x_{G}}, \chi)$. This notion of stability is equivalent to the one we give in the introduction.

We assume $\theta$ preserves an $F$-splitting of $G$. For $\p \in \P{G}$, suppose $\underline{\p}$ factors through $\mathcal{H}$ for a twisted endoscopic datum $[(H, \mathcal{H}, s, \xi)] \in \End{}{G^{\theta}}$, let us write $\underline{\p} = \xi \circ \underline{\p}_{\mathcal{H}}$. Clearly, $sZ(\D{G}) \cap S^{\theta}_{\underline{\p}} \neq \emptyset$ and we denote its image in $\cS{\underline{\p}}^{\theta}$ again by $s$. Let us fix a $z$-pair $(H_{1}, \xi_{H_{1}})$ for $\mathcal{H}$ and define $\underline{\p}_{H_{1}} = \xi_{H_{1}} \circ \underline{\p}_{\mathcal{H}}$. We call $(H_{1}, \underline{\p}_{H_{1}})$ corresponds to $(\underline{\p}, s)$ for $s \in \cS{\underline{\p}}^{\theta}$. It is easy to see that for any semisimple $s \in \cS{\underline{\p}}^{\theta}$ such a pair $(H_{1}, \underline{\p}_{H_{1}})$ always exists (see Section~\ref{subsec: Langlands parameter}). For abbreviation, we write $(H_{1}, \underline{\p}_{H_{1}}) \rightarrow (\underline{\p}, s)$. We always assume the Haar measure is preserved for any admissible embedding $T_{H} \xrightarrow{\simeq} T_{\theta}$ for maximal torus $T_{H} \subseteq H$ and $\theta$-stable maximal torus $T \subseteq G$.

Now we can state the conjectural twisted endoscopic character identity.

\begin{conjecture}% twisted character identity CONJECTURE
\label{conj: twisted character identity}
Suppose $\p \in \Pbd{G}$.

\begin{enumerate}

\item
\begin{align}% stable character Eq
\label{eq: stable character}
f(\underline{\p}) := \sum_{\r \in \Pkt{\p}} <1, \r>_{\underline{\p}} f_{G}(\r), \,\,\,\,\,\,\,\,\,\,\,\, f \in C^{\infty}_{c}(G(F))
\end{align}
is stable.

\item Suppose $s$ is a semisimple element in $\cS{\underline{\p}}^{\theta}$, and $(H_{1}, \underline{\p}_{H_{1}}) \rightarrow (\underline{\p}, s)$. Then
\begin{align}% twisted character identity Eq
\label{eq: twisted character identity}
f^{H_{1}}(\underline{\p}_{H_{1}}) = \sum_{\substack{\r \in \Pkt{\p} \\ \r \cong \r^{\theta}}} <x, \r^{+}>_{\underline{\p}} f_{G^{\theta}}(\r)
\end{align}
for $f \in C^{\infty}_{c}(G(F))$, where $x$ is the image of $s$ in $\com{\S{\underline{\p}}}$, and $\r^{+}$ is an extension of $\r$ to $G^{+}(F) := G(F) \times <\theta>$ with $\r^{+}(\theta) = A_{\r}(\theta)$.

\end{enumerate}

\end{conjecture}

\begin{remark}
In the statement of this conjecture, the character $<\cdot, \r>_{\underline{\p}}$ is given in Conjecture~\ref{conj: endoscopic parametrization}, and $<\cdot, \r^{+}>_{\underline{\p}}$ is given in Conjecture~\ref{conj: twisted endoscopic parametrization}, where one takes $\Sigma = <\theta>, G^{\Sigma} = G^{+}$ and $ \r^{\Sigma} = \r^{+}$.

\end{remark}

In the setup of this conjecture, we can let 
\[
\Theta_{\underline{\p}, x} = \sum_{\r \in \Pkt{\p}} <x, \r^{+}>_{\underline{\p}} \Theta^{\com{G}}_{\r}
\]
and
\[
\Theta_{\underline{\p}_{H_{1}}} = \sum_{\r \in \Pkt{\p_{H_{1}}}} <1, \r>_{\underline{\p}_{H_{1}}}\Theta^{H_{1}}_{\r}.
\]
Then by expanding \eqref{eq: twisted character identity} using the twisted Weyl integration formula, we get
\[
\Theta_{\underline{\p}, x}(\delta) = \sum_{\gamma_{1} \rightarrow \delta} \frac{D_{H_{1}}(\gamma_{1})^{2}}{D_{G^{\theta}}(\delta)^{2}} \Delta_{G, H_{1}}(\gamma_{1}, \delta) \Theta_{\underline{\p}_{H_{1}}}(\gamma_{1}),
\]
where the sum is over stable conjugacy classes of norms $\gamma_{1} \in H_{1, G-reg}(F)$ of $\delta \in G^{\theta}_{reg}(F)$. 

Let $Z_{F}$ be a closed subgroup of $Z_{G}(F)$ such that $Z_{F} \rightarrow Z_{H}(F)$ through \eqref{eq: central inclusion} is injective. If the elements in $\Pkt{\p_{H_{1}}}$ all have the same central character, let us denote its restriction to $C_{1, F}$ by $\chi'$. Then for $z \in Z_{F}$ and $z_{1}$ in its preimage in $C_{1, F}$, we have
\begin{align*}
\Theta_{\underline{\p}, x}(z \delta) & = \sum_{\gamma_{1} \rightarrow \delta} \frac{D_{H_{1}}(z_{1} \gamma_{1})^{2}}{D_{G^{\theta}}(z \delta)^{2}} \Delta_{G, H_{1}}(z_{1}\gamma_{1}, z\delta) \Theta_{\underline{\p}_{H_{1}}}(z_{1} \gamma_{1}) \\
& = \sum_{\gamma_{1} \rightarrow \delta} \chi_{C_{1}}(z_{1})^{-1} \frac{D_{H_{1}}(\gamma_{1})^{2}}{D_{G^{\theta}}(\delta)^{2}} \Delta_{G, H_{1}}(\gamma_{1}, \delta) \chi'(z_{1}) \Theta_{\underline{\p}_{H_{1}}}(\gamma_{1}) \\
& = \chi_{C_{1}}(z_{1})^{-1} \chi'(z_{1}) \Theta_{\underline{\p}, x}(\delta) 
\end{align*}
Note $\chi_{C_{1}}^{-1} \chi'$ is trivial on $Z_{1}(F)$ and hence descents to a quasicharacter on $Z_{F}$, for which we denote by $\chi$. By the linear independence of twisted characters of irreducible smooth representations, one must have 
\[
\Theta^{\com{G}}_{\r}(zg) = \chi(z)\Theta^{\com{G}}_{\r}(g)
\]
for $z \in Z_{F}$ and $\r \in \Pkt{\p}$. In particular, we can let $\theta = id$ and $Z_{F} = Z_{G}(F)$. Then the central character of elements in $\Pkt{\p}$ is $\chi$. This suggests if we want to show for any $\p \in \Pbd{G}$, the elements in $\Pkt{\p}$ have the same central character, we can reduce to the case of simple parameters. Since the L-packet for a simple parameter consists of only one element, there is nothing to show in that case. So we have the following proposition as a consequence of Conjecture~\ref{conj: twisted character identity}.

\begin{proposition}% central character PROPOSITION
\label{prop: central character}
Suppose $\p \in \Pbd{G}$, the elements in $\Pkt{\p}$ all have the same central character.
\end{proposition}

This proposition can be extended to all L-packets by the theory of Langlands quotient.

\subsection{Proof of Theorem~\ref{thm: twisting character}}% proof of thm SUBSECTION
\label{subsec: proof of thm}

Let $G \subseteq \lG$ be two quasisplit connected reductive groups over $F$ such that $G_{der} = \lG_{der}$ and we denote $\lG/G$ by $D$.
\begin{align*}
\xymatrix{1 \ar[r] & G \ar[r] & \lG \ar[r]^{\lambda}  & D \ar[r] & 1}
\end{align*}
We assume $\theta$ is an automorphism of $\lG$ preserving an $F$-splitting of $\lG$, and $\c$ is $\theta$-invariant. Let $G^{+} = G \rtimes <\theta>$.

\begin{lemma}% twisting character  LEMMA
\label{lemma: twisting character}
Suppose $\p \in \Pbd{G}$ and $\r \in \Pkt{\p}$ then
\begin{align}% twisting character Eq
\label{eq: twisting character}
<x, (\r^{+})^{g}>_{\underline{\p}} = \x_{x}(g)<x, \r^{+}>_{\underline{\p}}  
\end{align}
for any $g \in \lG(F)$ and $x \in \com{\S{\underline{\p}}}$, where $\x_{x} = \a(x)$ and $\r^{+}$ is an irreducible representation of $G^{+}(F)$ containing $\r$ in its restriction. 
\end{lemma}

\begin{proof}
Let $\r = \r(\rho)$ for $\rho \in \Irr{\S{\p}}$. If $\r \ncong \r^{\theta}$, then $\rho^{x} \ncong \rho$ for $x \in \com{\S{\underline{\p}}}$ (cf. Lemma~\ref{lemma: theta equivariant} and Conjecture~\ref{conj: twisting equivariant}). Therefore $<x, \r^{+}>_{\underline{\p}} = 0$ for $x \in \com{\S{\underline{\p}}}$, and hence \eqref{eq: twisting character} is clear. Now we will only concern the case $\r \cong \r^{\theta}$. Suppose $s \in \com{\cS{\underline{\p}}}$ and $(H_{1}, \underline{\p}_{H_{1}}) \rightarrow (\underline{\p}, s)$, then by \eqref{eq: twisted character identity} we have
\[
f^{H_{1}}(\underline{\p}_{H_{1}}) = \sum_{\substack{\r \in \Pkt{\p} \\ \r \cong \r^{\theta}}} <x, \r^{+}>_{\underline{\p}} f_{G^{\theta}}(\r)
\]
for $f \in C^{\infty}_{c}(G(F))$, where $x$ is the image of $s$ in $\com{\S{\underline{\p}}}$. We can also reformulate this identity by taking $f \in C^{\infty}_{c}(G(F) \rtimes \theta)$ and view it as in $C^{\infty}_{c}(G(F))$ by sending $g$ to $g \rtimes \theta$, so that we can define its transfer as before. The resulting identity is
\[
f^{H_{1}}(\underline{\p}_{H_{1}}) = \sum_{\r \in \Pkt{\p}} <x, \r^{+}>_{\underline{\p}} f_{G^{+}}(\r^{+})
\]
For $g \in \lG(F), h \in G(F) \rtimes \theta$ we denote $f^{g} (h) = f(ghg^{-1})$. Then by Lemma~\ref{lemma: twisted character} and \eqref{eq: conjugate function} we have
\[
(f^{g})^{H_{1}} = \x_{x}(g)f^{H_{1}},
\]
and hence
\[
(f^{g})^{H_{1}}(\underline{\p}_{H_{1}}) = \x_{x}(g)f^{H_{1}}(\underline{\p}_{H_{1}}).
\]
Using the character identity to expand each side, we get
\begin{align}% theta twisting character 1
\label{eq: theta twisting character 1}
\sum_{\r \in \Pkt{\p}}<x, \r^{+}>_{\underline{\p}} f^{g}_{G^{+}}(\r^{+}) = \sum_{\r \in \Pkt{\p}} \x_{x}(g)<x, \r^{+}>_{\underline{\p}} f_{G^{+}}(\r^{+})   
\end{align}
The left hand side of \eqref{eq: theta twisting character 1} is equal to 
\begin{align}
\label{eq: theta conjugation identity}
\sum_{\r \in \Pkt{\p}} <x, \r^{+}>_{\underline{\p}} f_{G^{+}}((\r^{+})^{g^{-1}}) = \sum_{\r \in \Pkt{\p}}<x, (\r^{+})^{g}>_{\underline{\p}} f_{G^{+}}(\r^{+})
\end{align}
where we substitute $\r^{+}$ for $(\r^{+})^{g^{-1}}$. Compared with the right hand side of \eqref{eq: theta twisting character 1}, this may possibly change the extension of $\r$ by some twist of characters of $G^{+}(F)/G(F)$. Nevertheless, the right hand side of \eqref{eq: theta twisting character 1} is independent of the extensions, so we can certainly choose the same extension as the right hand side of \eqref{eq: theta conjugation identity}. So after these changes, we get 
\begin{align*}
\sum_{\r \in \Pkt{\p}}<x, (\r^{+})^{g}>_{\underline{\p}} f_{G^{+}}(\r^{+}) = \sum_{\r \in \Pkt{\p}} \x_{x}(g)<x, \r^{+}>_{\underline{\p}} f_{G^{+}}(\r^{+}),  \label{eq: theta twisting character 2}
\end{align*}
and hence
\[
<x, (\r^{+})^{g}>_{\underline{\p}} = \x_{x}(g)<x, \r^{+}>_{\underline{\p}}
\]
by the linear independence of twisted characters.
\end{proof}

Now we are going to prove Theorem~\ref{thm: twisting character}. For $\p \in \Pbd{G}$, recall there is a homomorphism 
\[
\a: \S{\underline{\p}}^{\Sigma} \rightarrow \Hom(\lG(F)/G(F), \C^{\times}),
\]
so we can define the homomorphism $\lG(F) \rightarrow (\S{\underline{\p}}^{\Sigma})^{*}$ in the theorem by letting $\e_{g}(x) = \a(x) (g) = \x_{x}(g)$. Fix $\r \in \Pkt{\p}$ and $x \in \S{\underline{\p}}^{\Sigma}$, we denote the image of $x$ in $\D{\Sigma}$ by $\D{\theta}$, then $x \in \S{\underline{\p}}^{\theta}$. Let $\Sigma' = <\theta>$ and $\r^{\Sigma'} = \r^{+}$, it follows from Lemma~\ref{lemma: twisting character} that for any $g \in \lG(F)$
\[
<x, (\r^{\Sigma'})^{g}>_{\underline{\p}} = \e_{g}(x)<x, \r^{\Sigma'}>_{\underline{\p}}.  
\]
On the other hand, we have from \eqref{eq: twisted endoscopic parametrization}
\[
<\cdot, \r^{\Sigma}>_{\underline{\p}} |_{\S{\underline{\p}}^{\Sigma'}} = \sum_{\r^{\Sigma'} \in \Pkt{\p}^{\Sigma'}} m(\r^{\Sigma}, \r^{\Sigma'}) <\cdot, \r^{\Sigma'}>_{\underline{\p}}.
\]
Since $m((\r^{\Sigma})^{g}, (\r^{\Sigma'})^{g}) = m(\r^{\Sigma}, \r^{\Sigma'})$, then
\begin{align*}
<x, (\r^{\Sigma})^{g}>_{\underline{\p}} & = \sum_{\r^{\Sigma'} \in \Pkt{\p}^{\Sigma'}} m((\r^{\Sigma})^{g}, (\r^{\Sigma'})^{g}) <x, (\r^{\Sigma'})^{g}>_{\underline{\p}}  
= \sum_{\r^{\Sigma'} \in \Pkt{\p}^{\Sigma'}} m(\r^{\Sigma}, \r^{\Sigma'}) \e_{g}(x)<x, \r^{\Sigma'}>_{\underline{\p}} \\
& = \e_{g}(x)<x, \r^{\Sigma}>_{\underline{\p}}.
\end{align*}
As we vary $\r \in \Pkt{\p}$ and $x \in \S{\p}^{\Sigma}$, this equality still holds. Therefore we have proved the theorem.

\subsection{Proof of Conjecture~\ref{conj: twisting character 1}}% proof of conjecture SUBSECTION
\label{subsec: proof of conjecture}

In this section, we want to show that Conjecture~\ref{conj: twisting character 1} is a special case of Theorem~\ref{thm: twisting character}. The main step is to clarify the three ingredients in defining the homomorphism $G_{ad}(F) \rightarrow \S{\underline{\p}}^{*}$ in the statement of the conjecture. First we need to recall the construction of $z$-extension. It is a consequence of the following more general construction.

\begin{proposition}% z-extension PROPOSITION
\label{prop: z-extension}
Suppose $F$ is a field of characteristic zero, $G$ and $G'$ are reductive groups over $F$. If $G'$ is semisimple and there is a covering $G' \rightarrow G_{der}$, then there exists a central extension of $G$
\[
\xymatrix{1 \ar[r] & Z \ar[r] & \lG' \ar[r]  & G \ar[r] & 1}
\]
such that

\begin{itemize}

\item $\lG'_{der} = G'$;

\item The projection $\lG'_{der} \rightarrow G_{der}$ coincides with $G' \rightarrow G_{der}$;

\item $Z$ is an induced torus, in particular, $H^{1}(F, Z) = 1$.

\end{itemize}

\end{proposition}

\begin{remark}% z-extension REMARK
\label{rk: z-extension}
When $G' = G_{sc}$, $\lG'$ is the usual $z$-extension of $G$. For the proof of this proposition, we refer the reader to (\cite{MilneShih:1982}, Propositon 3.1) and \cite{Langlands:1979}.
\end{remark}

Now we want to construct the Tate local duality for nonabelian reductive groups. Let $F$ be a local field of characteristic zero and $G$ be a connected reductive group over $F$. Let $G' = G/Z_{G}^{0}$, then $Z_{G'} = Z_{G}/Z_{G}^{0}$. We apply Proposition~\ref{prop: z-extension} to the natural projection $G' \rightarrow G_{ad}$, and we get an extension $\lG'$ of $G_{ad}$
\[
\xymatrix{1 \ar[r] & Z \ar[r] & \lG' \ar[r]  & G_{ad} \ar[r] & 1}
\]
such that $\lG'_{der} = G'$ and $H^{1}(F, Z) = 1$. Moreover, $Z = Z_{\lG'}$, and we denote $\lG'/G'$ by $D$. Consider the exact sequence
\[
\xymatrix{1 \ar[r] & G' \ar[r] & \lG' \ar[r]^{\c}  & D \ar[r] & 1}
\]
The restriction to the centres gives
\[
\xymatrix{1 \ar[r] & Z_{G'} \ar[r] & Z_{\lG'} \ar[r]^{\c}  & D \ar[r] & 1,}
\]
and it induces the following exact sequence
\[
\xymatrix{Z_{\lG'}(F) \ar[r]^{\c} & D(F) \ar[r]  & H^{1}(F, Z_{G'}) \ar[r] & H^{1}(F, Z_{\lG'}) = 1.}
\]
Therefore
\[
H^{1}(F, Z_{G'}) = D(F) / \Im( Z_{\lG'}(F) \xrightarrow{\c} D(F) ).
\]
On the other hand, one considers the following diagram.
\[
\xymatrix{ && 1 \ar[d] && \\
&&\D{D} \ar[d] \ar[dr] && \\
1 \ar[r] & \D{G}_{sc} \ar[r]  \ar[dr] & \D{\lG'} \ar[d] \ar[r]  & \D{Z}_{\lG'} \ar[r] & 1 \\
&& \D{G}' \ar[d] && \\
&& 1 &&
}
\]
Note $\pi_{1}(\D{G'}) = \D{G}_{sc} \cap \D{D}$ and $\D{G}' \cong \D{G}_{der} $, so we get a short exact sequence
\[
\xymatrix{1 \ar[r] & \pi_{1}(\D{G}_{der}) \ar[r] & \D{D} \ar[r]  & \D{Z}_{\lG'} \ar[r] & 1.}
\]
This induces the following exact sequence
\[
\xymatrix{\pi_{0}(\D{Z}_{\lG'}^{\Gal{}}) \ar[r] & H^{1}(F, \pi_{1}(\D{G}_{der})) \ar[r] & H^{1}(F, \D{D}) \ar[r] & H^{1}(F, \D{Z}_{\lG'}).}
\]
By the Tate-Nakayama duality for tori (see \cite{Kottwitz:1984}, (3.3.1)), we have $\pi_{0}(\D{Z}_{\lG'}^{\Gal{}})^{*} = H^{1}(F, Z_{\lG'}) = 1$, and hence $\pi_{0}(\D{Z}_{\lG'}^{\Gal{}}) = 1$. Therefore
\[
H^{1}(F, \pi_{1}(\D{G}_{der})) = \Ker(H^{1}(F, \D{D}) \rightarrow H^{1}(F, \D{Z}_{\lG'})).
\]
It also follows from the Tate-Nakayama duality for tori that $H^{1}(F, \D{D})$ (resp. $H^{1}(F, \D{Z}_{\lG'})$) is canonically isomorphic to the group of continuous characters of finite orders on $D(F)$ (resp. $Z_{\lG'}(F)$) (see \cite{Kottwitz:1984}, (3.3.2)). Since $\Im( Z_{\lG'}(F) \xrightarrow{\c} D(F) )$ has finite index in $D(F)$, $\Ker(H^{1}(F, \D{D}) \rightarrow H^{1}(F, \D{Z}_{\lG'}))$ is isomorphic to characters of $D(F)$ that are trivial on $\Im( Z_{\lG'}(F) \xrightarrow{\c} D(F) )$, and this is exactly the dual of $D(F) / \Im( Z_{\lG'}(F) \xrightarrow{\c} D(F) )$. Hence we get a perfect pairing 
\begin{align}% pairing Eq
\label{eq: pairing}
H^{1}(F, Z_{G'}) \times H^{1}(F, \pi_{1}(\D{G}_{der})) \rightarrow \C^{\times}.
\end{align}
The fact that this pairing is independent of the choice of extension with respect to $G' \rightarrow G_{ad}$ is because of the following proposition.

\begin{proposition}% uniqueness of pairing PROPOSITION
\label{prop: uniqueness of pairing}

\begin{enumerate}

\item If there is another extension 
\[
\xymatrix{1 \ar[r] & Z_{1} \ar[r] & \lG'_{1} \ar[r]  & G_{ad} \ar[r] & 1}
\]
dominating the original extension, i.e., there is a surjection $\lG'_{1} \rightarrow \lG'$ such that the following diagram commutes
\[
\xymatrix{1 \ar[r] & Z_{1} \ar[d] \ar[r] & \lG'_{1} \ar[r] \ar[d]  & G_{ad} \ar[r] \ar@{=}[d]& 1 \\ 
1 \ar[r] & Z \ar[r] & \lG' \ar[r]  & G_{ad} \ar[r] & 1.}
\]
Then the pairing \eqref{eq: pairing} obtained from this extension is the same as the original one.

\item If there are two extensions
\[
\xymatrix{1 \ar[r] & Z_{i} \ar[r] & \lG'_{i} \ar[r]  & G_{ad} \ar[r] & 1.} \,\,\,\,\, (i = 1,2)
\]
Then one can find a third extension which dominates both of them.

\end{enumerate}

\end{proposition}

The proof of part (i) is straightforward and we leave it to the reader. The proof of part (ii) can be found in (\cite{Kottwitz:1982}, Lemma 1.1).

Since $H^{1}(F, Z_{\lG'}) = 1$, then $G_{ad}(F) = \lG'(F) / Z_{\lG'}(F)$ and 
\[
\xymatrix{G_{ad}(F) = \lG'(F) / Z_{\lG'}(F) \ar[r]^{\c \quad \quad \quad \quad} & D(F) / \Im( Z_{\lG'}(F) \xrightarrow{\c} D(F) ) = H^{1}(F, Z_{G'})}
\]
%\[
%G_{ad}(F) = \lG'(F) / Z_{\lG'}(F) \xrightarrow{\c} D(F) / \Im( Z_{\lG'}(F) \xrightarrow{\c} D(F) ) = H^{1}(F, Z_{G'}).
%\]
defines the homomorphism $G_{ad}(F) \rightarrow H^{1}(F, Z_{G'})$ in the introduction. Just like the Tate local duality pairing, one can show this homomorphism is independent of the choice of extension with respect to $G' \rightarrow G_{ad}$.

Finally, if $\p \in \Pbd{G}$, we have defined a homomorphism $\delta: S_{\underline{\p}} \rightarrow H^{1}(W_{F}, \D{D})$ (see \eqref{eq: old twisted endoscopic sequence}). By Lemma~\ref{lemma: centralizer}, the image of $\delta$ is finite. So $\delta$ factors through $A_{\underline{\p}}$ and the image lies in $H^{1}(F, \D{D})$. Moreover, we claim the image of $\delta$ lies in $\Ker(H^{1}(F, \D{D}) \rightarrow H^{1}(F, \D{Z}_{\lG'}))$. In fact, this follows from the proof of Lemma~\ref{lemma: centralizer}. For the convenience of the reader, we repeat that argument here. For $s \in S_{\underline{\p}}$, let $\lif{s}$ be a preimage of $s$ in $\D{\lG'}$. Recall
\[
\delta(s): u \mapsto \lif{s}\lp(u)\lif{s}^{-1}\lp(u)^{-1} = \lif{s}\sigma_{u}(\lif{s}^{-1}) \cdot \underbrace{\sigma_{u}(\lif{s})\lp(u)\lif{s}^{-1}\lp(u)^{-1}}_{\in \D{G}_{sc}}
\]
for $u \in L_{F}$ and $\sigma_{u}$ is the image of $u$ in $\Gal{}$. Note the decomposition of $\delta(s)(u)$ factors through $\Gal{}$. Then our claim follows from the following diagram.
\[
\xymatrix{& H^{1}(F, \D{D}) \ar[d] \ar[dr] & \\
H^{1}(F, \D{G}_{sc}) \ar[r] & H^{1}(F, \D{\lG'})) \ar[r] & H^{1}(F, \D{Z}_{\lG'})
}
\]
So we obtain a homomorphism $\delta: A_{\underline{\p}} \rightarrow H^{1}(F, \pi_{1}(\D{G}_{der}))$. 

From the construction above, we obtain a homomorphism $G_{ad}(F) \rightarrow A_{\underline{\p}}^{*}$, which sends $g$ to $\eta_{g}$. It is easy to check that 
\[
\eta_{g}(s) = \a(x)(\lif{g})
\]
for $s \in A_{\underline{\p}}$ with image $x \in \S{\p}$, and $\lif{g} \in \lG'(F)$ with image $g \in G_{ad}(F)$. As a consequence, $\eta_{g} \in \S{\underline{\p}}^{*}$ and Conjecture~\ref{conj: twisting character 1} follows from Theorem~\ref{thm: twisting character} immediately.

%-------------------------------------------------------------------------------------------

\section{Central character}% central character SECTION
\label{sec: central character}

For $\p \in \P{G}$, one can associate a character $\chi_{\p}$ of $Z_{G}(F)$ as in the introduction. By Proposition~\ref{prop: central character}, we see the central characters of elements in $\Pkt{\p}$ are the same. So we can talk about the central character of an L-packet, and we would like to show it is equal to $\chi_{\p}$. By the construction of $\chi_{\p}$ and also nontempered L-packets, we see it suffices to prove this for $\p \in \Pbd{G}$. Note if $\p$ is simple, $\Pkt{\p}$ contains only one element and we would like to assume the central character of $\Pkt{\p}$ is $\chi_{\p}$. Then it is enough to check how $\chi_{\p}$ and the central character of representations change with respect to the endoscopic transfer.

\begin{lemma}% transfer central character LEMMA
\label{lemma: transfer central character}
Let $\p \in \Pbd{G}$ and $s \in \cS{\underline{\p}}$. Suppose for any $(H_{1}, \underline{\p}_{H_{1}}) \rightarrow (\underline{\p}, s)$, the central character of $\Pkt{\p_{H_{1}}}$ is $\chi_{\p_{H_{1}}}$, then the central character of $\Pkt{\p}$ is $\chi_{\p}$.
\end{lemma}

\begin{proof}
First we want to reduce to the case $\mathcal{H} = \L{H}$. To do so, we can simply take a $z$-extension of $G$
\[
\xymatrix{1 \ar[r] & Z_{1} \ar[r] & G_{1} \ar[r]  & G \ar[r] & 1,}
\]
and denote the image of $\p$ in $\Pbd{G_{1}}$ by $\p_{1}$. Note $(G_{1})_{der} = G_{sc}$, then by a result of Langlands (see Proposition 1, \cite{Langlands:1979}), $\underline{\p}_{1}$ factors though an endoscopic datum $(H_{1}, \L{H_{1}}, s, \xi_{1})$. This gives a natural embedding $\xi_{H_{1}}: \mathcal{H} \rightarrow \L{H_{1}}$ and a $z$-extension
\[
\xymatrix{1 \ar[r] & Z_{1} \ar[r] & H_{1} \ar[r]  & H \ar[r] & 1.}
\]
Therefore, $(H_{1}, \xi_{H_{1}})$ is a $z$-pair for $\mathcal{H}$. By our assumption, $\chi_{\p_{H_{1}}}$ is the central character of $\Pkt{\p_{H_{1}}}$. If we can show $\chi_{\p_{1}}$ is the central character of $\Pkt{\p_{1}}$, then the same is true for $\chi_{\p}$.

From now on, we assume $\mathcal{H} = \L{H}$ and we take $H_{1} = H$. By the definition of $\chi_{\p}$, we need to take a torus $Z$ containing the centre $Z_{G}$ of $G$, and form $\lG = (G \times Z) / Z_{G}$. Then $H$ can be lifted to a twisted endoscopic group $\lif{H}$ of $\lG$. Let $\underline{\lp}_{H}$ be a lift of $\underline{\p}_{H}$, and it gives a lift $\underline{\lp}$ of $\underline{\p}$. Then $\chi_{\p_{H}} = \chi_{\lp_{H}}|_{Z_{H}}$ and $\chi_{\p} = \chi_{\lp}|_{Z_{G}}$. Note $\chi_{\lp} =  \chi_{\lif{\xi}} \cdot (\chi_{\lp_{H}}|_{Z_{\lG}})$, where $\chi_{\lif{\xi}}$ is dual to 
\[
W_{F} \xrightarrow{\lif{\xi}|_{W_{F}}} \L{\lG} \rightarrow \L{Z_{\lG}}.
\]
On the other hand, the central character of $\Pkt{\p}$ only differs from the restriction of that of $\Pkt{\p_{H}}$ to ${Z_{G}}$ by $\chi^{-1}_{C}$. This follows from our proof of Proposition~\ref{prop: central character} by taking $C = C_{1} = Z_{G}$. Since $\chi_{C} = \chi_{\lif{C}}|_{Z_{G}}$, it is enough to show $\chi_{\lif{\xi}} = \chi^{-1}_{\lif{C}}$. To give the definition of $\chi_{\lif{C}}$, we need to fix $\Gal{}$-splittings $(\mathcal{\lif{B}}_{\lif{H}}, \mathcal{\lif{T}}_{\lif{H}}, \{\mathcal{X}_{\alpha_{\lif{H}}} \} )$ and $(\mathcal{\lif{B}}, \mathcal{\lif{T}}, \{\mathcal{X}_{\alpha} \} )$ for both $\D{\lif{H}}$ and $\D{\lG}$. By taking certain $\D{\lG}$-conjugate of $\lif{\xi}$, we can assume $\lif{\xi}(\mathcal{\lif{T}}_{\lif{H}}) =  \mathcal{\lif{T}}$, and $\mathcal{\lif{B}}_{\lif{H}} \subseteq \mathcal{\lif{B}}$. We also choose a maximal torus $\lif{T}_{\lif{H}}$ of $\lif{H}$ defined over $F$, and choose an admissible embedding $\lif{T}_{\lif{H}} \rightarrow \lif{T}$ together with $\chi$-data on $R(\lG, \lif{T})$. The admissible embedding identifies $\L{\lif{T}}_{\lif{H}}$ with $\L{\lif{T}}$, and transports $\chi$-data from $R(\lG, \lif{T})$ to $R(\lif{H}, \lif{T}_{\lif{H}})$. The $\chi$-data give embeddings $\xi_{\lif{T}_{\lif{H}}}: \L{\lif{T}}_{\lif{H}} \rightarrow \L{\lif{H}}$ and $\xi_{\lif{T}}: \L{\lif{T}} \rightarrow \L{\lG}$. Then there exists a $1$-cocycle $a_{\lif{T}}$ of $W_{F}$ in $\D{\lif{\mathcal{T}}}$ with transported Galois action from $\lif{T}$ such that
\begin{align}% transfer central character Eq
\label{eq: transfer central character}
\lif{\xi} \circ \xi_{\lif{T}_{\lif{H}}} = a_{\lif{T}} \cdot \xi_{\lif{T}},
\end{align}
and $\chi^{-1}_{\lif{C}}$ is dual to 
\[
W_{F} \xrightarrow{a_{\lif{T}}} \D{\mathcal{\lif{T}}} \rightarrow \L{\lG} \rightarrow \L{Z_{\lG}}.
\]
By the constructions of $\xi_{\lif{T}_{\lif{H}}}$ and $\xi_{\lif{T}}$ (see \cite{LanglandsShelstad:1987}, 2.5), one can see $\xi_{\lif{T}_{\lif{H}}}(W_{F}) \subseteq \D{\lif{H}}_{der} \rtimes W_{F}$ and $\xi_{\lif{T}}(W_{F}) \subseteq \D{\lG}_{der} \rtimes W_{F}$, so if we restrict both sides of \eqref{eq: transfer central character} to $W_{F}$ and compose them with $\L{\lG} \rightarrow \L{Z_{\lG}}$, we get an equality for the duals of $\chi_{\lif{\xi}}$ and $\chi^{-1}_{\lif{C}}$. Therefore, $\chi_{\lif{\xi}} = \chi^{-1}_{\lif{C}}$.
\end{proof}

It is clear that this lemma implies Proposition~\ref{prop: central character desideratum}.

%-------------------------------------------------------------------------------------------
\section{Twist by automorphism and quasicharacter}% twist SECTION
\label{sec: twist}

Let $\theta$ be an automorphism of $G$ preserving an $F$-splitting. Let ${\bold a}$ be an element in $H^{1}(W_{F}, Z(\D{G}))$, which is associated with a quasicharacter $\x$ of $G(F)$. In this section we want to prove Proposition~\ref{prop: twisting equivariant}. %We are going to assume this is true for simple parameters. Note when $\p$ is simple, it is enough to assume $\Pkt{\p^{\theta}} = \Pkt{\p}^{\theta}$ and $\Pkt{\p \otimes {\bold a}} = \Pkt{\p} \otimes \x$ for both equalities \eqref{eq: theta equivariant} and \eqref{eq: twist equivariant} are trivial in this case. To give the proof, we will also assume Conjecture~\ref{conj: twisted character identity}.

\begin{lemma}% theta equivariant LEMMA
\label{lemma: theta equivariant}
Suppose $\p \in \Pbd{G}$ is not simple, then $\Pkt{\p^{\theta}} = \Pkt{\p}^{\theta}$. Moreover, 
\[
<x, \r^{\theta}>_{\underline{\p}^{\theta}} = <\D{\theta}^{-1}x\D{\theta}, \r>_{\underline{\p}}
\]
for any $\r \in \Pkt{\p}$ and $x \in \S{\underline{\p}^{\theta}}$.
\end{lemma}

\begin{proof}
For $s \in \cS{\underline{\p}}$, we assume $(H_{1}, \underline{\p}_{H_{1}}) \rightarrow (\underline{\p}, s)$, with respect to $(H, \mathcal{H},  s, \xi)$ and $z$-pair $(H_{1}, \xi_{H_{1}})$. Then we have $(H_{1}, \underline{\p}_{H_{1}}) \rightarrow (\underline{\p}^{\theta}, \D{\theta} s \D{\theta}^{-1})$, with respect to $(H, \mathcal{H}, \D{\theta} s \D{\theta}^{-1}, \xi^{\theta})$ and the same $z$-pair $(H_{1}, \xi_{H_{1}})$. To make a distinction, we denote the transfer factor with respect to $(H, \mathcal{H}, \D{\theta} s \D{\theta}^{-1}, \xi^{\theta})$ by $\Delta_{G, H'_{1}}$, and the transfer by $f^{H'_{1}}$ for $f \in C^{\infty}_{c}(G(F))$. Note $f^{H'_{1}}$ is defined on $H_{1}(F)$.

%If $f \in C^{\infty}_{c}(G(F))$, we denote the transfer with respect to $\{H, \mathcal{H},  s, \xi\}$ by $f^{H_{1}}$, and the transfer with respect to $\{H, \mathcal{H}, \D{\theta} s \D{\theta}^{-1}, \xi^{\theta}\}$ by $f^{H'_{1}}$. Note both $f^{H_{1}}, f^{H'_{1}}$ are functions on $H_{1}(F)$.

If $f \in C^{\infty}_{c}(G(F))$, we can choose the transfers so that they satisfy
\begin{align}% theta equivariant transfer Eq
\label{eq: theta equivariant transfer}
(f^{\theta})^{H'_{1}} = f^{H_{1}}.
\end{align}
To see this, let $\gamma_{1}$ be a semisimple strongly $G$-regular element of $H_{1}(F)$, let $T_{H_{1}}$ be the centralizer of $\gamma_{1}$. Let $T_{H}$ be the projection of $T_{H_{1}}$ on $H$, and $\gamma \in H(F)$ be the image of $\gamma_{1}$. We fix an admissible embedding $T_{H} \rightarrow T$ with respect to $(H, \mathcal{H},  s, \xi)$, and denote the image of $\gamma$ by $\delta$. Then the admissible embedding of $T_{H}$ with respect to $(H, \mathcal{H}, \D{\theta} s \D{\theta}^{-1}, \xi^{\theta})$ becomes the composition of 
\[
T_{H} \rightarrow T \xrightarrow{\theta^{-1}} \theta^{-1}(T).
\] 
This is because the endoscopic embedding $\xi$ changes to $\xi^{\theta}$. Note $\gamma$ maps to $\theta^{-1}(\delta)$ under this admissible embedding. Then
\begin{align*}
SO_{H_{1}}((f^{\theta})^{H'_{1}}, \gamma_{1}) & = \sum_{\{\delta'\}_{G(F)} \thicksim_{st} \{\delta\}_{G(F)}} \Delta_{G, H'_{1}}(\gamma_{1}, \theta^{-1}(\delta')) O_{G}(f^{\theta}, \theta^{-1}(\delta')) \\
& = \sum_{\{\delta'\}_{G(F)} \thicksim_{st} \{\delta\}_{G(F)}} \Delta_{G, H'_{1}}(\gamma_{1}, \theta^{-1}(\delta')) O_{G}(f, \delta').
\end{align*}
By the definition of transfer factors, one can check
\[
\Delta_{G, H'_{1}}(\gamma_{1}, \theta^{-1}(\delta')) = \Delta_{G, H_{1}}(\gamma_{1}, \delta'),
\] 
so we have
\[
SO_{H_{1}}((f^{\theta})^{H'_{1}}, \gamma_{1}) = SO_{H_{1}}(f^{H_{1}}, \gamma_{1}).
\]

It follows from \eqref{eq: theta equivariant transfer} that 
\[
f^{H_{1}}(\underline{\p}_{H_{1}}) = (f^{\theta})^{H'_{1}}(\underline{\p}_{H_{1}}).
\]
Now we can expand both sides by the endoscopic character identities:
\[
f^{H_{1}}(\underline{\p}_{H_{1}}) = \sum_{\r \in \Pkt{\p}} <x, \r>_{\underline{\p}} f_{G}(\r),
\]
and
\begin{align*}
(f^{\theta})^{H'_{1}}(\underline{\p}_{H_{1}}) &= \sum_{\r \in \Pkt{\p^{\theta}}} <\D{\theta} x \D{\theta}^{-1}, \r>_{\underline{\p}^{\theta}} f^{\theta}_{G}(\r) \\
&= \sum_{\r \in \Pkt{\p^{\theta}}} <\D{\theta} x \D{\theta}^{-1}, \r>_{\underline{\p}^{\theta}} f_{G}(\r^{\theta^{-1}}),
\end{align*}
where $x$ is the image of $s$ in $\S{\underline{\p}}$. By linear independence of characters, for any $\r' \in \Pkt{\p^{\theta}}$, there exists $\r \in \Pkt{\p}$ such that $\r^{\theta} \cong \r'$. This shows $\Pkt{\p}^{\theta} = \Pkt{\p^{\theta}}$. Moreover, 
\[
<x, \r>_{\underline{\p}} = <\D{\theta} x \D{\theta}^{-1}, \r'>_{\underline{\p}^{\theta}} = <\D{\theta} x \D{\theta}^{-1}, \r^{\theta}>_{\underline{\p}^{\theta}}
\]
Let $x' = \D{\theta} x \D{\theta}^{-1} \in \S{\underline{\p}^{\theta}}$, then we get $<\D{\theta}^{-1} x' \D{\theta}, \r>_{\underline{\p}} = < x', \r^{\theta}>_{\underline{\p}^{\theta}}$.

\end{proof}

\begin{lemma}% twist equivariant LEMMA
\label{lemma: twist equivariant}
Suppose $\p \in \Pbd{G}$ is not simple, then $\Pkt{\p \otimes {\bold a}} = \Pkt{\p} \otimes \x$. Moreover,
\[
<x, \r \otimes \x> _{\underline{\p} \otimes \underline{{\bold a}}} = <x, \r>_{\underline{\p}}.
\]
for any $\r \in \Pkt{\p}$ and $x \in \S{\underline{\p}} = \S{\underline{\p} \otimes \underline{{\bold a}}}$, where $\underline{{\bold a}}$ is a $1$-cocyle of $W_{F}$ in $Z(\D{G})$ representing ${\bold a}$. 
\end{lemma}

\begin{proof}
For $s \in \cS{\underline{\p}}$, we assume $(H_{1}, \underline{\p}_{H_{1}}) \rightarrow (\underline{\p}, s)$, with respect to $(H, \mathcal{H},  s, \xi)$ and $z$-pair $(H_{1}, \xi_{H_{1}})$. Then we have $(H'_{1}, \underline{\p}_{H_{1}}) \rightarrow (\underline{\p} \otimes \underline{{\bold a}}, s)$, with respect to $(H, \mathcal{H}, s, \xi \otimes \underline{{\bold a}})$ and the same $z$-pair $(H_{1}, \xi_{H_{1}})$. To make a distinction, we denote the transfer factor with respect to $(H, \mathcal{H}, s, \xi \otimes \underline{{\bold a}})$ by $\Delta_{G, H'_{1}}$, and the transfer by $f^{H'_{1}}$ for $f \in C^{\infty}_{c}(G(F))$. Note $f^{H'_{1}}$ is defined on $H_{1}(F)$.

%If $f \in C^{\infty}_{c}(G(F))$, we denote the transfer with respect to $\{H, \mathcal{H},  s, \xi\}$ by $f^{H_{1}}$, and the transfer with respect to $\{H, \mathcal{H}, s, \xi \otimes {\bold a} \}$ by $f^{H'_{1}}$. Note both $f^{H_{1}}, f^{H'_{1}}$ are functions on $H_{1}(F)$. 

If $f \in C^{\infty}_{c}(G(F))$, one can choose the transfers so that they satisfy 
\begin{align}% twist equivariant transfer Eq
\label{eq: twist equivariant transfer}
(f \otimes \x^{-1})^{H'_{1}} = f^{H_{1}}.
\end{align}
To see this, let $\gamma_{1}$ be a semisimple strongly $G$-regular element of $H_{1}(F)$, let $T_{H_{1}}$ be the centralizer of $\gamma_{1}$. Let $T_{H}$ be the projection of $T_{H_{1}}$ on $H$, and $\gamma \in H(F)$ be the image of $\gamma_{1}$. We fix an admissible embedding $T_{H} \rightarrow T$ with respect to $(H, \mathcal{H},  s, \xi)$, and denote the image of $\gamma$ by $\delta$. Then the admissible embedding of $T_{H}$ with respect to $(H, \mathcal{H}, s, \xi \otimes \underline{{\bold a}})$ is the same. 
\begin{align*}
SO_{H_{1}}((f \otimes \x^{-1})^{H'_{1}}, \gamma_{1}) & = \sum_{\{\delta'\}_{G(F)} \thicksim_{st} \{\delta\}_{G(F)}} \Delta_{G, H'_{1}}(\gamma_{1}, \delta') O_{G}(f \otimes \x^{-1}, \delta') \\
& = \sum_{\{\delta'\}_{G(F)} \thicksim_{st} \{\delta\}_{G(F)}} \Delta_{G, H'_{1}}(\gamma_{1}, \delta') \x^{-1}(\delta') O_{G}(f, \delta').
\end{align*}
Moreover, we have
\[
\Delta_{G, H'}(\gamma, \delta') = \x(\delta') \Delta_{G, H}(\gamma, \delta').
\]
In fact, this difference between transfer factors only comes from $\Delta_{III}$, or more precisely $\Delta_{2}$ (see \cite{LanglandsShelstad:1987}, Section 3.5 for its definition). Therefore,
\[
SO_{H_{1}}((f \otimes \x^{-1})^{H'_{1}}, \gamma_{1}) = SO_{H_{1}}(f^{H_{1}}, \gamma_{1}). 
\]

It follows from \eqref{eq: twist equivariant transfer} that
\[
f^{H_{1}}(\underline{\p}_{H_{1}}) = (f \otimes \x^{-1})^{H'_{1}}(\underline{\p}_{H_{1}}).
\]
Now we can expand both sides by the endoscopic character identities:
\[
f^{H_{1}}(\underline{\p}_{H_{1}}) = \sum_{\r \in \Pkt{\p}} <x, \r>_{\underline{\p}} f_{G}(\r),
\]
and
\begin{align*}
(f \otimes \x^{-1})^{H'_{1}}(\underline{\p}_{H_{1}}) & = \sum_{\r \in \Pkt{\p \otimes {\bold a}}} <x, \r>_{\underline{\p} \otimes \underline{{\bold a}}} (f \otimes \x^{-1})_{G}(\r) \\
&= \sum_{\r \in \Pkt{\p \otimes {\bold a}}} <x, \r>_{\underline{\p} \otimes \underline{{\bold a}}} f_{G}(\r \otimes \x^{-1}),
\end{align*}
where $x$ is the image of $s$ in $\S{\underline{\p}} = \S{\underline{\p} \otimes \underline{{\bold a}}}$. By linear independence of characters, for any $\r' \in \Pkt{\p \otimes {\bold a}}$ there exists $\r \in \Pkt{\p}$ such that $\r' \otimes \x^{-1} \cong \r$. This implies $\Pkt{\p \otimes {\bold a}} = \Pkt{\p} \otimes \x$. Furthermore,
\[
<x, \r \otimes \x >_{\underline{\p} \otimes \underline{{\bold a}}} = <x, \r'>_{\underline{\p} \otimes \underline{{\bold a}}} = <x, \r>_{\underline{\p}}.
\]
This finishes the proof.

\end{proof}

%-------------------------------------------------------------------------------------------
\section{Lifting L-packet}% Lifting packet SECTION
\label{sec: lifting L-packet}

Let $G \subseteq \lG$ be two quasisplit connected reductive groups over $F$, such that $G_{der} = \lG_{der}$ and we denote $\lG/G$ by $D$. Suppose $\lp \in \Pbd{\lG}$, and $\p$ is the image of $\lp$ under $\Pbd{\lG} \rightarrow \Pbd{G}$, then it is conjectured that $\Pkt{\lp}|_{G} = \Pkt{\p}$. The problem we want to study is to what extent one can understand the L-packet of $\lG$ from that of $G$. Therefore, we will only assume the endoscopic hypothesis (i.e., \eqref{eq: disjoint decomposition}, Conjecture~\ref{conj: twisted endoscopic parametrization} and Conjecture~\ref{conj: twisted character identity}) for $G$ and all its twisted endoscopic groups. To be more precise, this will be our working assumption in Section~\ref{subsec: coarse L-packet}-\ref{subsec: refinement}. It follows the previous results the we have proved about the desiderata of L-packets are valid for $G$.

\subsection{Representation theoretic preparation}% representation
\label{subsec: representation}

We start by investigating the restriction multi-map $\Pkt{}(\lG(F)) \rightarrow \Pkt{}(G(F))$. Similar discussion of this restriction multi-map can also be found in \cite{LabesseLanglands:1979}, \cite{HiragaSaito:2012} and \cite{GelbartKnapp:1982}.

\begin{lemma}% finite restriction LEMMA
\label{lemma: finite restriction}
If $\lr$ is an irreducible smooth representation of $\lG(F)$, then the restriction of $\lr$ to G(F) is a direct sum of finitely many irreducible smooth representations.
\end{lemma}

\begin{proof}
Since $\lr$ has a central character $\chi_{\lr}$, it is enough to show the restriction of $\lr$ to $\lZ(F)G(F)$ is a direct sum of finitely many irreducible smooth representations. Note $|D(F) : \c(\lZ(F))|$ is finite, so the index $|\lG(F) : \lZ(F)G(F)|  =  |\c(\lG(F)) : \c(\lZ(F))| < |D(F) : \c(\lZ(F))|$ is also finite. Then this lemma follows from the following algebraic result.

\end{proof}

\begin{lemma}% algebra LEMMA
\label{lemma: algebra}
Let $G$ and $H$ be two groups, such that $H$ is a normal subgroup of $G$ and $G/H$ is finite.
\begin{enumerate}
\item If $\lr$ is an irreducible representation of $G$, then the restriction of $\lr$ to $H$ is a direct sum of finitely many irreducible representations.
\item If $\r$ is an irreducible representation of $H$, then there exists an irreducible representation $\lr$ of $G$ which contains $\r$ in its restriction to $H$.
\end{enumerate}
\end{lemma}

\begin{proof}
\begin{enumerate}
\item Let $g_{1}, g_{2}, \cdots, g_{r}$ be the representatives of $G/H$ and  $g_{1} = 1$. Let us assume the restriction of $\lr$ to $H$ is reducible. We first need to show there exists a direct sum decomposition of the representation space $V = V(\lr|_{H}) = \bigoplus_{i=1}^{l} \lr(g_{v_{i}}) W$ for some proper $H$-invariant subspace W, and $1 \leqslant v_{i} \leqslant r$. Suppose there exists a direct sum $0 \neq\bigoplus_{i=1}^{l} \lr(g_{v_{i}}) W \subsetneq V$, then $\bigcap_{k=1}^{r}\bigoplus_{i=1}^{l} \lr(g_{k}g_{v_{i}}) W = 0$ for it is invariant under $G$, but not equal to $V$. Hence we can choose $\{k_{1}, k_{2}, \dots, k_{m} \} \subseteq \{1, 2, \cdots, r\}$ so that $W \bigcap_{j=1}^{m} \bigoplus_{i=1}^{l}\lr(g_{k_{j}}g_{v_{i}}) W = 0 $, but $W' = W \bigcap_{j=1}^{m-1}\bigoplus_{i}^{l}\lr(g_{k_{j}}g_{v_{i}})W \neq 0$. Here we let $W' = W$ if $m =1$. Since $W' \cap \bigoplus_{i=1}^{l}\lr(g_{k_{m}}g_{v_{i}})W= 0$, then $W' + \bigoplus_{i=1}^{l} \lr(g_{k_{m}}g_{v_{i}})W'$ is again a direct sum. Note that we have increased the number of direct summands by 1. By repeating this argument, we will end up with a direct sum which is either the whole space $V$ or equal to $\bigoplus_{i=1}^{r}\lr(g_{i})W''$ with respect to some $H$-invariant subspace $0 \neq W'' \subseteq W$. In the latter case, it is again equal to $V$ for it is invariant under $G$.

Now we can assume there is a direct sum decomposition of $V = \bigoplus_{i=1}^{l} \lr(g_{v_{i}})W $ with respect to some $W$. Suppose $W$ is reducible,  then there exists an $H$-invariant subspace $W'$ in $W$, and $\bigoplus_{i=1}^{l} \lr(g_{v_{i}})W' \neq V$. This implies $l < r$. Hence $W$ must be irreducible if $l = r$. In case $l < r $, we can apply the argument in the previous paragraph and find $W''$ in $W'$ so that $V = \bigoplus_{i=1}^{m} \lr(g_{v_{i}})W'' $ and $m > l$. If $W''$ is reducible we can repeat this argument until either we get an irreducible subrepresentation in which case the proof is done, or we decompose $V$ into a direct sum of $r$ subspaces.  In the latter case, it is clear now each subspace has to be irreducible. Therefore $\lr$ can be decomposed into a finite direct sum of irreducible $H$-representations. Moreover, it is easy to see the direct summands run over all the isomorphism classes of $G$-conjugates of any irreducible representation $\r$ contained in $\lr|_{H}$.

\item Let $\lr$ be any irreducible representation of $G$, from Frobenius reciprocity we have
\[
\text{Hom}_{H}( \, \Res^{G}_{H} \lr,  \r \, )  \cong  \text{Hom}_{G}(\lr , \Ind^{G}_{H} \r).
\]
Then it is easy to see from part (i) of the lemma that $\lr$ contains $\r$ in its restriction to $H$ if and only if $\lr$ is a subrepresentation of $\sigma = \Ind^{G}_{H} \r$. So it is enough to show $\sigma$ has an irreducible subrepresentaion. Note that $\sigma|_{H} = \bigoplus_{i=1}^{r} \r^{g_{i}}$, so we have projections $p_{i}: V(\sigma) \rightarrow V(\r^{g_{i}})$. If $W$ is a $G$-invariant subspace of $V(\sigma)$, we are going to define a sequence of subspaces as follows. Let $W_{1} = W, W_{2} = \text{Ker} \,p_{1}|_{W_{1}}, W_{3} = \text{Ker} \,p_{2}|_{W_{2}}, \cdots, W_{r} = \text{Ker} \,p_{r-1}|_{W_{r-1}}$, and $W_{r+1} = 0$. Then we have 
\[
0 = W_{r+1} \subseteq W_{r} \subseteq W_{r-1} \subseteq \cdots \subseteq W_{1} = W,
\]
where $W_{i}/ W_{i+1} \simeq \r^{g_{i}} \text{ or } 0$ for $1 \leqslant i \leqslant r$. In particular there exists a unique sequence of integers $r \geqslant s_{m} > s_{m-1} > \cdots > s_{1} \geqslant 1$ such that
\[
0 \subsetneq W_{s_{m}} \subsetneq W_{s_{m-1}} \subsetneq \cdots \subsetneq W_{s_{1}} = W,
\]
with $W_{s_{i}}/ W_{s_{i+1}} \simeq \r^{g_{s_{i}}} $ for $1 \leqslant i \leqslant m-1$ and $W_{s_{m}} = \r^{g_{m}}$. We call $m = m(W)$ the length of $W$. Now let us take a proper $G$-invariant subspace $W$ of minimal length, then $W$ has to be irreducible. Otherwise, there exists another $G$-invariant subspace $W' \subsetneq W$, and if $\{s_{1}, s_{2}, \cdots, s_{m}\}$ is associated with $W$, then $W'_{s_{i}}/ W'_{s_{i+1}} \subseteq W_{s_{i}} / W_{s_{i+1}} \simeq \r^{g_{s_{i}}} $. From here we see $m(W') \leqslant m(W)$, and hence $m(W') = m(W)$.  This means $W'_{s_{m}} = W_{s_{m}}$ and $W'_{s_{i}}/ W'_{s_{i+1}} = W_{s_{i}} / W_{s_{i+1}}$. Therefore $W' = W$. 

\end{enumerate}
\end{proof}

As an immediate consequence of part (ii) of this lemma, we have the following corollary.

\begin{corollary}% existence COROLLARY
\label{cor: existence}
If $\r$ is an irreducible smooth representation of $G(F)$, then there exists an irreducible smooth representation $\lr$ of $\lG(F)$ which contains $\r$ in its restriction to $G(F)$. In particular, the  central character $\chi_{\r}$ can be extended to a character of $\lZ(F)$. 

\end{corollary}

\begin{proof}
Let $\lif{Z}_{F}$ be a closed subgroup of $\lZ(F)$ such that $\lif{Z}_{F} \cap \Z(F) = 1$ and $D(F)/\c(\lif{Z}_{F})$ is finite. Then we can extend $\r$ to $\lif{Z}_{F}G(F)$ through the trivial character on $\lif{Z}_{F}$. Since $\lG(F)/\lif{Z}_{F}G(F)$ is finite, the existence of $\lr$ follows from Lemma~\ref{lemma: algebra} directly, and moreover its central character $\chi_{\lr}$ extends $\chi_{\r}$.

The closed subgroup $\lif{Z}_{F}$ can be constructed as follows. We first choose an $F$-subtorus $C$ of $\lZ^{0}$, such that $\lZ^{0} = C \Z^{0}$ and $C \cap \Z^{0}$ is finite. It is easy to see $\c(C(F))$ has finite index in $D(F)$, and $|C(F) \cap \Z(F)|$ is finite. Next we choose integer $m$ such that $\lif{Z}_{F} := \{x^{m} : x \in C(F)\}$ has no torsion points. Then $\lif{Z}_{F} \cap \Z(F) = 1$, and $\c(\lif{Z}_{F})$ also has finite index in $D(F)$.

\end{proof}

Reviewing part (ii) of Lemma~\ref{lemma: algebra}, we see the irreducible subrepresentations of $\Ind^{G}_{H}\r$ give all the irreducible representations of $G$ whose restriction to $H$ contains $\r$. So it is interesting to determine the structure of $\Ind^{G}_{H}\r$. This may not be easy in general, but when $G/H$ is abelian and the irreducible representations of $H$ satisfy Schur's Lemma, one can actually compute the induction very explicitly. Especially, note that if $\lif{Z}_{F}$ is a closed subgroup of $\lZ(F)$ such that $D(F)/\c(\lif{Z}_{F})$ is finite, $\lG(F)/\lif{Z}_{F}G(F)$ is also abelian. So now we are going to calculate $\Ind^{G}_{H}\r$ under the assumption that $G/H$ is abelian. In fact, we can take any sequence of normal subgroups  
\[
H = H_{0} \subseteq H_{1} \subseteq \cdots \subseteq H_{i} \subseteq \cdots \subseteq H_{r} = G
\]
such that $H_{i+1}/H_{i}$ is cyclic and of prime order. Then 
\[
\Ind^{G}_{H}\r = \Ind^{G}_{H_{r-1}} \cdots \Ind^{H_{i+1}}_{H_{i}} \cdots \Ind^{H_{1}}_{H} \r.
\]
As we will see for any irreducible representation $\sigma$ of $H_{i}$, the induction $\Ind^{H_{i+1}}_{H_{i}} \sigma$ is always semisimple, so it is enough for us to consider the case when $G/H$ is a cyclic group of prime order $p$. Let $g \in G$ be a generator of the cyclic group $G/H$ and let us assume $\r^g \cong \r$, then there exists an intertwining operator $A$ of $V(\r)$ such that for all $h \in H$ we have
\[
A \circ \r(h) = \r(g h g^{-1}) \circ A. \]
So
\[
A^p \circ \r(h) = \r(g^p h g^{-p}) \circ A^p = \r(g^p) \circ  \r(h) \circ  \r(g^p)^{-1} \circ A^p, \]
and 
\[
(\r(g^p)^{-1} \circ A^p) \circ \r(h) = \r(h) \circ (\r(g^p)^{-1} \circ A^p).
\]
Since $\r$ is irreducible, $\r(g^p)^{-1} \circ A^p = cI$ for some constant $c$. After rescaling $A$, we can assume $c=1$ and hence $A^p = \r(g^p)$. This shows we can extend $\r$ to an irreducible representation $\lr$ of $G$ by defining $\lr(g) = A$. In fact if we change the scaling of $A$ by a $p$-th root of unity, we can get another extension $\lr \otimes \x$ for some character $\x$ of $G/H$. Let $\{\x_{i}\}_{i=1}^{p}$ be all the characters of $G/H$, then it is easy to see that $\lr \otimes \x_{i}$ are distinct for all $1 \leqslant i \leqslant p$. Our claim is 
\begin{align}% cyclic extension Eq
\label{eq: cyclic extension}
\Ind^{G}_{H}\r \cong \bigoplus_{i=1}^{p} \lr \otimes \x_{i}.
\end{align}
To see this, we first get inclusions from $\lr \otimes \x_{i}$ to $\Ind^{G}_{H}\r$ for all $1 \leqslant i \leqslant p$ by Frobenius reciprocity. Then this gives a $G$-invariant homomorphism from $\bigoplus_{i=1}^{p} \lr \otimes \x_{i}$ to $\Ind^{G}_{H}\r$. Since $\lr \otimes \x_{i}$ are distinct, then this homomorphism must be injective. Otherwise the image of some $\lr \otimes \x_{k}$ will be contained in the image of $\bigoplus_{i \neq k }\lr \otimes \x_{i}$, but that is impossible. The subjectivity will follow from a simple argument on the lengths of representations as defined in the proof of part (ii) of Lemma~\ref{lemma: algebra}, when we restrict to $H$. Finally, if $\r^{g} \ncong \r$, then $\Ind^{G}_{H} \r$ is irreducible because any irreducible subrepresentation of $\Ind^{G}_{H} \r$ contains $\r^{g^{i}}$ for $1 \leqslant i \leqslant p$ in its restriction to $H$.

Next, we will give a formula for $\Ind^{\lG(F)}_{\lif{Z}_{F}G(F)}\r$, where $\r$ is an irreducible smooth representation of $G(F)$, which can be extended to $\lif{Z}_{F}G(F)$ through some quasicharacter $\lif{\chi}$ of $\lif{Z}_{F}$. Let us denote 
\[
\lG(\r) = \{ g \in \lG(F) : \r^g \cong \r \}. 
\]
Suppose $\iG{F}$ is a maximal subgroup of $\lG(F)$, to which one can extend $\r$. Note that such $\iG{F}$ may not be unique. If we denote such an extension by $\ir$, then by \eqref{eq: cyclic extension} we have
\[
 \Ind^{\iG{F}}_{\lif{Z}_{F}G(F)} \r \cong \bigoplus_{\x \in (\iG{F}/\lif{Z}_{F}G(F))^*} \ir \otimes \x,
\]
and 
\[
\Ind^{\lG(F)}_{\lif{Z}_{F}G(F)} \r \cong \Ind^{\lG(F)}_{\iG{F}} \Ind^{\iG{F}}_{\lif{Z}_{F}G(F)} \r \cong \bigoplus_{\x \in (\iG{F}/\lif{Z}_{F}G(F))^*} \Ind^{\lG(F)}_{\iG{F}} (\ir \otimes \x).
\]

Note that $\Ind^{\lG(F)}_{\iG{F}} (\ir \otimes \x)$ is irreducible, so we can assume $\lr \cong \Ind^{\lG(F)}_{\iG{F}} \ir$ by making a good choice of $\ir$. Now we want to count the multiplicities in the decomposition of $\Ind^{\lG(F)}_{\lif{Z}_{F}G(F)} \r$ above. Observe $\Ind^{\lG(F)}_{\iG{F}} \ir \cong \Ind^{\lG(F)}_{\iG{F}} (\ir \otimes \x) $ if and only if $(\ir)^g \cong \ir \otimes \x$ for some $g \in \lG(F)$. In fact such $g$ must be in $\lG(\r)$. So we consider the homomorphism 
\begin{displaymath}
\xymatrix{\lG(\r) \ar@{->}[rr]  & &  (\iG{F} /  \lif{Z}_{F}G(F))^* \\
g \ar@{|->}[rr] & & \x : (\ir)^g \cong \ir \otimes \x}
\end{displaymath}
and the kernel is $\iG{F}$ by maximality. If we denote the image of this homomorphism by $c(\r)$, then 
\begin{align}% induction Eq
\label{eq: induction}
\Ind^{\lG(F)}_{\lif{Z}_{F}G(F)} \r \cong |c(\r)| \bigoplus_{\x \in (\iG{F}/\lif{Z}_{F}G(F))^* / c(\r)} \Ind^{\lG(F)}_{\iG{F}} (\ir \otimes \x).
\end{align}

As a consequence of this formula, we have the following corollaries.

\begin{corollary}% uniqueness COROLLARY
\label{cor: uniqueness}
If $\r$ is an irreducible smooth representation of G(F), then the irreducible smooth representation $\lr$ of $\lG(F)$, which contains $\r$ in its restriction to $G(F)$, is unique up to twisting by $\Hom(\lG(F)/G(F), \C^{\times})$.
\end{corollary}

\begin{proof}
As in Corollary~\ref{cor: existence}, we can let $\lif{Z}_{F}$ be a closed subgroup of $\lZ(F)$ such that $\lif{Z}_{F} \cap \Z(F) = 1$ and $D(F)/\c(\lif{Z}_{F})$ is finite. Then for any two irreducible smooth representations $\lr_{1}, \lr_{2}$, which contain $\r$ in their restrictions to $G(F)$, one can choose $\x \in \Hom(\lG(F)/G(F), \C^{\times})$ such that the restrictions of $\lr_{1} \otimes \x$ and $\lr_{2}$ to $\lif{Z}_{F}G(F)$ all contain the same representation which extends $\r$. By Frobenius reciprocity and \eqref{eq: induction}, $\lr_{1} \otimes \x \cong \lr_{2} \otimes \x'$ for some $\x' \in (\lG(F)/\lif{Z}_{F}G(F))^*$. Therefore, $\lr_{1} \cong \lr_{2} \otimes \x' \x^{-1}$.
\end{proof}

\begin{corollary}% restriction multiplicity COROLLARY
\label{cor: restriction multiplicity}
If $\lr$ is an irreducible smooth representation of $\lG(F)$, then the irreducible smooth representation $\r$ of $G(F)$ in the restriction of $\lr$ all have the same multiplicity and it is equal to $|c(\r)|$.
\end{corollary}

\begin{proof}
It follows from the proof of Lemma~\ref{lemma: algebra} that $\Res^{\lG(F)}_{G(F)} \lr$ consists of isomorphism classes of $\r^{g}$ for $g \in \lG(F)$. By \eqref{eq: induction} and Frobenius reciprocity, the multiplicity of $\r^{g}$ is $|c(\r^{g})| = |c(\r)|$. This finishes the proof.
\end{proof}

\begin{lemma}% multiplicity one for generic representation LEMMA
\label{lemma: multiplicity one for generic representation}
Suppose $\lr$ is an irreducible smooth generic representation of $\lG(F)$, then the multiplicity of irreducible smooth representation $\r$ of $G(F)$ in the restriction of $\lr$ is equal to one.
\end{lemma}

\begin{proof}
Since $\lr$ is generic, there exists a generic representation $\r$ of $G(F)$ in the restriction of $\lr$. For $g \in \lG(\r)$, the intertwining operator $A_{g}: \r \rightarrow \r^{g}$ will preserve the Whittaker functional up to a scalar. Here we are using the uniqueness of Whittaker model. As a consequence, we can normalize $A_{g}$ for all $g \in \lG(\r)$, so that they all preserve the Whittaker functional. Then one can check easily that $\r$ can be extended by these intertwining operators to $\lG(\r)$. This means $\iG{F} = \lG(\r)$, and hence $|c(\r)| = 1$. Now this lemma will follow from Corollary~\ref{cor: restriction multiplicity}. 
\end{proof}

If $\lr$ is an irreducible smooth representation of $\lG(F)$, let us denote 
\[
X(\lr) = \{ \x \in (\lG(F)/\lZ(F)G(F))^{*} : \lr \otimes \x \cong \lr \}.
\]
We denote the multiplicity of an irreducible smooth representation $\r$ of $G(F)$ in the restriction of $\lr$ by $m(\lr, \r)$. Next we want to give a formula for $m(\lr, \r)$ in terms of $X(\lr)$ and $\lG(\r)$.

%\begin{corollary}%COROLLARY
%\label{cor: restriction multiplicity one}
%If $\lr$ is an irreducible admissible representation of $\lG(F)$ and $\r$ is contained in its restriction to $G(F)$, then $\x$ is in $X(\lr)$ if and only if $\x$ is trivial on $\lG(\r)$. Moreover, the restriction of $\lr$ contains $|X(\lr)|$ irreducible admissible representations of G(F).
%\end{corollary}

%\begin{proof}
%Suppose $\r$ is contained in the restriction of $\lr$, then by Frobenius reciprocity we have
%\[
%\text{Hom}_{\lZ(F)G(F)}( \, \Res^{\lG(F)}_{\lZ(F)G(F)} \lr,  \r \, )  \cong  \text{Hom}_{\lG(F)}(\lr , \Ind^{\lG(F)}_{\lZ(F)G(F)} \r),
%\]
%where on the right hand side we extend $\r$ to $\lZ(F)G(F)$ through $\chi_{\lr}$. By Conjecture~\ref{conj: restriction multiplicity one} the dimension of the left hand side is one, so is the dimension of the right hand side. In view of the formula \eqref{eq: induction}, this means $|c(\r)| = |\lG(\r)/\iG{F}| = 1$. So $\lG(\r) = \iG{F}$, and we conclude $\lr \cong \lr \otimes \x$ if and only if $\x \in (\lG(F)/\lG(\r))^*$ by \eqref{eq: induction}. Moreover, the number of irreducible constituents in the restriction of $\lr$ is $|\lG(F) / \iG{F}| = |\lG(F) / \lG(\r)| = |X(\lr)|$.
%\end{proof}

\begin{corollary}% restriction multiplicity formula COROLLARY
\label{cor: restriction multiplicity formula}
If $\lr$ is an irreducible smooth representation of $\lG(F)$ and $\r$ is contained in its restriction to $G(F)$, then 
\begin{align}% restriction multiplicity formula Eq
\label{eq: restriction multiplicity formula}
m(\lr, \r)^{2} = \frac{|X(\lr)|}{|\lG(F)/\lG(\r)|}.
\end{align}
\end{corollary}

\begin{proof}
It follows from Corollary~\ref{cor: restriction multiplicity} that $m(\lr, \r) = |c(\r)|$. By definition, $|c(\r)| = |\lG(\r)/\iG{F}|$. On the other hand, it follows from \eqref{eq: induction} that $X(\lr)$ is the preimage of $c(\r)$ under 
\[
(\lG(F)/\lZ(F)G(F))^{*} \longrightarrow (\iG{F}/\lZ(F)G(F))^{*}.
\]
Note the kernel of this map is $(\lG(F)/\iG{F})^{*}$, so $|X(\lr)| = |c(\r)| \cdot |\lG(F)/\iG{F}|$. Cancelling $\iG{F}$ from these two identities, we get
\[
|c(\r)|^{2} = \frac{|X(\lr)|}{|\lG(F)/\lG(\r)|}.
\]
This finishes the proof.

\end{proof}

\begin{remark}% restriction multiplicity formula REMARK
\label{rk: restriction multiplicity formula}
In the next section, we will consider the situation that both $G$ and $H$ as in Lemma~\ref{lemma: algebra} are finite groups. It is not hard to see that the corollaries above can also be stated for such pairs, and the proofs are the same. 
\end{remark}

At last, we show the restriction multi-map $\Pkt{}(\lG(F)) \rightarrow \Pkt{}(G(F))$ preserves temperedness.

\begin{lemma}% restrict tempered character LEMMA
\label{lemma: restrict tempered character}
Suppose $\lr$ is an irreducible smooth unitary representation of $\lG(F)$, then $\lr$ is an essential discrete series representation of $\lG(F)$ if and only if its restriction to $G(F)$ consists of essential discrete series representations. The same is true of the tempered representations.
\end{lemma}

\begin{proof}
If $\lr$ is an essential discrete series representation then the matrix coefficient $<\lr(g)v, w^{\vee}>$ for $v \in V(\lr)$ and $w^{\vee} \in V(\lr)^{\vee}$ is a square integrable function modulo the centre. In particular, its restriction to $G(F)$ is square integrable modulo the centre, hence the restriction of $\lr$ consists of essential discrete series representations. Conversely, we can write the matrix coefficient of $\lr$ as a piecewise defined function on the components of $\lG(F) / \lZ(F)G(F)$, where on each component it is defined as 
\[
<\lr(hg)v, w^{\vee}> = <\lr(h)(\lr(g)v), w^{\vee}>
\]
for some fixed representatives $g \in \lG(F)$ of $\lG(F)/\lZ(F)G(F)$ and $h \in \lZ(F)G(F)$, which is a matrix coefficient of the restriction of $\lr$. So the restriction of $\lr$ consisting of essential discrete series representations implies $\lr$ is an essential discrete series representation. The same kind of argument also applies to tempered representations when we replace the condition of square integrability by $L^{2+\epsilon}$.

%For the tempered representations, they are induced from the essential discrete series representations. From the Bruhat decomposition, it is easy to see $\c(\lif{B}(F)) = \c(\lG(F))$ for any Borel subgroup $\lif{B}$ of $\lG$, and hence we have $\lP(F) \backslash \lG(F) \cong P(F) \backslash G(F)$ for any parabolic subgroup $\lP$ of $\lG$. Then there is a natural isomorphism  
%\[
%\Res^{\lG(F)}_{G(F)} \Ind^{\lG(F)}_{\lP(F)} ( \lr_{M} \otimes \delta_{\lP}^{1/2} ) \cong  \Ind^{G(F)}_{P(F)} (\Res^{\lM(F)}_{M(F)} \lr_{M} \otimes \delta_{P}^{1/2} ), 
%\]
%where $\delta_{\lP}$ and $\delta_{P}$ are the usual modulus characters and one notes $\delta_{\lP} |_{M} = \delta_{P}$. This shows $\lr$ is tempered if and only if its restriction to $G(F)$ is tempered.
\end{proof}

\subsection{Coarse L-packet}% coarse L-packet SUBSECTION
\label{subsec: coarse L-packet}

%From now on we assume for any $\p \in \Pbd{G}$, the restriction to $\S{\underline{\lp}}^{\Sigma}$ of any irreducible representation of $\S{\underline{\p}}^{\Sigma}$ is multiplicity free. 

%Restrict the map $\Pkt{}(\lG(F)) \rightarrow \Pkt{}(G(F))$ to the L-packets of $G$, we have the following results, which are consequences of Theorem~\ref{thm: twisting character}. 

In this section, we want to describe the preimage of L-packets of $G$ under $\Pkt{}(\lG(F)) \rightarrow \Pkt{}(G(F))$. To do so, we need the following hypothesis.

\begin{hypothesis}% twisting character HYPOTHESIS
\label{hypo: twisting character}
Suppose $\p \in \Pbd{G}$, let $\rho \in \Irr{\S{\underline{\p}}}$ and $\tau \in \Irr{\S{\underline{\lp}}}$ be in the restriction $\rho|_{\S{\underline{\lp}}}$. Let $\lr$ be an irreducible smooth representation of $\lG(F)$, whose restriction to $G(F)$ contains $\r = \r(\rho)$, then for any $x \in \S{\underline{\p}}$
\[
\tau^{x} \cong \tau \Longleftrightarrow \lr \cong \lr \otimes \x_{x}.
\]
Moreover, 
\[
X(\lr) = \a(\S{\underline{\p}}(\tau)),
\]
where $\S{\underline{\p}}(\tau) = \{x \in \S{\underline{\p}} : \tau^{x} \cong \tau\}$.
\end{hypothesis}

It is clear that this hypothesis is a consequence of Conjecture~\ref{conj: twisting equivariant} for $\lG$, which is not assumed in Section~\ref{sec: lifting L-packet}. Since this hypothesis will be used on top of our working assumption for this section, we will point it out whenever we assume this hypothesis.
%So whenever we use this conjecture in the following sections, we will point it out. 
The next proposition is kind of dual to this hypothesis.

\begin{proposition}% twisting character PROPOSITION
\label{prop: twisting character}
Suppose $\p \in \Pbd{G}$, let $\rho \in \Irr{\S{\underline{\p}}}$ and $\r = \r(\rho)$. Let 
\[
X(\rho) = \{\e \in (\S{\underline{\p}}/\S{\underline{\lp}})^{*} : \rho \otimes \e \cong \rho\},
\]
then $\{\e_{g} : g \in \lG(\r)\} = X(\rho)$, where $\e_{g}(x) = \x_{x}(g) = \a(x)(g)$ for $x \in \S{\underline{\p}}$.
\end{proposition}

\begin{proof}
For $g \in \lG(\r)$, by Lemma~\ref{lemma: twisting character},
\[
<x, \r>_{\underline{\p}} = <x, \r^{g}>_{\underline{\p}} = \e_{g}(x)<x, \r>_{\underline{\p}},
\]
and hence $\e_{g} \in X(\rho)$. This shows $\{\e_{g} : g \in \lG(\r)\} \subseteq X(\rho)$. 

For the other direction, 
%if $\S{\underline{\p}}$ is abelian, $X(\rho) = 1$ and we can conclude $\{\e_{g} : g \in \lG(\r)\} = X(\rho)$. In particular, this is the case when $F$ is archimedean (cf. \cite{Shelstad:}). In case $F$ is nonarchimedean, 
note the map $\a: x \mapsto \x_{x}$ embeds $\S{\underline{\p}} / \S{\underline{\lp}}$ into $\Hom(\lG(F)/G(F), \C^{\times})$ (see \eqref{eq: twisted endoscopic sequence with dual}), so the map $g \mapsto \e_{g}$ from $\lG(F)$ to $(\S{\underline{\p}}/\S{\underline{\lp}})^{*}$ is surjective. Hence for any $\e \in X(\rho)$, we can assume $\e = \e_{g}$ for some $g \in \lG(F)$. Then
\[
<x, \r^{g}>_{\underline{\p}} = \e_{g}(x)<x, \r>_{\underline{\p}} = <x, \r>_{\underline{\p}}.
\]
By injectivity of the map $\r \rightarrow <\cdot, \r>_{\underline{\p}}$, one must have $\r^{g} \cong \r$, i.e., $g \in \lG(\r)$.
\end{proof}

\begin{corollary}% twisting character COROLLARY
\label{cor: twisting character}
Suppose $\p \in \Pbd{G}$, let $\rho \in \Irr{\S{\underline{\p}}}$ and $\r = \r(\rho)$. Let 
\[
\Ker(X(\rho)) = \{x \in \S{\underline{\p}} : \e(x) = 1 \text{ for all } \e \in X(\rho)\},
\]
then $\a(\Ker(X(\rho))) = (\lG(F)/\lG(\r))^{*}$.
\end{corollary}

\begin{proof}
Consider the pairing $\lG(F) \times \S{\underline{\p}} \rightarrow \C^{\times}$ which sends $(g, x)$ to $\e_{g}(x) = \a(x)(g)$. It becomes a perfect pairing of abelian groups after taking quotients by $\S{\underline{\lp}} \subseteq \S{\underline{\p}}$, and $U \subseteq \lG(F)$, which is annihilated by $\S{\underline{\p}}$. We claim $U \subseteq \lG(\r)$. This is because if $\e_{g} = 1$, then
\[
<x, \r^{g}>_{\underline{\p}} = \e_{g}(x)<x, \r>_{\underline{\p}} = <x, \r>_{\underline{\p}}.
\]
By injectivity of the map $\r \rightarrow <\cdot, \r>_{\underline{\p}}$, one must have $\r^{g} \cong \r$, i.e., $g \in \lG(\r)$. By Proposition~\ref{prop: twisting character} and the Pontryagin duality applied to the perfect pairing $\lG(F)/U \times \S{\underline{\p}}/\S{\underline{\lp}} \rightarrow \C^{\times}$, we have a perfect pairing $(\lG(F)/U)/(\lG(\r)/U) \times \Ker(X(\rho))/\S{\underline{\lp}} \rightarrow \C^{\times}$. Therefore, $(\lG(F)/\lG(\r))^{*} = ((\lG(F)/U)/(\lG(\r)/U))^{*} = \a(\Ker(X(\rho)))$.
\end{proof}

\begin{proposition}% restriction multiplicity formula PROPOSITION
\label{prop: restriction multiplicity formula}
Suppose $\p \in \Pbd{G}$, let $\rho \in \Irr{\S{\underline{\p}}}$ and $\tau \in \Irr{\S{\underline{\lp}}}$ be in the restriction $\rho|_{\S{\underline{\lp}}}$ with multiplicity $m(\rho, \tau)$. Let $\lr$ be an irreducible smooth representation of $\lG(F)$ whose restriction to $G(F)$ contains $\r = \r(\rho)$. Under Hypothesis~\ref{hypo: twisting character}, we have $m(\lr, \r) = m(\rho, \tau)$.
\end{proposition}

\begin{proof}
By Corollary~\ref{cor: restriction multiplicity formula}, we have 
\[
m(\lr, \r)^{2} = \frac{|X(\lr)|}{|\lG(F)/\lG(\r)|}.
\]
Similarly, one can show 
\[
m(\rho, \tau)^{2} = \frac{|X(\rho)|}{|\S{\underline{\p}}/\S{\underline{\p}}(\tau)|}
\] 
(see Remark~\ref{rk: restriction multiplicity formula}). To relate these two expressions, we take Hypothesis~\ref{hypo: twisting character} and apply Corollary~\ref{cor: twisting character} to the formula of $m(\lr, \r)^{2}$, and we get
\begin{align*}
m(\lr, \r)^{2} &= \frac{|\a(\S{\underline{\p}}(\tau))|}{|\a(\Ker(X(\rho)))|} = \frac{|\S{\underline{\p}}(\tau) / \S{\underline{\lp}}|}{|\Ker(X(\rho))/\S{\underline{\lp}}|} \\
& = \frac{|\S{\underline{\p}}(\tau) / \S{\underline{\lp}}|}{|(\S{\underline{\p}}/\S{\underline{\lp}})^{*}/X(\rho)|} = \frac{|\S{\underline{\p}}(\tau) / \S{\underline{\lp}}| |X(\rho)|}{|(\S{\underline{\p}}/\S{\underline{\lp}})|} = \frac{|X(\rho)|}{|\S{\underline{\p}}/\S{\underline{\p}}(\tau)|} \\
& = m(\rho, \tau)^{2}.
\end{align*}
Hence $m(\lr, \r) = m(\rho, \tau)$.
\end{proof}

%\begin{remark}
This proposition suggests that $m(\lr, \r) = 1$ if $\S{\underline{\p}}$ is abelian. For classical groups, it has been shown that $\S{\underline{\p}}$ is always abelian (see \cite{Arthur:2013}, \cite{Mok:2014}). On the other hand, when $G$ is a symplectic group or special even orthogonal group, and $\lG$ is its similitude group, it has been proved that $m(\lr, \r) = 1$ (see \cite{AdlerPrasad:2006}, Theorem 1.4). In fact, one can prove Hypothesis~\ref{hypo: twisting character} under the assumption that $m(\lr, \r) = m(\rho, \tau) = 1$.
%\end{remark}

\begin{proposition}% twisting character under multiplicity one PROPOSITION
\label{prop: twisting character under multiplicity one}
Suppose $\p \in \Pbd{G}$, let $\rho \in \Irr{\S{\underline{\p}}}$ and $\tau \in \Irr{\S{\underline{\lp}}}$ be in the restriction $\rho|_{\S{\underline{\lp}}}$. Let $\lr$ be an irreducible smooth representation of $\lG(F)$, whose restriction to $G(F)$ contains $\r = \r(\rho)$. If $m(\lr, \r) = m(\rho, \tau) = 1$, then for any $x \in \S{\underline{\p}}$
\[
\tau^{x} \cong \tau \Longleftrightarrow \lr \cong \lr \otimes \x_{x}.
\]
Moreover, 
\[
X(\lr) = \a(\S{\underline{\p}}(\tau)),
\]
where $\S{\underline{\p}}(\tau) = \{x \in \S{\underline{\p}} : \tau^{x} \cong \tau\}$. 
\end{proposition}

\begin{proof}
If $m(\lr, \r) = m(\rho, \tau) = 1$, then $X(\lr) = (\lG(F)/\lG(\r))^{*}$ and $X(\rho) = (\S{\underline{\p}}/\S{\underline{\p}}(\tau))^{*}$. It follows $\Ker(X(\rho)) = \S{\underline{\p}}(\tau)$. By Corollary~\ref{cor: twisting character}, $X(\lr) = \a(\Ker(X(\rho))) = \a(\S{\underline{\p}}(\tau))$. This implies the direction $``\Rightarrow"$. For the other direction, 
%it is trivial when $F$ is archimedean for $\S{\underline{\p}}$ is then abelian. In the nonarchimedean case, 
one can always choose $x_{0} \in \S{\underline{\p}}(\tau)$ such that $\x_{x} = \x_{x_{0}}$, which implies $xx^{-1}_{0} \in \S{\underline{\lp}}$. Hence $x \in \S{\underline{\p}}(\tau)$.
\end{proof}

For $\p \in \Pbd{G}$, we assume the central character of $\Pkt{\p}$ is $\chi_{\p}$. Let us fix a character $\lif{\chi}_{\p}$ of $\lZ(F)$ such that $\lif{\chi}_{\p}|_{\Z(F)} = \chi_{\p}$. Then we define $\lPkt{\p, \lif{\chi}_{\p}}$ to be the subset of $\Pkt{}(\lG(F))$ with central character $\lif{\chi}_{\p}$, whose restriction to $G(F)$ are contained in $\Pkt{\p}$. Let $X = \Hom(\lG(F)/\lZ(F)G(F), \C^{\times})$, then $X$ acts on $\lPkt{\p, \lif{\chi}_{\p}}$ by twisting. We call $\lPkt{\p, \lif{\chi}_{\p}}$ a coarse L-packet for $\lG$ and its structure can be described in the following proposition.

\begin{proposition}% coarse L-packet PROPOSITION
\label{prop: coarse L-packet}
Suppose $\p \in \Pbd{G}$ and $\lif{\chi}_{\p}$ is chosen as above. We assume Hypothesis~\ref{hypo: twisting character}.
\begin{enumerate}
\item If $\rho \in \Irr{\S{\underline{\p}}}$, then the $\lG(F)$-conjugate orbit of $\r(\rho)$ has size $|\a(\Ker(X(\rho)))|$. \\

\item There is a pairing (not necessarily unique)
\begin{align}% lift pairing Eq
\label{eq: lift pairing}
\lr \longrightarrow <\cdot, \lr>_{\underline{\p}}
\end{align}
from $\lPkt{\p, \lif{\chi}_{\p}}$ to $\D{\S{\underline{\lp}}}$, such that

\begin{enumerate}

\item 
\[
<\cdot, \lr \otimes \x_{x}>_{\underline{\p}} = <x(\cdot)x^{-1}, \lr>_{\underline{\p}}
\] 
for $x \in \S{\underline{\p}}$.

\item 
\[
<\cdot, \r>_{\underline{\p}}|_{\S{\underline{\lp}}} = m(\lr, \r) \sum_{x \in \S{\underline{\p}}/\S{\underline{\p}}(\tau)}< \cdot, \lr \otimes \x_{x}>_{\underline{\p}}
\] 
for any $\r \in \Pkt{\p}$ in the restriction of $\lr$.

\end{enumerate}
Moreover, it sends the generic representation to the trivial character of $\S{\underline{\lp}}$.
%Moreover, this mapping from $\lPkt{\p, \chi} / X$ to $\D{\S{\lp}}$ is injective and when $F$ is nonarchimedean it is in fact a bijection.
\end{enumerate} 
\end{proposition}

\begin{proof}
Suppose $\r \in \Pkt{\p}$, then the orbit of $\r$ under the conjugate action of $\lG(F)$ has size $|\lG(F)/\lG(\r)|$. By Corollary~\ref{cor: twisting character}, we know $\a(\Ker(X(\rho))) = (\lG(F)/ \lG(\r))^{*}$. Hence $|(\lG(F)/ \lG(\r))| = |(\lG(F)/ \lG(\r))^{*}| = |\a(\Ker(X(\rho)))|$. 

For the second part, we can choose any $\r(\rho)$ in the restriction of $\lr \in \lPkt{\p, \lif{\chi}_{\p}}$ and choose any irreducible subrepresentation $\tau$ in $\rho|_{\S{\underline{\lp}}}$. We also fix a set of representatives $\{\x_{i}\}$ in $X$ of $X / \a(\S{\underline{\p}})$. We assign $\tau$ to all $\lr \otimes \x_{i}$ and extend to $\lr \otimes \x$ for any $\x \in X$ 
by letting
\begin{align}% lift twisting character Eq
\label{eq: lift twisting character}
<\cdot, \lr \otimes \x_{x}>_{\underline{\p}} := <x(\cdot)x^{-1}, \lr>_{\underline{\p}}
\end{align}
%i.e.,
%\[
%\lr(\tau^{x}) := \lr(\tau) \otimes \x_{x}
%\]
for $x \in \S{\underline{\p}}$. This is well-defined because of Hypothesis~\ref{hypo: twisting character}.  By this construction, it is clear that (a) is satisfied. Moreover, this definition is independent of choice of $\r(\rho)$. To see this, let us replace $\r(\rho)$ by $\r(\rho)^{g}$ for $g \in \lG(F)$, by Lemma~\ref{lemma: twisting character} we have 
\[
<\cdot, \r^{g}>_{\underline{\p}}|_{\S{\underline{\lp}}} = \x_{x}(g)<\cdot, \r>_{\underline{\p}}|_{\S{\underline{\lp}}}  = <\cdot, \r>_{\underline{\p}}|_{\S{\underline{\lp}}}.
\]
Then (b) follows from (a) and Proposition~\ref{prop: restriction multiplicity formula}. Finally, if $\lr$ is generic, there exists a generic representation $\r$ in its restriction, i.e., $<\cdot, \r>_{\underline{\p}} = 1$. It is easy to see that $<\cdot, \lr>_{\underline{\p}} = 1$ by our construction.

\end{proof}

\subsection{Compatibility with $\theta$-twist}% compatibility with twist SUBSECTION
\label{subsec: compatibility with twist}

Before we give the refinement of $\lPkt{\p, \lif{\chi}_{\p}}$, we want to show how the pairing in Proposition~\ref{prop: coarse L-packet} can also be made to satisfy a special case of Conjecture~\ref{conj: twisting equivariant}. First we would like to generalize Hypothesis~\ref{hypo: twisting character} to the $\theta$-twisted case.

\begin{hypothesis}% theta twisting character HYPOTHESIS
\label{hypo: theta twisting character}
Suppose $\p \in \Pbd{G}$, let $\rho \in \Irr{\S{\underline{\p}}}$ and $\tau \in \Irr{\S{\underline{\lp}}}$ be in the restriction $\rho|_{\S{\underline{\lp}}}$. Let $\lr$ be an irreducible smooth representation of $\lG(F)$ whose restriction to $G(F)$ contains $\r(\rho)$, then for any $x \in \S{\underline{\p}}^{\theta}$
\[
\tau^{x} \cong \tau \Longleftrightarrow \lr^{\theta} \cong \lr \otimes \x_{x}.
\]
\end{hypothesis}

\begin{remark}% theta twisting character REMARK
\label{rk: theta twisting character}

\begin{enumerate}

\item

Fix $\tau_{0} \in \Irr{\S{\underline{\lp}}}$, we can construct $1-1$ correspondences between $\{\tau_{0}^{y}: y \in \S{\underline{\p}}\}$ and $\{\lr(\tau_{0}) \otimes \x_{y}: y \in \S{\underline{\p}}\}$ through \eqref{eq: lift twisting character}, where $\lr(\tau_{0}) \in \lPkt{\p, \lif{\chi}_{\p}}$. If we fix such a correspondence, and suppose $\S{\underline{\p}}^{\theta}$ acts on $\{\tau_{0}^{y}: y \in \S{\underline{\p}}\}$, then it follows from this hypothesis that 
\[
\lr(\tau^{x})^{\theta} \cong \lr(\tau) \otimes \x_{x},
\]
for any $\tau \in \{\tau_{0}^{y}: y \in \S{\underline{\p}}\}$ and $x \in \S{\underline{\p}}^{\theta}$. More generally, if $\tau_{0}' := \tau_{0}^{x_{0}} \notin \{\tau_{0}^{y}: y \in \S{\underline{\p}}\}$ for some $x_{0} \in \S{\underline{\p}}^{\theta}$, then by taking $\lr(\tau_{0}')$ such that $\lr(\tau_{0}')^{\theta} \cong \lr(\tau_{0}) \otimes \x_{x_{0}}$, we can obtain a $1-1$ correspondence between $\{(\tau_{0}')^{y}: y \in \S{\underline{\p}}\}$ and $\{\lr(\tau_{0}') \otimes \x_{y}: y \in \S{\underline{\p}}\}$ again through \eqref{eq: lift twisting character}. Note $\lr(\tau'_{0}) \in \lPkt{\p, \lif{\chi}_{\p}}$ (see Remark~\ref{rk: theta twisting character 1}). In this way, one can construct a pairing from $\lPkt{\p, \lif{\chi}_{\p}}$ to $\Irr{\S{\underline{\lp}}}$ as in Proposition~\ref{prop: coarse L-packet}, which further satisfies  
\[
\lr(\tau^{x})^{\theta} \cong \lr(\tau) \otimes \x_{x},
\]
for any $\tau \in \Irr{\S{\underline{\lp}}}$ and $x \in \S{\underline{\p}}^{\theta}$.

%\[
%\lr(\tau^{yx})^{\theta} = \lr(\tau^{y}) \otimes \x_{x}
%\]
%for $y \in \S{\underline{\p}}$. That follows from
%\begin{align*}
%\lr(\tau^{yx})^{\theta} = \lr(\tau^{x (x^{-1}yx) })^{\theta} = (\lr(\tau^{x}) \otimes \x_{y})^{\theta} \\
%= \lr(\tau^{x})^{\theta} \otimes \x_{y} = \lr(\tau) \otimes \x_{x} \x_{y} = \lr(\tau^{y}) \otimes \x_{x}.
%\end{align*}

\item

For $\rho \in \Irr{\S{\underline{\p}}}$ and $\tau \in \Irr{\S{\underline{\lp}}}$ being in the restriction $\rho|_{\S{\underline{\lp}}}$, it is easy to see for $x \in \S{\underline{\p}}^{\theta}$, $\tau^{x} \cong \tau$ implies $\rho^{x} \cong \rho \otimes \e$ for some $\e \in (\S{\underline{\p}}/\S{\underline{\lp}})^{*}$. %Let $\S{\underline{\p}}^{+} = \S{\underline{\p}}^{\Sigma}$ for $\Sigma = <\theta>$. 
%If $F$ is archimedean, $\S{\underline{\p}}^{+}$ is always abelian and we let $\epsilon = 1$. If $F$ is nonarchimedean, 
By the proof of Proposition~\ref{prop: twisting character}, there exists $h \in \lG(F)$ such that $\e = \e_{h}$. Since $X(\rho) = \{\e_{g} : g \in \lG(\r(\rho))\}$, then $h$ is uniquely determined modulo $\lG(\r(\rho))$.
%we get a homomorphism 
%\[
%\xymatrix{\S{\underline{\p}}^{+}(\tau) \ar[r]   &  \lG(F)/\lG(\r(\rho)) \\
%x \ar@{|->}[r]   & h}
%\]
%where $\S{\underline{\p}}^{+}(\tau) = \{x \in \S{\underline{\p}}^{+}: \tau^{x} \cong \tau\}$. 
It follows 
\[
\r(\rho)^{\theta^{-1}} \cong \r(\rho^{x}) \cong \r(\rho \otimes \e)\cong \r(\rho)^{h},
\] 
so $\r(\rho)^{\theta_{h}} \cong \r(\rho)$, where $\theta_{h} = h \rtimes \theta$. 
%Following the same idea in Proposition~\ref{prop: twisting character}, 
In the special case $\rho^{x} \cong \rho$, we can prove the hypothesis under Hypothesis~\ref{hypo: twisting character}. 

\end{enumerate}

\end{remark}

\begin{proposition}% theta twisting character PROPOSITION
\label{prop: theta twisting character}
Suppose $\p \in \Pbd{G}$, let $\rho \in \Irr{\S{\underline{\p}}}$ and $\tau \in \Irr{\S{\underline{\lp}}}$ be in the restriction $\rho|_{\S{\underline{\lp}}}$. Let $\lr$ be an irreducible smooth representation of $\lG(F)$ whose restriction to $G(F)$ contains $\r = \r(\rho)$. We assume Hypothesis~\ref{hypo: twisting character} and $\rho^{x} \cong \rho$ for $x \in \S{\underline{\p}}^{\theta}$, then for any $x \in \S{\underline{\p}}^{\theta}$
\[
\tau^{x} \cong \tau \Longleftrightarrow \lr^{\theta} \cong \lr \otimes \x_{x}.
\]
\end{proposition}

\begin{proof}
Since $\r(\rho)^{\theta} \cong \r(\rho^{x^{-1}})$ for $x \in \S{\underline{\p}}^{\theta}$, then by our assumption $\r \cong \r^{\theta}$. This means we have $x_{0} \in \S{\underline{\p}}^{\theta}$ such that $<x_{0}, \r^{+}>_{\underline{\p}} \neq 0$, in particular, $\tau^{x_{0}} \cong \tau$. By \eqref{eq: twisting character},
\[
<x_{0}, (\r^{+})^{g}>_{\underline{\p}} = \x_{x_{0}}(g)<x_{0}, \r^{+}>_{\underline{\p}}.
\]
Take $g \in \lG(\r)$, we get 
\[
<x_{0}, (\r^{+})^{g}>_{\underline{\p}} f_{G^{+}}(\r^{+}) =   \x_{x_{0}}(g)<x_{0}, \r^{+}>_{\underline{\p}} f_{G^{+}}(\r^{+}) =  \x_{x_{0}}(g)<x_{0}, (\r^{+})^{g}>_{\underline{\p}} f_{G^{+}}((\r^{+})^{g}),
\]
for $f \in C^{\infty}_{c}(G(F) \rtimes \theta)$. Hence 
\begin{align}% theta twisting character 3
\label{eq: theta twisting character 4}
f_{G^{+}}((\r^{+})^{g}) = \x_{x_{0}}(g)^{-1} f_{G^{+}}(\r^{+}).
\end{align}

Let $\iG{F}$ be a maximal subgroup of $\lG(F)$, to which one can extend $\r$. Then we take the extension $\r^{1}$ of $\r$ such that $\lr \cong \Ind^{\lG(F)}_{\iG{F}} \r^{1}$. Since $\r^{1}(g)$ intertwines between $\r$ and $\r^{g}$, and $\r^{+}(\theta)$ intertwines between $\r$ and $\r^{\theta}$, it follows from \eqref{eq: theta twisting character 4} that
\begin{align*}% theta twisting character 3
%\label{eq: theta twisting character 3}
\r^{1}(\theta g \theta^{-1})  = \x_{x_{0}}(g) \cdot \r^{+}(\theta) \r^{1}(g) \r^{+}(\theta)^{-1}.
\end{align*}
This means $(\r^{1})^{\theta} \cong \r^{1} \otimes (\x_{x_{0}}|_{\iG{F}})$, and hence $\lr^{\theta} \cong \lr \otimes \x_{x_{0}}$. 

Now suppose $\tau^{x} \cong \tau$ for some $x \in \S{\underline{\p}}^{\theta}$, then $xx_{0}^{-1} \in \S{\underline{\p}}(\tau)$. From Hypothesis~\ref{hypo: twisting character}, we have $\lr \cong \lr \otimes \x_{xx_{0}^{-1}}$, then $\lr^{\theta} \cong \lr \otimes \x_{x_{0}} \cong \lr \otimes \x_{x}$. Conversely, if $\lr^{\theta} \cong \lr \otimes \x_{x}$ for some $x \in \S{\underline{\p}}^{\theta}$, then $\lr \otimes \x_{x} \cong \lr \otimes \x_{x_{0}}$. It follows again from Hypothesis~\ref{hypo: twisting character} that $xx_{0}^{-1} \in \S{\underline{\p}}(\tau)$, and hence $\tau^{x} \cong \tau$.

%then $\r^{+}(\theta)$ permutes the summands
%\[
%\lr|_{G} \cong \bigoplus_{h \in \lG(\r) \backslash \lG(F)} (\r^{1})^{h}.
%\]
%If $\r^{+}(\theta): V(\r^{1}) \rightarrow V((\r^{1})^{h})$, then $(\r^{1})^{\theta} \cong (\r^{1})^{h} \otimes \x_{x}$ which implies $\r^{\theta} \cong \r^{h}$. Since $\r \cong \r^{h}$, we must have $h \in \lG(\r)$ and hence
%\(
%(\r^{1})^{\theta} \cong \r^{1} \otimes \x_{x}.
%\)
%This implies $\lr^{\theta} \cong \lr \otimes \x_{x}$. 

\end{proof}

\begin{remark}% theta twisting character 1 REMARK
\label{rk: theta twisting character 1}
Let $\tau_{0} \in \Irr{\S{\lp}}$ and $\rho_{0} \in \Irr{\S{\p}}$ be both trivial. Then it is clear that $\rho_{0}^{x} \cong \rho_{0}$ for $x \in \S{\underline{\p}}^{\theta}$ and $m(\rho_{0}, \tau_{0}) = 1$.
Let $\lr$ be an irreducible smooth representation of $\lG(F)$ whose restriction to $G(F)$ contains $\r = \r(\rho_{0})$. Note  $\r = \r(\rho_{0})$ is generic, so by Lemma~\ref{lemma: multiplicity one for generic representation} we have $m(\lr, \r) = 1$. It follows from Proposition~\ref{prop: twisting character under multiplicity one} that Hypothesis~\ref{hypo: twisting character} is satisfied for such $\r$ and $\lr$. Therefore the assumptions of this proposition are all satisfied in this case, and we have $\lr^{\theta} \cong \lr \otimes \x_{x}$ for any $x \in \S{\p}^{\theta}$. Suppose $\lr \in \lPkt{\p, \lif{\chi}_{\p}}$, i.e., $\lr$ has central character $\lif{\chi}_{\p}$ (see Proposition~\ref{prop: coarse L-packet}), then this implies $\lif{\chi}_{\p}^{\theta} = \lif{\chi}_{\p} \cdot \x_{x}|_{\lZ(F)}$ any $x \in \S{\p}^{\theta}$.
\end{remark}

\subsection{Conjectural refinement}% refinement SUBSECTION
\label{subsec: refinement}

The refinement of L-packets of $\lG$ should be a section of certain choice of the pairing $\lPkt{\p, \lif{\chi}_{\p}} \rightarrow \D{\S{\underline{\lp}}}$ given in Proposition~\ref{prop: coarse L-packet}, for which we make the following conjecture.

\begin{conjecture}% refined L-packet
\label{conj: refined L-packet}

Suppose $\p \in \Pbd{G}$, and $\lif{\chi}_{\p}$ is a character of $\lZ(F)$ whose restriction to $\Z(F)$ is $\chi_{\p}$. Let $\lif{\chi} = \lif{\chi}_{\p}|_{\lif{Z}_{F}}$. Then one can construct a pairing of $\lPkt{\p, \lif{\chi}_{\p}} \rightarrow \D{\S{\lp}}$ as in Proposition~\ref{prop: coarse L-packet} and a section $\Pkt{\lp}$, which satisfies the following properties:

\begin{enumerate}
\item 
\[
\lPkt{\p, \lif{\chi}_{\p}} = \bigsqcup_{\x \in X / \a(\S{\underline{\p}})} \Pkt{\lp} \otimes \x. 
\]  

\item For $\lf \in C^{\infty}_{c}(\lG(F), \lif{\chi})$, the distribution
\[
\lf(\underline{\lp}) := \sum_{\lr \in \Pkt{\lp}} <1, \lr>_{\underline{\p}} \lf_{\lG}(\lr)
\]
is stable.

\item Suppose $s$ is a semisimple element in $\cS{\underline{\lp}}$ and  $(H_{1}, \underline{\p}_{H_{1}}) \rightarrow (\underline{\p}, s)$. Suppose $\Pkt{\lp_{H_{1}}}$ exists and it satisfies (i) and (ii). Then we can choose some twist of $\Pkt{\lp_{H_{1}}}$ by $\Hom(\lif{H}_{1}(F)/H_{1}(F), \C^{\times})$, which is still denoted the same, such that
\begin{align}% lifted character relation Eq
\label{eq: lifted character relation}
\lf^{\lif{H}_{1}}(\underline{\lp}_{H_{1}}) = \sum_{\lr \in \Pkt{\lp}} <x, \lr>_{\underline{\p}} \lf_{\lG}(\lr), \,\,\,\,\,\, \lf \in C^{\infty}_{c}(\lG(F), \lif{\chi})  
\end{align}
where $x$ is the image of $s$ in $\S{\underline{\lp}}$.

%\item Suppose $s$ is a semisimple element in $\cS{\underline{\p}}$ and  $(H_{1}, \underline{\p}_{H_{1}}) \rightarrow (\underline{\p}, s)$. Let $\x = \a(s)$. Suppose $\Pkt{\lp_{H_{1}}}$ exits and it satisfies (1) and (2). Then we can choose some twist of $\Pkt{\lp_{H_{1}}}$, still denoted the same, such that
%\begin{align}% theta twisted character relation Eq
%\label{eq: theta twisted character relation}
%\lf^{H_{1}}(\lp_{H_{1}}) = \sum_{\substack{\lr \in \Pkt{\lp} \\ \lr \cong \lr \otimes \x}} \lf_{\lG}(\lr, \x), \,\,\,\,\,\, \lf \in C^{\infty}_{c}(\lG, \lif{\chi})  
%\end{align}
%where $\lif{\chi} = \lif{\chi}_{\p}|_{\lif{Z}_{F}}$, and $\lf_{\lG}(\lr, \x) = tr (\lr(\lf) \circ A(\x))$ and $A(\x)$ is an intertwining operator between $\lr \otimes \x$ and $\lr$, which is normalized in a way so that if $f$ is the restriction of $\lf$ on $G(F)$ viewed as an element in $C^{\infty}_{c}(G, \chi)$, then 
%\begin{align}
%\label{eq: theta twisted intertwining operator}% theta twisted intertwining operator Eq
%(\lf|_{\lif{Z}_{F}G(F)})_{\lG}(\lr(\tau), \x) = \sum_{\substack{\rho \in \Irr{\S{\p}} \rho|_{\S{\lp}} \supseteq \tau \\ \tau_{1} \in \D{\S{\underline{\p}}}(\tau), \tau_{1}|_{\S{\lp}} = \tau, \tau_{1} \subseteq \rho|_{\S{\p}(\tau)}}} \tau_{1}(x) f_{G}(\r(\rho))  
%\end{align}
%where $x$ is the image of $s$ in $\S{\p}$.

\item Suppose $s$ is a semisimple element in $\com{\cS{\underline{\p}}}$ and  $(H_{1}, \underline{\p}_{H_{1}}) \rightarrow (\underline{\p}, s)$. Let $x$ be the image of $s$ in $\S{\underline{\p}}^{\theta}$, and $\x = \a(x)$. Suppose $\Pkt{\lp_{H_{1}}}$ exists and it satisfies (i) and (ii). Then for any $\tau \in \Irr{\S{\underline{\lp}}}$ such that $\tau^{x} \cong \tau$, and any extension $\tau_{1}$ of $\tau$ to the group generated by $\S{\lp}$ and $x$, one can associate it with an intertwining operator $A_{\lr(\tau)}(\theta, \x): \lr(\tau) \otimes \x \rightarrow \lr(\tau)^{\theta}$ such that for some twist of $\Pkt{\lp_{H_{1}}}$ by $\Hom(\lif{H}_{1}(F)/H_{1}(F), \C^{\times})$, which is still denoted the same, we have
\begin{align}% theta twisted character relation Eq
\label{eq: theta twisted character relation}
\lf^{\lif{H}_{1}}(\underline{\lp}_{H_{1}}) = \sum_{\substack{\tau \in \Irr{\S{\underline{\lp}}} \\ \tau^{x} \cong \tau}} trace (\tau_{1}(x)) \cdot \lf_{\com{\lG}}(\lr(\tau), \x), \,\,\,\,\,\, \lf \in C^{\infty}_{c}(\lG(F), \lif{\chi})  
\end{align}

\end{enumerate}

\end{conjecture}

It is clear that \eqref{eq: theta twisted character relation} generalizes \eqref{eq: lifted character relation}. In the setup of this conjecture, for $x \in \S{\underline{\p}}^{\theta}$ and $\tau \in \Irr{\S{\underline{\lp}}}$ such that $\tau^{x} \cong \tau$, let $\r$ be an irreducible constituent in $\lr(\tau)|_{G}$, then $\r^{\theta_{h}} \cong \r$, where $x$ determines $h \in \lG(F)/\lG(\r)$ as in Remark~\ref{rk: theta twisting character} (ii). We fix a representative of $h$ in $\lG(F)$, then $\lr(h) \circ A_{\lr(\tau)}(\theta, \x)$ induces an intertwining operator $A_{I(\r)}(\theta_{h}): I(\r) \rightarrow I(\r)^{\theta_{h}}$ by restricting to the $\r$-isotypic component $I(\r)$ in $\lr(\tau)|_{G}$. So
\begin{align}
\label{eq: theta twisted intertwining operator}% theta twisted intertwining operator Eq
(\lf|_{\lif{Z}_{F}G(F) \cdot h})_{\com{\lG}}(\lr(\tau), \x) = \sum_{\r \in \lr(\tau)|_{G}} f_{G^{\theta_{h}}}(I(\r)),  
\end{align}
where $f \in C^{\infty}_{c}(G(F), \chi)$ is obtained by letting $f(g) = \lf(gh)$ and $f_{G^{\theta_{h}}}(I(\r))$ is the twisted character of $I(\r)$ generalizing \eqref{eq: twisted character}. We would like to restrict \eqref{eq: theta twisted character relation} to $\lf \in C^{\infty}_{c}(\lG(F), \lif{\chi})$ supported on $\lif{Z}_{F}G(F) \cdot h$. To write down the formula we make the following conjecture.

\begin{conjecture}
In the setup of the Conjecture~\ref{conj: refined L-packet}, let $\tau' = \tau^{y}, \tau_{1}' = \tau^{y}_{1}$ for $y \in \S{\underline{\p}}$, and suppose $\tau'_{1}$ is associated with $A_{\lr(\tau')}(\theta, \x): \lr(\tau') \otimes \x \rightarrow \lr(\tau')^{\theta}$. If we identify the representation space of $\lr(\tau')$ and $\lr(\tau)$ such that $\lr(\tau') = \lr(\tau) \otimes \x_{y}$, then $A_{\lr(\tau')}(\theta, \x) = A_{\lr(\tau)}(\theta, \x)$.
\end{conjecture}

As a result, we have 
\begin{align}% restriction of theta twisted character relation Eq
\label{eq: restriction of theta twisted character relation}
\lf^{\lif{H}_{1}}(\underline{\lp}_{H_{1}}) = \sum_{\substack{\r \in \Pkt{\p} \\ \r \cong \r^{\theta_{h}}}} (\sum_{y \in \S{\underline{\p}}/\S{\underline{\p}}(\tau)} trace (\tau^{y}_{1}(x)) \cdot \x_{y}(h)) f_{G^{\theta_{h}}}(I(\r)),
\end{align}
where $A(\theta_{h})$ is normalized according to $\tau_{1}$ and $\lf$ is supported on $\lif{Z}_{F}G(F) \cdot h$. We should point out when $\theta_{h} = id$, $A_{I(\r)}(id)$ is not necessarily trivial, although the notation for the twisted character then becomes the same as that for the ordinary one. Moreover, it is implied by this formula that if $f_{G^{\theta_{h}}}(I(\r))$ is not zero, then the sum 
\[
\sum_{y \in \S{\underline{\p}}/\S{\underline{\p}}(\tau)} trace (\tau^{y}_{1}(x)) \cdot \x_{y}(h)
\]
must be well-defined, i.e., for any $y' \in \S{\underline{\p}}(\tau)$,
\[
trace (\tau^{y}_{1}(x)) \cdot \x_{y}(h) =  trace (\tau^{yy'}_{1}(x)) \cdot \x_{yy'}(h).
\]
At last we want to point out \eqref{eq: restriction of theta twisted character relation} generalizes the formula \eqref{eq: twisted character identity} to the case where the automorphism of the group needs not preserve an $F$-splitting.

\subsection{Classical groups}% classical group SUBSECTION
\label{subsec: classical group}

The endoscopic hypothesis (\eqref{eq: disjoint decomposition}, Conjecture~\ref{conj: twisted endoscopic parametrization} and Conjecture~\ref{conj: twisted character identity}) has been proven under slight modification for quasisplit classical groups (cf. \cite{Arthur:2013}, \cite{Mok:2014}). In this section we will look into the case of symplectic groups and special even orthogonal groups. So from now on, $G$ will always be a split symplectic group, or a quasisplit special even orthogonal group, where the outer twist comes from the conjugation by the full orthogonal group. Let $\lG$ be the corresponding similitude group. There is an exact sequence
\begin{align}% central character for classical group Eq
\label{eq: central character for classical group}
\xymatrix{1 \ar[r] & G \ar[r] & \lG \ar[r]^{\c}  & \mathbb{G}_{m} \ar[r] & 1,}
\end{align}
where $\c$ is called the similitude character. We fix an automorphism $\theta_{0}$ of $G$ preserving an $F$-splitting. When $G$ is symplectic, we require $\theta_{0}$ to be trivial. When $G$ is special even orthogonal, we require $\theta_{0}$ to be the unique nontrivial outer automorphism induced from the conjugation of the full orthogonal group. Clearly, $\theta_{0}^{2} = 1$, $\theta_{0}$ extends to $\lG$ by acting trivially on $\lZ$, and $\c$ is $\theta_{0}$-invariant. Let $\Sigma_{0} = <\theta_{0}>$. Note $\Sigma_{0}$ acts on $\Pkt{}(G(F))$ and its dual $\D{\Sigma}_{0}$ acts on $\P{G}$. So we denote the set of $\Sigma_{0}$-orbits in $\Pkt{}(G(F))$ by $\cPkt{}(G(F))$ and the set of $\Sigma_{0}$-orbits in $\P{G}$ by $\cP{G}$. Similarly, we can define $\cPkt{temp}(G(F))$, $\cPbd{G}$, and analogues of these sets for $\lG$. Now we will recall the conjectures in the introduction by stating them as theorems in the case of symplectic groups and special even orthogonal groups.

\begin{theorem}[\cite{Arthur:2013}, Theorem 1.5.1]% LLC THEOREM
\label{thm: LLC}
There is a canonical way to associate any $[\p] \in \cP{G}$ with a finite subset $\cPkt{\p}$ of $\cPkt{}(G(F))$ such that 
\begin{align*}% disjoint decomposition Eq
\label{eq: disjoint decomposition}
\cPkt{}(G(F)) = \bigsqcup_{[\p] \in \cP{G}} \cPkt{\p},
\end{align*}
and
\begin{align*}
\cPkt{temp}(G(F)) = \bigsqcup_{[\p] \in \cPbd{G}} \cPkt{\p}.
\end{align*}
\end{theorem}

\begin{theorem}[\cite{Arthur:2013}, Theorem 1.5.1 and Proposition 8.3.2]% endoscopic parametrization THEOREM
\label{thm: endoscopic parametrization}
We fix a $\Sigma_{0}$-stable Whittaker datum $(B, \Lambda)$ for $G$, and suppose $[\p] \in \cPbd{G}$.
\begin{enumerate}
\item
There is a $\Sigma_{0}$-orbit of $(B, \Lambda)$-generic representations in $\cPkt{\p}$.
\item
There is a canonical pairing between $\cPkt{\p}$ and $\S{\underline{\p}}$, which induces an inclusion from $\cPkt{\p}$ to the characters $\D{\S{\underline{\p}}}$
\[
\xymatrix{\cPkt{\p} \ar[r] & \D{\S{\underline{\p}}} \\
                 [\r] \ar@{{|}->}[r]  & <\cdot, \r>_{\underline{\p}},}
\]
such that it sends the $(B, \Lambda)$-generic representation to the trivial character. This becomes a bijection when $F$ is nonarchimedean. 
\end{enumerate}
\end{theorem}

\begin{remark}
When $F$ is archimedean, it follows from \cite{Kostant:1978} that the $\Sigma_{0}$-orbit of $(B, \Lambda)$-generic representations in $\cPkt{\p}$ is unique. When $F$ is nonarchimedean, one can deduce the uniqueness of generic representation using the results from \cite{Jiang-Soudry:2003},  \cite{Liu:2011} and \cite{Jantzen-Liu:2014} (see \cite{Arthur:2013}, Remark in 8.3).
\end{remark}

For $[\p] \in \cPbd{G}$, we can define $\Pkt{\p}^{\Sigma_{0}}$ to be the set of all isomorphism classes of irreducible smooth representations of $G^{\Sigma_{0}}(F)$, whose restriction to $G(F)$ belong to $\cPkt{\p}$. Note $\S{\underline{\p}}^{\Sigma_{0}}$ is always abelian in the current case (see \cite{Arthur:2013}, Section 1.4).

\begin{theorem}[Arthur]% twisted endoscopic parametrization THEOREM
\label{thm: twisted endoscopic parametrization}

We fix a $\Sigma_{0}$-stable Whittaker datum $(B, \Lambda)$ for $G$, and suppose $[\p] \in \cPbd{G}$.

\begin{enumerate}

\item
There is a canonical pairing between $\Pkt{\p}^{\Sigma_{0}}$ and $\S{\underline{\p}}^{\Sigma_{0}}$, which induces an inclusion from $\Pkt{\p}^{\Sigma_{0}}$ to the characters $\D{\S{\underline{\p}}}^{\Sigma_{0}}$
\[
\xymatrix{\Pkt{\p}^{\Sigma_{0}} \ar[r] & \D{\S{\underline{\p}}}^{\Sigma_{0}} \\
                 \r^{\Sigma_{0}} \ar@{{|}->}[r]  & <\cdot, \r^{\Sigma_{0}}>_{\underline{\p}}}
\]
This becomes a bijection when $F$ is nonarchimedean. Moreover, this pairing is an extension of that in Theorem~\ref{thm: endoscopic parametrization} in the sense that
\[
<\cdot, \r^{\Sigma_{0}}>_{\underline{\p}} |_{\S{\underline{\p}}} = <\cdot, \r>_{\underline{\p}},
\]
where $\r \in \r^{\Sigma_{0}}|_{G}$.

\item In case $G$ is special even orthogonal, the following statements are equivalent:

        \begin{enumerate}

        \item $\cPkt{\p}$ contains an element $[\r]$ such that $\r^{\theta_{0}} \cong \r$;
       
        \item For any $[\r] \in \cPkt{\p}$, $\r^{\theta_{0}} \cong \r$;
        
        \item $\S{\underline{\p}}^{\theta_{0}} \neq \emptyset$.

        \end{enumerate}

\end{enumerate}
\end{theorem}

\begin{remark}
Although this theorem is not stated in \cite{Arthur:2013}, one can view it as a consequence of Theorem~\ref{thm: twisted character identity}. Moreover, we expect the $(B, \Lambda)$-generic representation in $\Pkt{\p}^{\Sigma_{0}}$ correspond to the trivial character of $\S{\underline{\p}}^{\Sigma_{0}}$. 
\end{remark}

If $H$ is a $\theta$-twisted endoscopic group of $G$ for $\theta \in \Sigma_{0}$, Arthur shows $H \cong M_{l} \times G_{1} \times G_{2}$, where $M_{l}$ is a product of general linear groups, $G_{i}$ ($i = 1, 2$) is also a symplectic group or special even orthogonal group. We define a group of automorphisms of $H$ by taking the product of $\Sigma_{0}$ on each $G_{i}$, and we denote this group again by $\Sigma_{0}$. Then by combining the local Langlands correspondence for $GL(n)$ (cf. \cite{HarrisTaylor:2001}, \cite{Henniart:2000} and \cite{Scholze:2013}), all the previous theorems of Arthur can be extended to $H$. In particular, the L-packets for $H$ are formed by tensor products of those of each factor. Let $\sH(G)$ (resp. $\sH(H)$) be the space of $\Sigma_{0}$-invariant smooth compactly supported functions on $G(F)$ (resp. $H(F)$). Then the twisted endoscopic transfer sends $\sH(G)$ to $\sH(H)$, and there is no need to consider $z$-pairs here.

\begin{theorem}[\cite{Arthur:2013}, Theorem 2.2.1 and Theorem 2.2.4]% twisted character identity for classical group THEOREM
\label{thm: twisted character identity}
Suppose $[\p] \in \cPbd{G}$.

\begin{enumerate}

\item
\begin{align}% stable character for classical group Eq
\label{eq: stable character for classical group}
f(\underline{\p}) := \sum_{[\r] \in \cPkt{\p}} f_{G}(\r), \,\,\,\,\,\,\,\,\,\,\,\, f \in \sH(G)
\end{align}
is stable.

\item Suppose $\theta \in \Sigma_{0}$, $s$ is a semisimple element in $\cS{\underline{\p}}^{\theta}$, and $(H, \underline{\p}_{H}) \rightarrow (\underline{\p}, s)$. Then
\begin{align}% twisted character identity for classical group Eq
\label{eq: twisted character identity for classical group}
f^{H}(\underline{\p}_{H}) = \sum_{[\r] \in \cPkt{\p}} <x, \r^{+}>_{\underline{\p}} f_{G^{\theta}}(\r)
\end{align}
for $f \in \sH(G)$, where $x$ is the image of $s$ in $\com{\S{\underline{\p}}}$, and $\r^{+}$ is an extension of $\r$ to $G^{+}(F) := G(F) \times <\theta>$ with $\r^{+}(\theta) = A_{\r}(\theta)$. If $G$ is special even orthogonal and $\S{\underline{\p}}^{\theta_{0}} \neq \emptyset$, one can replace $\sH(G)$ by $C^{\infty}_{c}(G(F))$.

\end{enumerate}

\end{theorem}

It follows from the second part of Theorem~\ref{thm: twisted endoscopic parametrization} that, $\cPkt{\p} = \Pkt{\p}$ unless $G$ is special even orthogonal and $\S{\underline{\p}}^{\theta_{0}} = \emptyset$. In the exceptional case, we have the following refined statement.

\begin{theorem}[\cite{Arthur:2013}, Theorem 8.4.1]% refined twisted character identity for classical group THEOREM
\label{thm: refined twisted character identity for classical group}
Suppose $G$ is special even orthogonal, $[\p] \in \cPbd{G}$ and $\S{\underline{\p}}^{\theta_{0}} = \emptyset$.

\begin{enumerate}

\item
There exists a unique subset $\Pkt{\p} \subseteq \Pkt{\p}^{\Sigma_{0}}|_{G}$ up to $\theta_{0}$-twist, such that 
   
   \begin{itemize}
   
   \item
   \[
   \Pkt{\p}^{\theta_{0}} \sqcup \Pkt{\p} = \Pkt{\p}^{\Sigma_{0}}|_{G}
    \]
    
    \item
       \begin{align}% refined stable character for classical group Eq
       \label{eq: refined stable character for classical group}
       f(\underline{\p}) := \sum_{\r \in \Pkt{\p}} f_{G}(\r), \,\,\,\,\,\,\,\,\,\,\,\, f \in C^{\infty}_{c}(G(F))
       \end{align}
    is stable.
   
   \end{itemize}

\item 
Suppose $s$ is a semisimple element in $\cS{\underline{\p}}$, and $(H, \underline{\p}_{H}) \rightarrow (\underline{\p}, s)$. Then there exists $\Pkt{\p_{H}} \subseteq \Pkt{\p_{H}}^{\Sigma_{0}}|_{H}$, which can be constructed from part (i), such that
\begin{align}% refined twisted character identity for classical group Eq
\label{eq: refined twisted character identity for classical group}
f^{H}(\underline{\p}_{H}) = \sum_{\r \in \Pkt{\p}} <x, \r>_{\underline{\p}} f_{G}(\r)
\end{align}
for $f \in C^{\infty}_{c}(G(F))$, where $x$ is the image of $s$ in $\S{\underline{\p}}$.

\end{enumerate}

\end{theorem}

It follows from this theorem and Proposition~\ref{prop: central character} that the central character of $\Pkt{\p}$ is well defined. Since $\Sigma_{0}$ acts trivially on $\Z$, we can define the central character of $\cPkt{\p}$ to be that of $\Pkt{\p}$. Moreover, $\chi_{\p}$ only depends on $[\p]$. 

\begin{proposition}% central character for classical group PROPOSITION
\label{prop: central character for classical group}
For $[\p] \in \cPbd{G}$, the central character of $\cPkt{\p}$ is equal to $\chi_{\p}$.
\end{proposition}

\begin{proof}
Let $\r_{0}$ be the generic representation in $\cPkt{\p}$. Since $\Z(F) = \Two$, it suffices to show $\chi_{\r_{0}}(-1) = \chi_{\p}(-1)$. Suppose $G$ is split, then Deligne \cite{Deligne:1976} showed $\chi_{\p}(-1) = \e(1/2, \rho_{std} \circ \p, \q_{F})$ (defined by Langlands), and Lapid \cite{Lapid:2004} showed $\chi_{\r_{0}}(-1) = \e(1/2, \r_{0}, \rho_{std}, \q_{F})$ (defined by Shahidi). In both formulas $\rho_{std}$ is the standard representation of $\L{G}$. It is now known the local Langlands correspondence for $G$ preserves these epsilon factors (see \cite{Jiang-Soudry:2003},  \cite{Liu:2011}, \cite{Jantzen-Liu:2014}), in particular,
\[
\e(1/2, \rho_{std} \circ \p, \q_{F}) = \e(1/2, \r_{0}, \rho_{std}, \q_{F}).
\]
So $\chi_{\r_{0}}(-1) = \chi_{\p}(-1)$. Suppose $G$ is not split, then $G$ has to be special even orthogonal. We can view $G$ as an endoscopic group of the split symplectic group $G_{+}$ of the same $\bar{F}$-rank, and let $[\p]$ map to $[\p_{+}] \in \cPbd{G_{+}}$ through the endoscopic embedding. Let $\r_{+, 0}$ be the generic representation in $\cPkt{\p_{+}}$. From the proof of Lemma~\ref{lemma: transfer central character}, we have 
\[
\chi_{\p_{+}}(-1) / \chi_{\r_{+, 0}}(-1) = \chi_{\p}(-1) / \chi_{\r_{0}}(-1).
\]
Note we have $\chi_{\p_{+}}(-1) = \chi_{\r_{+, 0}}(-1)$ from the split case. Therefore, $\chi_{\p}(-1) =  \chi_{\r_{0}}(-1)$. This finishes the proof.
\end{proof}

As a consequence of these results, the results in Section~\ref{subsec: coarse L-packet} and Section~\ref{subsec: compatibility with twist} are unconditional. In fact we could obtain stronger results, which will be summarized below.

\begin{proposition}% theta twisting character for classical group PROPOSITION
\label{prop: theta twisting character for classical group}
Suppose $[\p] \in \cPbd{G}$, and $[\r] \in \cPkt{\p}$. If $\lr$ is an irreducible smooth representation of $\lG(F)$ whose restriction to $G(F)$ contains $\r$, then for $\theta \in \Sigma_{0}$ and $\x \in \Hom(\lG(F)/G(F), \C^{\times})$,
\[
\lr^{\theta} \cong \lr \otimes \x \Longleftrightarrow \x \in \a(\S{\underline{\p}}^{\theta}).
\]
In particular, $X(\lr) = \a(\S{\underline{\p}})$.
\end{proposition}

\begin{proof}
If $\theta = id$, this follows from Proposition~\ref{prop: twisting character under multiplicity one}, and we will have $X(\lr) = \a(\S{\underline{\p}})$. So we can assume $G$ is special even orthogonal and $\theta = \theta_{0}$. Note the direction $``\Leftarrow"$ follows from Proposition~\ref{prop: theta twisting character}. For the other direction, we suppose $\lr^{\theta_{0}} \cong \lr \otimes \x$. Then $\r^{\theta_{0}} \cong \r^{g}$ for some $g \in \lG(F)$. If $\S{\underline{\p}}^{\theta_{0}} = \emptyset$, by Theorem~\ref{thm: refined twisted character identity for classical group}, we can assume $\r \in \Pkt{\p}$. Then $\r^{\theta_{0}} \in \Pkt{\p}^{\theta_{0}}$, $\r^{g} \in \Pkt{\p}$, and we get a contradiction. So $\S{\underline{\p}}^{\theta_{0}} \neq \emptyset$, and by Theorem~\ref{thm: twisted character identity}, $\r^{\theta_{0}} \cong \r$. Let $\x_{0} \in \a(\S{\underline{\p}}^{\theta_{0}})$, we know $\lr^{\theta_{0}} \cong \lr \otimes \x_{0}$. Therefore, $\lr \cong \lr \otimes \x \x_{0}^{-1}$, which means $\x \x_{0}^{-1} \in \a(\S{\underline{\p}})$. Hence $\x \in \a(\S{\underline{\p}}^{\theta_{0}})$.
\end{proof}

For $[\p] \in \cPbd{G}$, let us fix a character $\lif{\chi}_{\p}$ of $\lZ(F)$ such that $\lif{\chi}_{\p}|_{\Z(F)} = \chi_{\p}$. We define $\clPkt{\p, \lif{\chi}_{\p}}$ to be the subset of $\cPkt{}(\lG(F))$ with central character $\lif{\chi}_{\p}$, whose restriction to $G(F)$ are contained in $\cPkt{\p}$. Let $X = \Hom(\lG(F)/\lZ(F)G(F), \C^{\times})$.

\begin{proposition}% coarse L-packet for classical group PROPOSITION
\label{prop: coarse L-packet for classical group}
Suppose $[\p] \in \cPbd{G}$ and $\lif{\chi}_{\p}$ is chosen as above.
\begin{enumerate}
\item The orbits in $\cPkt{\p}$ under the conjugate action of $\lG(F)$ all have size $|\S{\underline{\p}} / \S{\underline{\lp}}|$. If $F$ is nonarchimedean, there are exactly $|\S{\underline{\lp}}|$ orbits. \\
\item There is a natural fibration 
\[
\xymatrix{ X / \a(\S{\underline{\p}}^{\Sigma_{0}})  \ar[r]   &   \clPkt{\p, \lif{\chi}_{\p}}    \ar[r]^{Res \quad}     &  \cPkt{\p} / \lG(F).}
\]
\item There is a unique pairing 
\[
[\lr] \longrightarrow <\cdot, \lr>_{\underline{\p}}
\]
from $\clPkt{\p, \lif{\chi}_{\p}} / X$ into $\D{\S{\underline{\lp}}}$, satisfying 
\[
<x, \lr>_{\underline{\p}} = <\iota(x), \r>_{\underline{\p}},
\]
where $\iota: \S{\underline{\lp}} \hookrightarrow \S{\underline{\p}}$, $\r$ is in the restriction of $\lr$. It sends the generic representation to the trivial character. Moreover, this map from $\clPkt{\p, \chi} / X$ to $\D{\S{\underline{\lp}}}$ is injective and when $F$ is nonarchimedean it is in fact a bijection.
\end{enumerate} 
\end{proposition}

\begin{proof}
The proof essentially follows from that of Proposition~\ref{prop: coarse L-packet}, and the uniqueness of this pairing is due to the fact that $\S{\underline{\p}}$ is abelian. The last property follows from the same property of the pairing between $\cPkt{\p}$ and $\S{\underline{\p}}$.
\end{proof}

Finally, for the conjectural refinement of $\clPkt{\p, \lif{\chi}_{\p}}$, we would like to state it in the following two theorems.

\begin{theorem}% refined L-packet THEOREM
\label{thm: refined L-packet}

Suppose $[\p] \in \cPbd{G}$, and $\lif{\chi}_{\p}$ is a character of $\lZ(F)$ whose restriction to $\Z(F)$ is $\chi_{\p}$. Let $\lif{\chi} = \lif{\chi}_{\p}|_{\lif{Z}_{F}}$. Then there exists a subset $\cPkt{\lp}$ of $\clPkt{\p, \lif{\chi}_{\p}}$, unique up to twisting by $X$, and it is characterized by the following properties:

\begin{enumerate}
\item 
\[
\clPkt{\p, \lif{\chi}_{\p}} = \bigsqcup_{\x \in X / \a(\S{\underline{\p}}^{\Sigma_{0}})} \cPkt{\lp} \otimes \x. 
\]  

\item For $\lf \in \sH(\lG, \lif{\chi})$, the distribution
\[
\lf(\underline{\lp}) := \sum_{[\lr] \in \cPkt{\underline{\lp}}} \lf_{\lG}(\lr)
\]
is stable.

\end{enumerate}

\end{theorem}

%Both Proposition~\ref{prop: theta twisting character for classical group} and Proposition~\ref{prop: coarse L-packet for classical group} can be extended to twisted endoscopic group $\lif{H}$ of $\lG$. 

Recall a $\theta$-twisted endoscopic group $H$ of $G$ for $\theta \in \Sigma_{0}$ takes the form $H \cong M_{l} \times G_{1} \times G_{2}$. Let $\lG_{i}$ ($i = 1, 2$) be the similitude group of $G_{i}$ with similitude character $\c_{i}$. Suppose $\lif{H}$ is the $(\theta, \x)$-twisted endoscopic group of $\lG$ lifted from $H$ under Proposition~\ref{prop: lifting endoscopic group}, then (cf. \cite{Morel:2011})
\[
\lif{H} = \{(x, g_{1}, g_{2}) \in M_{l} \times \lG_{1} \times \lG_{2} : \c_{1}(g_{1}) = \c_{2}(g_{2})\},
\]
where $M_{l}$ is a product of general linear groups, and $\c_{H}(x, g_{1}, g_{2}) : = \c_{1}(g_{1})$. For $[\p_{H}] \in \cPbd{H}$, we can assume $\p_{H} = \p_{l} \times \p_{1} \times \p_{2}$, where $\p_{l} \in \Pbd{M_{l}}$, $[\p_{i}] \in \cPbd{G_{i}}$ ($i = 1, 2$). Fix a character $\lif{\chi}_{\p_{H}}$ of $Z_{\lif{H}}(F)$, which is the restriction of some character $\chi_{\p_{l}} \otimes \lif{\chi}_{\p_{1}} \otimes \lif{\chi}_{\p_{2}}$ of $M_{l} \times \lG_{1} \times \lG_{2}$, then by Theorem~\ref{thm: refined L-packet}, we can define $\cPkt{\lp_{H}}$ to be the restriction of $\Pkt{\p_{l}} \otimes \cPkt{\lp_{1}} \otimes \cPkt{\lp_{2}}$, which is unique up to twisting by $\Hom(\lif{H}(F)/Z_{\lif{H}}(F)H(F), \C^{\times})$.

\begin{theorem}% twisted character identity for similitude group THEOREM
\label{thm: twisted character identity for similitude group}

Suppose $[\p] \in \cPbd{G}$, and $\lif{\chi}_{\p}$ is a character of $\lZ(F)$ whose restriction to $\Z(F)$ is $\chi_{\p}$. Let $\lif{\chi} = \lif{\chi}_{\p}|_{\lif{Z}_{F}}$.
%\item Suppose $s$ is a semisimple element in $\cS{\underline{\lp}}$ and  $(H, \underline{\p}_{H}) \rightarrow (\underline{\p}, s)$. Fix $\cPkt{\lp_{H}}$, then there exists $\cPkt{\lp}$ such that
%\begin{align}% lifted character relation Eq
%\label{eq: lifted character relation}
%\lf^{\lif{H}}(\lp_{H}) = \sum_{\lr \in \Pkt{\lp}} <x, \lr>_{\underline{\p}} \lf_{\lG}(\lr), \,\,\,\,\,\, \lf \in \sH(\lGF, \lif{\chi})  
%\end{align}
%where $x$ is the image of $s$ in $\S{\underline{\lp}}$.
Suppose $\theta \in \Sigma_{0}$, $s$ is a semisimple element in $\com{\cS{\underline{\p}}}$ and  $(H, \underline{\p}_{H}) \rightarrow (\underline{\p}, s)$. Let $x$ be the image of $s$ in $\S{\underline{\p}}^{\theta}$, and $\x = \a(x)$. Fix a packet $\cPkt{\lp_{H}}$ with $\lif{\chi}_{\p_{H}}|_{\lZ} = \lif{\chi}_{\p} \chi_{\lif{C}}$ (cf. Section~\ref{subsec: character identity}), then we can choose $\cPkt{\lp}$ in Theorem~\ref{thm: refined L-packet} such that
\begin{align}% theta twisted character relation for similitude group Eq
\label{eq: theta twisted character relation for similitude group}
\lf^{\lif{H}}(\underline{\lp}_{H}) = \sum_{[\lr] \in \cPkt{\lp}} \lf_{\com{\lG}}(\lr, \x), \,\,\,\,\,\, \lf \in \sH(\lG, \lif{\chi})
\end{align}
where $A_{\lr}(\theta, \x)$ is normalized in a way so that if $f \in \sH(G, \chi)$ is the restriction of $\lf$ on $G(F)$, then 
\begin{align}% theta twisted intertwining operator for similitude group Eq
\label{eq: theta twisted intertwining operator for similitude group}
(\lf|_{\lif{Z}_{F}G(F)})_{\com{\lG}}(\lr, \x) = \sum_{\r \in \lr|_{G}} <x, \r^+>_{\underline{\p}} f_{G^{\theta}}(\r)  
\end{align}
where $\r^{+}$ is an extension of $\r$ to $G^{+}(F) := G(F) \times <\theta>$ with $\r^{+}(\theta) = A_{\r}(\theta)$.

\end{theorem}

\begin{remark}% refined L-packet for similitude group REMARK

\begin{enumerate}

\item Theorem~\ref{thm: refined L-packet} and Theorem~\ref{thm: twisted character identity for similitude group} are the main local results in \cite{Xu:preprint2}. Their proofs involve global methods, and the main tool is the stabilization of the twisted Arthur-Selberg trace formula due to M{\oe}glin and Waldspurger. 

\item If $F$ is archimedean, both theorems will
%are essentially known due to works of Langlands \cite{Langlands:1989}, Shelstad \cite{Shelstad:1982} and Mezo \cite{Mezo:2013}. Moreover, they can also 
follow from Theorem~\ref{thm: twisted character identity} directly. This is clear when $F = \C$ for $\lG(\C) = \lZ(\C)G(\C)$. When $F = \mathbb{R}$, it is known by results of Harish-Chandra (cf. \cite{H-C:1975}, Theorem 27.1) that if $\cPkt{\p}$ consists of discrete series representations of $G(\mathbb{R})$, then $X(\lr) = X$ for any $[\lr] \in \clPkt{\p, \lif{\chi}_{\p}}$. So $\cPkt{\lp} = \clPkt{\p, \lif{\chi}_{\p}}$. Moreover, for $\lif{Z}_{\mathbb{R}} = \lZ(\mathbb{R})$ and $\lf \in \sH(\lG, \lif{\chi}_{\p})$
\[
\lf(\underline{\lp}) = \frac{1}{|X|} \sum_{\x \in X} \sum_{[\lr] \in \cPkt{\lp}} \lf_{\lG}(\lr \otimes \x) = \frac{1}{|X|} \sum_{\x \in X} (\lf \otimes \x) (\underline{\lp}) = (\lf|_{\lZ(\mathbb{R})G(\mathbb{R})})(\underline{\lp}) = f(\underline{\p}),
\]
where $f \in \sH(G, \chi_{\p})$ is the restriction of $\lf$. So the stability of $\cPkt{\lp}$ follows from that of $\cPkt{\p}$. For general tempered L-packets, they can be constructed by parabolic induction from the discrete series L-packets of Levi subgroups of $\lG$. For \eqref{eq: theta twisted character relation for similitude group}, by a standard descent argument we can reduce it to the case that $H$ is {\bf elliptic} (i.e., $H = G_{1} \times G_{2}$) and $\cPkt{\p_{H}}$ consists of discrete series representations of $H(F)$. In this case, by Proposition~\ref{prop: theta twisting character for classical group} one can check $X(\lr) = X$ for any $[\lr] \in \clPkt{\p, \lif{\chi}_{\p}}$ (cf. \cite{Xu:preprint2}, Proposition 6.9). Let $\lif{Z}_{\mathbb{R}} = \lZ(\mathbb{R})$, then the right hand side of \eqref{eq: theta twisted character relation for similitude group} becomes
\[
\sum_{\lr \in \cPkt{\lp}} \lf_{\com{\lG}}(\lr, \x) = \sum_{\lr \in \cPkt{\lp}} (\lf|_{\lZ(\mathbb{R})G(\mathbb{R})})_{\com{\lG}}(\lr, \x) = \sum_{\r \in \cPkt{\p}} <x, \r^+>_{\underline{\p}} f_{G^{\theta}}(\r). 
\]
One can also check $\c_{H}(Z_{\lif{H}}(\mathbb{R})) = \c(\lZ(\mathbb{R}))$. As a result, under $\lZ \hookrightarrow Z_{\lif{H}}$, we have $Z_{\lif{H}}(\mathbb{R})H(\mathbb{R}) = \lZ(\mathbb{R})H(\mathbb{R})$. So the left hand side of \eqref{eq: theta twisted character relation for similitude group} becomes
\[
\lf^{\lif{H}}(\underline{\lp}_{H}) = (\lf^{\lif{H}}|_{Z_{\lif{H}}(F)H(\mathbb{R})})(\underline{\lp}_{H}) = (\lf^{\lif{H}}|_{\lZ(\mathbb{R})H(\mathbb{R})})(\underline{\lp}_{H}) = (\lf|_{\lZ(\mathbb{R})G(\mathbb{R})})^{\lif{H}}(\underline{\lp}_{H}).
\]
By Lemma~\ref{lemma: twisted endoscopic transfer}, $(\lf|_{\lZ(\mathbb{R})G(\mathbb{R})})^{\lif{H}}(\underline{\lp}_{H}) =  \lif{f^{H}}(\underline{\lp}_{H}) = f^{H}(\underline{\p}_{H})$. Therefore, \eqref{eq: theta twisted character relation for similitude group} follows from \eqref{eq: twisted character identity for classical group} in this case.

\end{enumerate}

\end{remark}

\appendix

\section{}% character 
\label{sec: character}

Let $F$ be a local field of characteristic zero and let $G$ be a quasisplit connected reductive group over $F$. In this appendix, we would like to recall Langlands' construction of
\begin{align}% local character 1 Eq
\label{eq: local character 1}
H^{1}(W_{F}, Z(\D{G})) \longrightarrow \Hom(G(F), \C^{\times}),
\end{align}
and we will also show it is an isomorphism. To define this homomorphism, we first need to take a $z$-extension of $G$
\[
\xymatrix{1 \ar[r] & Z \ar[r] & \lG' \ar[r]  & G \ar[r] & 1,}
\]
where $G' := \lG'_{der}$ is simply connected and $H^{1}(F, Z) = 1$. Let $\lG'/G' = D$, and we have an exact sequence
\[
\xymatrix{1 \ar[r] & G' \ar[r] & \lG' \ar[r]^{\c'}  & D \ar[r] & 1.}
\]
Since $\D{G}'$ is adjoint, $\D{D} \cong Z(\D{\lG'})$ and hence $H^{1}(W_{F}, Z(\D{\lG'})) \cong H^{1}(W_{F}, \D{D}) \cong \Hom(D(F), \C^{\times})$ by the local Langlands correspondence for tori. By pulling back quasicharacters of $D(F)$ to $\lG'(F)$, we then get a homomorphism 
\begin{align}% local character 2
\label{eq: local character 2}
H^{1}(W_{F}, Z(\D{\lG'})) \rightarrow \Hom(\lG'(F), \C^{\times}).
\end{align}
Next we consider the following $\Gal{F}$-equivariant exact sequence
\[
\xymatrix{1 \ar[r] & Z(\D{G}) \ar[r] & Z(\D{\lG'}) \ar[r]  & \D{Z} \ar[r] & 1.}
\]
It induces a long exact sequence 
\[
\xymatrix{\pi_{0}(\D{Z}^{\Gal{F}}) \ar[r] & H^{1}(W_{F}, Z(\D{G})) \ar[r]  & H^{1}(W_{F}, Z(\D{\lG'})) \ar[r] & H^{1}(W_{F}, \D{Z})}
\]
By Tate-Nakayama duality, we have $\pi_{0}(\D{Z}^{\Gal{F}}) \cong H^{1}(F, Z)^{*} = 1$. So we get an inclusion $H^{1}(W_{F}, Z(\D{G})) \hookrightarrow H^{1}(W_{F}, Z(\D{\lG'}))$. On the other hand, $\lG'(F)/Z(F) \cong G(F)$, so we also have an inclusion $\Hom(G(F), \C^{\times}) \hookrightarrow \Hom(\lG'(F), \C^{\times})$. Then \eqref{eq: local character 1} is defined to satisfy the following commutative diagram
\[
\xymatrix{1 \ar[r] \ar@{=}[d] & H^{1}(W_{F}, Z(\D{G})) \ar[r]  \ar[d]^{\eqref{eq: local character 1}} & H^{1}(W_{F}, Z(\D{\lG'})) \ar[d]^{\eqref{eq: local character 2}} \ar[r] & H^{1}(W_{F}, \D{Z}) \ar[d]^{\simeq} \\
1 \ar[r] & \Hom(G(F), \C^{\times}) \ar[r] & \Hom(\lG'(F), \C^{\times}) \ar[r] & \Hom(Z(F), \C^{\times}).} 
\]
To show \eqref{eq: local character 1} is an isomorphism, from this diagram it is enough to know \eqref{eq: local character 2} is an isomorphism. Since $G'$ is semisimple simply connected, $\Hom(G'(F), \C^{\times}) = 1$, which implies \eqref{eq: local character 2} is surjective. For the injectivity, we need to show $\c'(\lG'(F)) = D(F)$. We choose a maximal torus $\lif{T}'$ of $\lG'$, and let $T' = \lif{T}' \cap G'$. The short exact sequence
\begin{align*}
\xymatrix{
                 1 \ar[r] &  T' \ar[r]    & \lif{T}'  \ar[r]^{\c'}  & D  \ar[r]  & 1
                 }
\end{align*}
induces the following exact sequence
\[
\xymatrix { \lif{T}'(F) \ar[r]^{\c'}  & D(F)  \ar[r]^{\delta_{T'} \quad}  & H^{1}(F, T').}
\]
By Tate-Nakayama duality, $H^{1}(F, T') \cong \pi_{0}(\D{T}'^{\Gal{}})^{*}$. Now let $T'$ be the Levi component of a Borel subgroup $B'$ of $G'$, and we fix a $\Gal{}$-splitting $\{\D{B}', \D{T}', \{\mathcal{X}'_{\alpha}\}\}$ for $\D{G}'$. Then there is a $\Gal{}$-equivariant isomorphism
\[
\xymatrix{\D{T}' \ar[r] & \prod_{\alpha} \C^{\times}_{\alpha} \\
t \ar@{|->}[r] & (\alpha^{\vee}(t)),}
\]
where $\C^{\times}_{\alpha} = \C^{\times}$, $\alpha^{\vee}$ are simple coroots of $(G', T')$, and the $\Gal{}$-action on $\prod_{\alpha} \C^{\times}_{\alpha}$ is given by permutations on the indexing set of simple roots. Clearly, $\D{T}'^{\Gal{}} \cong (\prod_{\alpha} \C^{\times}_{\alpha})^{\Gal{}}$ is connected, i.e., $\pi_{0}(\D{T}'^{\Gal{}})^{*} = 1$. This implies $\c'(\lif{T}'(F)) = D(F)$, and hence $\c'(\lG'(F)) = D(F)$.

\bibliographystyle{amsalpha}

\bibliography{reps}

\end{document}